\documentclass[11pt,a4paper]{amsart}
\usepackage[T1]{fontenc}
\usepackage{lmodern}
\usepackage{amsmath,amssymb,mathtools}
\usepackage{microtype,booktabs,array,enumitem,needspace}
\usepackage[a4paper,textwidth=155mm,textheight=235mm,centering,headheight=14pt]{geometry}
\usepackage[hidelinks]{hyperref}
\hypersetup{pdftitle={A classification of zero p-rank curves with large automorphism groups in odd characteristic},
 pdfauthor={Saeed Tafazolian},
 pdfsubject={Automorphism groups, ramification, and algebraic curves},
 pdfkeywords={Algebraic curves; automorphism groups; p-rank; ramification}}
\hypersetup{bookmarksopen=true,bookmarksopenlevel=1,bookmarksnumbered=true}
\numberwithin{equation}{section}
\newtheorem{theorem}{Theorem}[section]
\newtheorem{proposition}[theorem]{Proposition}
\newtheorem{lemma}[theorem]{Lemma}
\newtheorem{corollary}[theorem]{Corollary}
\newtheorem{fact}[theorem]{Fact}
\theoremstyle{remark}
\newtheorem{remark}[theorem]{Remark}
\newcommand{\Aut}{\operatorname{Aut}}
\newcommand{\Syl}{\operatorname{Syl}}
\newcommand{\PSL}{\operatorname{PSL}}
\newcommand{\PGL}{\operatorname{PGL}}
\newcommand{\PSU}{\operatorname{PSU}}
\newcommand{\PGU}{\operatorname{PGU}}
\newcommand{\Ree}{\operatorname{Ree}}

\newcommand{\F}{\mathbb F}
\newcommand{\PP}{\mathbb P}
\newcommand{\cX}{\mathcal X}
\newcommand{\cY}{\mathcal Y}

\title[Classification of zero $p$-rank curves]
{A classification of zero $p$-rank curves with large\\
automorphism groups in odd characteristic}
\author{Saeed Tafazolian}
\address{IMECC, Universidade Estadual de Campinas (UNICAMP),
Campinas, SP, Brazil}
\email{saeed@unicamp.br}
\date{}
\begin{document}
\begin{abstract}
Let $\cX$ be a curve of genus $g\ge2$ and $p$-rank zero over an
algebraically closed field of odd characteristic $p$, and let
$G\le\Aut(\cX)$ satisfy $|G|>24g(g-1)$. We prove that $G$ fixes
a point unless $(\cX,G)$ belongs to an explicit list. For $g\ge4$,
the curve is a cyclic quotient of a Hermitian curve, of a
Giulietti--Korchm\'aros curve, or of Skabelund's cover of a Ree curve.
We determine the equations and all the groups $G$ that occur,
including their central kernels. The exceptional $A_7$ action in
characteristic $5$ is on the Hermitian curve of degree $6$;
genera $2$ and $3$ are treated separately. The equations are
recovered from the group action by two rigidity lemmas for
equivariant one-point $p$-group covers of the affine line.
If $G$ has no common fixed point and its Sylow $p$-subgroups are
noncyclic, then, unless $G$ induces ${}^2G_2(3)$ on the orbit of
Sylow fixed points, the bound $|G|>\max\{84(g-1),2g^2\}$ gives a
rational quotient with exactly one wild and one tame short orbit,
and a rational Sylow quotient.
\end{abstract}
\maketitle

\section{Introduction}\label{sec:introduction}

Let $K$ be an algebraically closed field of characteristic $p$, and let
$\cX/K$ be a smooth projective irreducible curve of genus $g\ge2$.
The group $\Aut(\cX)$ of $K$-automorphisms is finite. In characteristic
zero its order is bounded by $84(g-1)$, whereas in positive
characteristic this bound can fail because of wild ramification.
A basic question is to determine the curves on which a group much
larger than the Hurwitz bound can act; see \cite[Chapter~11]{HKT}.

The $p$-rank is particularly useful in this problem. It is the integer
$\gamma$ for which the Jacobian of $\cX$ has $p^\gamma$ geometric
$p$-torsion points. When $\gamma=0$, every nontrivial $p$-subgroup of
$\Aut(\cX)$ fixes exactly one point and acts freely elsewhere.
Consequently, for every $G\le\Aut(\cX)$ with $p\mid|G|$, the fixed
points of the Sylow $p$-subgroups of $G$ form a single $G$-orbit; see
Lemma~\ref{lem:sylow}. This connects the ramification of the action with the
permutation group induced on that orbit.

We determine the pairs $(\cX,G)$ with $p$-rank zero in odd
characteristic for which $|G|>24g(g-1)$ and $G$ has no common
fixed point. Equivalently, a group above this bound fixes a point
unless it is one of the actions listed below or in
Theorem~\ref{thm:small} for genera $2$ and $3$.
No assumption of maximality over a finite field is made.
We write $H_q$ for the Hermitian curve $z^q+z=y^{q+1}$ and
$\Ree(q)$ for ${}^2G_2(q)$.

\begin{theorem}[Geometric classification]\label{thm:geometric}
Let $K$ be an algebraically closed field of characteristic $p>2$.
Let $\cX/K$ be a curve of genus $g\ge4$ and $p$-rank zero, and let
$G\le\Aut(\cX)$ satisfy $|G|>24g(g-1)$. Suppose that $G$ has no common
fixed point on $\cX$. Then $\cX$ is $K$-isomorphic to precisely one of
the following curves. Here $q$ is a power of $p$, and every displayed
model denotes its smooth projective normalization.
\begin{enumerate}[label=\textup{(\roman*)},leftmargin=*]
\item The linear family
\[
 L(q,d):\quad y^q-y=x^d,\qquad 2\le d<q,\quad d\mid q+1,
 \qquad g=\frac{(q-1)(d-1)}2.
\]
Its full automorphism group has order $dq(q^2-1)$ and is the central lift
$A_L(q,d)$ in \eqref{eq:lineargroup}.
\Needspace{7\baselineskip}
\item The unitary family
\[
 U(q,d):\quad y^{q^2}-y=x^d,\qquad z^q+z=y^{q+1},
 \qquad d\mid q^2-q+1,
\]
where $q\ge5$, $d\ge1$, and
\[
 2g=(q-1)\bigl(d(q+1)^2-(q^2+q+1)\bigr).
\]
Its full group is the central lift $A_U(q,d)$ in
\eqref{eq:unitarygroup}, of order $dq^3(q^3+1)(q^2-1)$.
For $d=1$ the model is the Hermitian curve $H_q$.
\item The Ree family, with $p=3$, $q=3q_0^2\ge27$, $q_0=3^{f_0}$:
\[
 R(q,d):\quad
 v^q-v=x^d,\quad
 y^q-y=v^{q_0}(v^q-v),\quad
 z^q-z=v^{2q_0}(v^q-v),
 \qquad d\mid q-3q_0+1,
\]
where $d\ge1$, and
\[
 2g=(q-1)\bigl(d(q+1)(q+3q_0+1)-(q^2+q+1)\bigr).
\]
Its full group is $C_d\times\Ree(q)$, of order
$dq^3(q^3+1)(q-1)$.
\end{enumerate}
The subgroups $G$ are those described in Theorem~\ref{thm:subgroups},
together with $A_7\le\Aut(H_5)$ in characteristic $5$.
Every model above has zero $p$-rank. In each case the inequality
$|G|>24g(g-1)$ is to be imposed on the group under consideration.
In particular, a model belongs to the classification if and only if
its full automorphism group satisfies this inequality.
\end{theorem}

\begin{theorem}[Quadratic reduction outside the small Ree case]
\label{thm:quadratic-reduction}
Let $K$ be algebraically closed of odd characteristic $p$, let
$\cX/K$ have genus $g\ge2$ and $p$-rank zero, and let
$G\le\Aut(\cX)$ have no common fixed point. Let
$S\in\Syl_p(G)$ be noncyclic, put
$M=Z(N_G(S))$, and let $\Omega$ be the orbit of Sylow fixed points.
Suppose that
\[
 G/M\not\cong {}^2G_2(3)\cong\operatorname{P\Gamma L}(2,8).
\]
If
\[
 |G|>\max\{84(g-1),2g^2\},
\]
then $\cX/G\cong\PP^1$ and $G$ has exactly two short orbits, one wild
and one tame. Moreover,
\[
 |\Omega|=|S|+1,\qquad \cX/S\cong\PP^1.
\]
In particular, for $g\ge41$ the hypothesis $|G|>2g^2$ alone suffices.
\end{theorem}

Theorem~\ref{thm:quadratic-reduction} gives the orbit configuration
and rational Sylow quotient under a smaller hypothesis when $g\ge4$.
The
inertia-order exclusions and geometric reconstruction below retain
the stronger bounds stated where they are used. The group-theoretic
input is \cite[Theorem~3.16]{GMP}, so this reduction depends on the
classification of finite simple groups. In contrast, the uniform
orbit reduction in Theorem~\ref{thm:two} is a ramification argument.
The proof of Theorem~\ref{thm:geometric} also depends on the
classification of finite simple groups, through the same
group-theoretic reduction in the noncyclic Sylow case
(Lemma~\ref{lem:quadratic-data} and Proposition~\ref{prop:outer}).

The small Ree exception is necessary. The Ree curve $R_3$ has genus
$15$ and a ${}^2G_2(3)$-action of order $1512$ on its $28$ rational
points; see \cite[Sections~2 and~4]{MZRee}. Its degree-$27$ Sylow
cover, with quotient coordinate $u=v^3-v$, has one totally ramified
branch point; Deuring--Shafarevich gives zero $3$-rank. Its Sylow
normalizer is $9{:}6$, of order $54$ and with trivial center
\cite{AtlasL28}. Hurwitz gives
Sylow different $82$, hence different $53+(82-26)=109$ for the
point stabilizer of order $54$. Thus
\[
 2\cdot15-2=-2\cdot1512+28\cdot109.
\]
There is no further short orbit, although
$1512>\max\{84\cdot14,2\cdot15^2\}$. The group has no common
fixed point, since it is transitive on the $28$ Sylow fixed points.

\subsection{Relation to previous work}
Giulietti and Korchm\'aros \cite{GKzero} developed the orbit and
ramification method for curves of zero $2$-rank, determining the
possible large groups and their genera. Their result gives
a fixed point above $24g(g-1)$ outside the listed exceptions. Their
odd-characteristic solvable-group theorem
\cite[Theorem~4.1]{GKzero} also gives the bound $24g(g-1)$ under
its stated divisibility hypotheses. Thus the threshold used here
has an established role in this problem. The rationality and
intersection arguments in Section~\ref{sec:noncyclic}, and the
resulting genus identities, have characteristic-$2$ antecedents in
\cite[Section~5]{GKzero}; see also the explicit comparisons below.

Under a different local hypothesis, Giulietti and Korchm\'aros
\cite{GKfixed} characterized Hermitian, Suzuki and Ree-type curves
when the full group has no fixed point and a point stabilizer
contains a sufficiently large $p$-subgroup. Montanucci's
\cite{Montanucci} applies to
an arbitrary subgroup $G\le\Aut(\cX)$: above $24g^2$ there is
one wild short orbit, or one wild and one tame short orbit;
the one-orbit case has order at most $336g^2$, and order at least
$60g^2$ in that case forces positive $p$-rank. Hence, at $p$-rank
zero, the two-orbit conclusion already follows from $|G|\ge60g^2$.
Theorem~\ref{thm:two} lowers that threshold to
$\max\{84(g-1),8(g-1)^2\}$, which is compatible with
$24g(g-1)$ for $g\ge4$. Its one-orbit calculation extends the identity in
\cite{Montanucci} without assuming a rational
Sylow quotient. Under a rationality hypothesis on the Sylow quotient,
\cite{Montanucci} give the double transitivity
on the wild orbit and the trivial intersection of wild and tame point
stabilizers; compare Lemma~\ref{lem:intersection}.

The odd-characteristic results of Giulietti and Korchm\'aros
\cite{GKmany} give zero $p$-rank under even-genus hypotheses for
sufficiently large groups, and a sufficiently large solvable group
has zero $p$-rank and a fixed point. Bounds involving positive
$p$-rank and a rational point-stabilizer quotient are studied in
\cite{GKTpositive}. Transitivity on rational points of maximal
curves leads to a different classification in
\cite{TafazolianTransitive}. In \cite{TafazolianSharp} we classify
the zero $p$-rank curves whose full automorphism group fixes no point
and contains a $p$-subgroup of order larger than the genus; the
exclusion of field automorphisms in Proposition~\ref{prop:outer} and
the fixed-point argument in Lemma~\ref{lem:intersection} also appear
there. Here we start from zero $p$-rank and classify arbitrary
subgroups without any finite-field maximality assumption.

The curves in the conclusion belong to established covering towers.
The linear family is a cyclic quotient of $H_q$: in the model
$y^q-y=w^{q+1}$, put $x=w^{(q+1)/d}$. The unitary family is the
$n=3$ specialization of the covers in \cite{TTH}, and consists of
cyclic quotients of the Giulietti--Korchm\'aros curve
\cite{GKoriginal,GQZ}. The Ree family consists of cyclic quotients
of Skabelund's cover \cite{Skabelund}. The full groups in the linear
and unitary families were computed in \cite{BMZ} and \cite{MTZ},
and that of the top Ree cover in \cite{GMQZ}. For the Suzuki curve,
the full groups of the intermediate cyclic subcovers of Skabelund's
cover were determined in \cite[Theorem~8]{AFTV};
Proposition~\ref{prop:fullorders} gives the Ree analogue. In
\cite{TafazolianSkabelund} it is shown that the first map in each of
Skabelund's towers of ray class field covers of the Suzuki and Ree
curves is an isomorphism. The divisibility arguments in
Propositions~\ref{prop:unitarymodel} and~\ref{prop:reemodel} use
similar ingredients: the centrality of the Kummer group among the
lifted automorphisms, an involution interchanging two rational points,
and the first positive nongap at infinity. We use these results,
and prove the intermediate Ree case, before determining all large
subgroups and their central kernels. The contribution is the
converse classification from the group action to the equations,
together with the exact subgroup and realization conditions; it is
not a new computation of the previously known full groups.

\subsection{Covering degrees and the quadratic range}
The size condition selects particular quotients in each tower.
The following exact test is proved in
Section~\ref{subsec:parameter-tests}, after the full groups have
been determined.

\begin{proposition}[Parameters satisfying the largeness bound]
\label{prop:admissible-d}
For the full automorphism groups in Theorem~\ref{thm:geometric}, the
condition $|\Aut(\cX)|>24g(g-1)$ can be read directly from the covering
degree.
\begin{enumerate}[label=\textup{(\roman*)},leftmargin=*]
\item For $L(q,d)$ put $k=(q+1)/d$. Then the condition is equivalent to
\begin{equation}\label{eq:linear-k-test}
 (k-6)q^2+13kq-6k^2+6>0.
\end{equation}
In particular it always holds for $k\ge6$. For $2\le k\le5$ and
$g\ge4$, the additional possibilities are
\[
 (k,q)=(2,5),\quad (3,11),\quad
 (4,11),(4,19),(4,23),
\]
\[
 (5,9),(5,19),(5,29),(5,49),(5,59).
\]
\item For $U(q,d)$ put $k=(q^2-q+1)/d$. For $g\ge4$ the condition holds
if and only if $k\ge7$. In particular no unitary model with $q=3$
appears in the genus-$\ge4$ classification.
\item For $R(q,d)$ put $k=(q-3q_0+1)/d$. The condition holds if and
only if $k>1$; equivalently, every quotient of degree $k>1$ of the top
Skabelund cover satisfies the full-group bound, whereas the top cover
itself does not.
\end{enumerate}
\end{proposition}

\begin{remark}[The top covers]\label{rem:topcovers}
The top Giulietti--Korchm\'aros and Skabelund Ree covers are the
$k=1$ members of their respective towers. They do not satisfy
$|G|>24g(g-1)$, but their full groups satisfy $|\Aut(\cX)|>3g^2$;
explicit identities are given in Section~\ref{subsec:parameter-tests}.
Their genera are at least $41$, so both lie in the range of
Theorem~\ref{thm:quadratic-reduction}. Whether the complete
geometric classification extends from $24g(g-1)$ to that quadratic
range remains open. The constant $2$ in the quadratic theorem is
needed in the rationality step; in its noncyclic setting the
orbit step already holds under
$|G|>\max\{84(g-1),\frac{16}{9}(g-1)^2\}$.
\end{remark}

\subsection{Method of proof}
The argument has three stages. First, the zero-$p$-rank fixed-point
property and Hurwitz yield the uniform two-orbit reduction.
For a noncyclic Sylow subgroup, the theorem of Guralnick,
Malmskog and Pries \cite{GMP} gives an almost simple quotient.
Its Sylow action, the rationality of $\cX/S$, and the central
kernel lead to exact genus and inertia identities. The prime cyclic
case and the exceptional $A_7$ action are treated separately.

Second, two rigidity lemmas recover the equations. One determines an
elementary-abelian one-point cover from its irreducible tame action;
the other controls the central difference of two marked covers by
its weights and Swan conductors. Third, the central lift models
identify all subgroups with the prescribed projective image, giving
the exact existence and largeness tests. These steps distinguish
classification of curves from classification of the pairs
$(\cX,G)$.

Section~\ref{sec:prelim} records the common ramification and
rank-one data. Section~\ref{sec:two} proves the orbit reductions
and the numerical group lemma used in the quadratic theorem.
Sections~\ref{sec:noncyclic} and~\ref{sec:cyclic} handle the
noncyclic and cyclic Sylow cases. The equivariant covering tools are
proved in Section~\ref{sec:rigidity} and applied in
Section~\ref{sec:equations}. Section~\ref{sec:fullgroups} determines
the full groups and all large subgroups, and
Section~\ref{sec:smallgenera} treats genera $2$ and $3$.

\section{Preliminary results}\label{sec:prelim}

Our notation for curves and ramification follows \cite{HKT,SerreLocal}.
Unless a finite ground field is specified, all curves and all
automorphisms are over an algebraically closed field $K$ of
characteristic $p>2$.
The curve $\cX$ is smooth, projective and irreducible, and
$g=g(\cX)\ge2$ unless stated otherwise; quotient curves may have
smaller genus. We write $\gamma(\cX)$ for its $p$-rank, so that
$0\le\gamma(\cX)\le g$ and
$|\operatorname{Jac}(\cX)[p](K)|=p^{\gamma(\cX)}$.
Specialized covering and rank computations are recalled at their first use.
We write $q=p^f$ for a field size and, in the Ree case,
$q_0=3^{f_0}$, $q=3^{2f_0+1}$, so $f=2f_0+1$.
For an elementary-abelian group of order $Q=p^a$, $a=\log_p Q$
denotes its $\F_p$-dimension, independently of the Ree parameter $f_0$.

\subsection{Quotients, ramification and Hurwitz}\label{subsec:quotientprelim}
For a finite group $A\le\Aut(\cX)$, the quotient $\cX/A$ is the
smooth projective curve with function field $K(\cX)^A$. Its genus
is denoted by $g(\cX/A)$ and its $p$-rank by $\gamma(\cX/A)$.
For $P\in\cX$, let $A_P$ be the stabilizer of $P$. The orbit of
$P$ is short if $A_P\ne1$, and is wild if $p\mid|A_P|$.
An action is semiregular on a set if every point stabilizer is
trivial. A common fixed point is a point fixed by every element
of the group.
Since $K$ is algebraically closed, a point stabilizer is also its
inertia group, and its order is the ramification index of the quotient
map. A cover or group action is called tame if all these indices are
prime to $p$; a $p'$-group has order prime to $p$.

For a local parameter $u$ at $P$, the lower ramification groups are
\[
 (A_P)_i=\{\sigma\in A_P:v_P(\sigma(u)-u)\ge i+1\},
 \qquad i\ge0.
\]
The group $(A_P)_0=A_P$ has a unique Sylow $p$-subgroup
$(A_P)_1$, and $A_P/(A_P)_1$ is cyclic of order prime to $p$.
Moreover,
\[
 A_P=(A_P)_1\rtimes C_m,
 \qquad (m,p)=1,
\]
and the quotients $(A_P)_i/(A_P)_{i+1}$ are elementary abelian for
$i\ge1$. The tangent character embeds $C_m$ into $K^\times$.
For a subgroup $B\le A_P$, lower ramification groups are obtained
by intersection: $B_i=B\cap(A_P)_i$.

The Riemann--Hurwitz and Hilbert different formulas give
\begin{equation}\label{eq:RHprelim}
 2g(\cX)-2=|A|\bigl(2g(\cX/A)-2\bigr)
          +\sum_{P\in\cX}\sum_{i\ge0}(|(A_P)_i|-1).
\end{equation}
A tame stabilizer of order $e$ contributes $e-1$.
If a $p$-group $S$ of order $s$ acts on $\cX$ and its short orbits
have lengths $\ell_1,\ldots,\ell_b$, the Deuring--Shafarevich formula is
\begin{equation}\label{eq:DSprelim}
 \gamma(\cX)-1=s\bigl(\gamma(\cX/S)-1\bigr)
                    +\sum_{i=1}^b(s-\ell_i).
\end{equation}
The same Hurwitz formula applies to every smooth quotient curve,
not only to $\cX$. See \cite[Chapter~11]{HKT}.

\paragraph{Upper numbering and tame base change.}
For a totally ramified finite Galois extension with group $I$, set
$I_t=I_{\lceil t\rceil}$ for real $t\ge0$ and define the Herbrand
function and upper groups by
\begin{equation}\label{eq:Herbrandprelim}
 \varphi_I(t)=\int_0^t\frac{du}{[I_0:I_u]},\qquad
 \psi_I=\varphi_I^{-1},\qquad I^v=I_{\psi_I(v)}.
\end{equation}
A positive lower jump is an integer $j$ with $I_j\ne I_{j+1}$;
its upper jump is $\varphi_I(j)$. Upper numbering is compatible
with quotients: for $N\triangleleft I$,
$(I/N)^v=I^vN/N$.
This contrasts with lower numbering, which is compatible with
subgroups by intersection. In particular, if a
$p$-group has consecutive nontrivial lower groups of orders
$s,s_1,s_2$ and positive jumps $j_1<j_2<j_3$, its largest upper
jump is
\begin{equation}\label{eq:upperprelim}
 u_3=j_1+\frac{s_1}{s}(j_2-j_1)
           +\frac{s_2}{s}(j_3-j_2).
\end{equation}
For an abelian extension all upper jumps are integers by the
Hasse--Arf theorem. If a cyclic $p$-extension has order at least
$p^2$, its first two upper jumps satisfy $u_2\ge pu_1$.
Consequently, its first two lower jumps satisfy
\begin{equation}\label{eq:cyclicjumpprelim}
 j_2=j_1+p(u_2-u_1)\ge(p^2-p+1)j_1.
\end{equation}
See \cite[Chapter~IV]{SerreLocal} and \cite[Theorem~1.1(c), with trivial tame part]{OP}.
For a nontrivial degree-$p$ character, the word \emph{conductor}
will mean its Swan conductor, namely the sole positive upper jump
of the corresponding cyclic quotient. We assign conductor $0$ to
the trivial character. Thus, if the largest upper jump of a cover
is at most $U$, every degree-$p$ character obtained from it has
conductor at most $U$. The same quotient compatibility applies to
the absolute upper ramification groups: a finite quotient with
largest upper jump at most $U$ kills all groups of level $v>U$.

Under a totally ramified tame base change of degree $d$ prime to
$p$ at the branch point, a $p$-group cover keeps its Galois group
and its positive lower jumps are multiplied by $d$. Here one takes
the normalization of the fiber product. Locally the two extensions
are linearly disjoint because their degrees are coprime. More
generally, a tame local extension of index $e$ has index
$e/\gcd(e,d)$ after a tame base change of index $d$; this is the
form of Abhyankar's lemma used below. See
\cite[Chapter~IV, \S3]{SerreLocal} and \cite[Expos\'e~XIII]{SGA1}.

\begin{fact}[Intermediate covers]\label{fact:intermediate}
Let $U\le A$, let $Q\in\cX/A$, and let $I=A_P$ for a point $P$ above
$Q$. The points of $\cX/U$ above $Q$ correspond to the $I$-orbits on
the left cosets $A/U$. The orbit containing $aU$ has length
\[
 [I:I\cap aUa^{-1}],
\]
which is its ramification index in $\cX/U\to\cX/A$.
The corresponding stabilizer for $\cX\to\cX/U$ is conjugate to
$U\cap a^{-1}Ia$, and its lower groups are obtained by intersection
with the conjugates of $I_i$. These statements follow from the
double-coset description $I\backslash A/U$ of the fiber and
multiplicativity of ramification indices in a tower.
Furthermore, put
\[
 \operatorname{core}_A(U)=\bigcap_{a\in A}aUa^{-1}.
\]
The Galois closure of $K(\cX)^U/K(\cX)^A$ inside $K(\cX)$ is
$K(\cX)^{\operatorname{core}_A(U)}$: the compositum of the conjugate
fields is fixed by the intersection of the conjugate subgroups.
In particular, a core-free subgroup $U$ recovers the full $A$-cover
as the Galois closure of its intermediate cover.
\end{fact}

\begin{fact}[The Hurwitz threshold]\label{fact:hurwitz-orbits}
If $A\le\Aut(\cX)$ and $|A|>84(g-1)$, then $p\mid|A|$,
$\cX/A$ is rational, and the short orbits consist of one wild
orbit, two wild orbits, one wild and one tame orbit, or one wild
and two tame orbits; see \cite[Theorem~11.56]{HKT}.
\end{fact}

\begin{fact}[Tame covers with at most two branch points]\label{fact:tame-two}
Let $Y\to\PP^1$ be a connected nontrivial finite separable tame
cover branched over at most two points. After choosing coordinates,
it is $z\mapsto z^n$ with $p\nmid n$. In particular, it is cyclic and totally ramified at
exactly two points. Indeed, put the possible branch points at
$0,\infty$ and let $r_0,r_\infty$ be their numbers of preimages.
Tame Hurwitz gives $2g(Y)-2=-r_0-r_\infty$, forcing
$g(Y)=0$ and $r_0=r_\infty=1$. The divisor of the map on the
rational source then gives $cz^n$, and the constant is absorbed
by a coordinate change. Hurwitz also shows that $\PP^1$ has no
nontrivial connected finite \'etale cover.
\end{fact}

\subsection{Zero \texorpdfstring{$p$}{p}-rank and Sylow fixed points}\label{subsec:sylowprelim}
Let $G\le\Aut(\cX)$ be finite and let $S\in\Syl_p(G)$ be nontrivial.  Put
$s=|S|$.

\begin{lemma}\label{lem:sylow}
Assume that $\gamma(\cX)=0$.  Then every nontrivial $p$-subgroup of
$\Aut(\cX)$ fixes a unique point and acts freely away from that point.
The fixed points of the Sylow $p$-subgroups form one $G$-orbit $\Omega$.
If $P=P_S$ is the unique fixed point of $S$, then
\[
 G_P=N_G(S),\qquad G_P=S\rtimes C_m,
\]
where $C_m$ is cyclic of order $m$ prime to $p$.  Moreover the map
$S\mapsto P_S$ is a bijection from $\Syl_p(G)$ onto $\Omega$, and
\begin{equation}\label{eq:sylowcount}
 n:=|\Omega|=[G:G_P]\equiv1\pmod{s}.
\end{equation}
If $G$ has no global fixed point, then
\begin{equation}\label{eq:nlower}
 n\ge s+1,\qquad |G|\ge s(s+1).
\end{equation}
\end{lemma}

\begin{proof}
An automorphism of order $p$ has exactly one fixed point when
$\gamma(\cX)=0$; see \cite[Lemma~3.1]{GKzero}.  Choosing an element of
order $p$ in the center of a nontrivial $p$-group shows that the whole
$p$-group fixes the same point.  Any nonidentity element has a power of
order $p$, and hence cannot fix any other point.  If two Sylow $p$-subgroups intersect nontrivially, their fixed
points coincide: a common element of order $p$ has only one fixed
point. Both subgroups then lie in the stabilizer of that point,
whose Sylow $p$-subgroup is unique, and hence they coincide.
Thus distinct Sylow $p$-subgroups have trivial intersection; cf.\
\cite[Lemma~3.3 and Theorem~3.4]{GKzero}.

The first ramification group of a point stabilizer is its unique Sylow
$p$-subgroup; hence $S$ is normal in $G_P$, and $G_P\le N_G(S)$.
Conversely, every element of $N_G(S)$ preserves the fixed-point set
$\{P\}$ of $S$, so $N_G(S)\le G_P$. This gives $G_P=N_G(S)$ and the
usual decomposition $G_P=S\rtimes C_m$ with $C_m$ cyclic and prime
to $p$.  The correspondence
between Sylow subgroups and their fixed points is $G$-equivariant and
bijective.  Since $S$ acts freely on $\Omega\setminus\{P\}$, the order
$s$ divides $n-1$.  If $n>1$, then $n\ge s+1$.  Finally,
$|G|=nsm$, which gives \eqref{eq:nlower}.
\end{proof}

\subsection{The tame action on ramification quotients}\label{subsec:ramificationprelim}
Write $I=G_P$ and $I_i=(G_P)_i$. For $i,j\ge1$,
\begin{equation}\label{eq:commutatorprelim}
 [I_i,I_j]\le I_{i+j+1};
\end{equation}
see \cite[Chapter~IV, \S2, Proposition~10]{SerreLocal}. In particular,
\begin{equation}\label{eq:equal-commutator}
 [I_1,I_1]\le I_3.
\end{equation}
For every $i\ge1$, not only the first jump, the coefficient map
\begin{equation}\label{eq:gradedprelim}
 I_i/I_{i+1}\longrightarrow(K,+),\qquad
 \sigma(u)=u+a_\sigma u^{i+1}+\cdots\longmapsto a_\sigma
\end{equation}
is an injective homomorphism. If a tame element $\tau$ has tangent
multiplier $\zeta$, then, under the convention
$\tau(u)=\zeta u$ for the action on local functions,
conjugation acts on its image by $\zeta^i$.
See \cite[Chapter~IV, \S2]{SerreLocal}.
Consequently each nonzero graded quotient is an isotypic
$\F_p[\langle\tau\rangle]$-module: it is a vector space over
$\F_p(\zeta^i)$, and hence a direct sum of copies of the same
irreducible scalar module. In particular, two nonisomorphic
irreducible tame modules cannot occur in the same ramification
graded quotient.

At a wild point $P$, write $S=(G_P)_1$, $s=|S|$, and
$G_P=S\rtimes C_m$. Put
\[
 \mathcal D_P=\sum_{i\ge1}(|(G_P)_i|-1).
\]
The following form of the tame-character bound is useful.

\begin{lemma}[Local tame-complement bound]\label{lem:local}
Let $j\ge1$ be the first lower jump, so that
$(G_P)_1=\cdots=(G_P)_j=S$ and $(G_P)_{j+1}<S$.  If
\[
 |(G_P)_j/(G_P)_{j+1}|=p^b,
\]
then
\begin{equation}\label{eq:local}
 m\mid j(p^b-1),
 \qquad
 m\le j(p^b-1)\le j(s-1)\le \mathcal D_P.
\end{equation}
\end{lemma}

\begin{proof}
Let $u$ be a uniformizer at $P$.  The tangent character embeds $C_m$ into
$K^\times$.  If a generator of $C_m$ has tangent multiplier $\zeta$, then
the standard coefficient map
\[
 (G_P)_j/(G_P)_{j+1}\longrightarrow (K,+),
 \qquad
 \sigma(u)=u+a_\sigma u^{j+1}+\cdots\longmapsto a_\sigma
\]
identifies the quotient with an $\F_p$-vector space $V$ of size $p^b$.
With the convention $\tau(u)=\zeta u$ for the action on
local functions, conjugation by $C_m$ acts on $V$ by the scalar
$\zeta^j$. Replacing this convention by the inverse action
inverts the scalar and does not affect its order.
The order of this scalar is $m/\gcd(m,j)$, and it acts freely on
$V\setminus\{0\}$.  Hence
\[
 \frac{m}{\gcd(m,j)}\mid p^b-1,
\]
which gives the first divisibility in \eqref{eq:local}.  The remaining
inequalities are immediate from $p^b\le s$ and the definition of $\mathcal D_P$.
\end{proof}

\subsection{Artin--Schreier covers of the affine line}\label{subsec:coversprelim}
Artin--Schreier theory describes torsors on $\mathbb A^1_K$ under
the constant additive group $(\F_Q,+)$, where $Q=p^a$. Writing
$\wp_Q(h)=h^Q-h$, it gives
\begin{equation}\label{eq:AScohomology}
 H^1_{\mathrm{\acute et}}(\mathbb A^1_K,\F_Q)
      =K[x]/\{h^Q-h:h\in K[x]\}.
\end{equation}
The class $[f]$ corresponds to $y^Q-y=f(x)$; a torsor need not be
connected. A polynomial is called $Q$-reduced if no nonconstant
exponent is divisible by $Q$, and $p$-reduced when $Q=p$.
Every class in \eqref{eq:AScohomology} has a unique $Q$-reduced
representative with zero constant term. Existence follows by
successively replacing $c x^{Qj}$ by $c^{1/Q}x^j$ and removing
constants; uniqueness follows because the leading degree of
$h^Q-h$ is divisible by $Q$ for nonconstant $h$.
For a connected elementary-abelian cover with group $V$, its
degree-$p$ characters form an $\F_p$-subspace
\[
 W\subset K[x]/\wp_p(K[x]),\qquad
 \dim_{\F_p}W=\log_p|V|.
\]
Conversely, the compositum of the cyclic covers belonging to a
basis of $W$ recovers the cover. For $y^Q-y=f(x)$ the characters
are represented by $[\lambda f]$ modulo $\wp_p$, with
$\lambda\in\F_Q$, using the trace pairing on $\F_Q/\F_p$.
These classes have dimension $a$ precisely when the torsor is
connected. See \cite[Tags~0A3J and~03P2]{Stacks} and \cite{HKT}.
For a cyclic cover of $\PP^1$ of degree $p$ with a single branch point,
one may choose an equation $y^p-y=f(x)$ with $f$ a polynomial
whose nonconstant exponents are prime to $p$.
If $j=\deg f>0$, its sole positive lower and upper jump is $j$, and
\[
 g=\frac{(p-1)(j-1)}2.
\]
\subsection{Rank-one groups and the GMP theorem}\label{subsec:groupsprelim}
A subgroup $S$ of a finite group $A$ is a TI subgroup if
$S\cap S^a=1$ whenever $a\notin N_A(S)$. Write $O_{p'}(A)$
for the largest normal subgroup of order prime to $p$.
A group $B$ is \emph{almost simple} if
$L\le B\le\Aut(L)$ for a nonabelian simple group $L$;
its \emph{socle} is $L$, denoted $\operatorname{soc}(B)$.
A proper subgroup $J<A$ is \emph{strongly $p$-embedded} if
$p\mid |J|$ but $p\nmid |J\cap J^a|$ for every $a\notin J$.
We use the following odd-characteristic form of
\cite[Theorem~3.16]{GMP}, with the nonnormality hypothesis stated
explicitly; see also its geometric application in
\cite[Theorem~3.19]{GMP}. The Suzuki alternative does not occur in
odd characteristic.

\begin{theorem}[Guralnick--Malmskog--Pries]\label{thm:GMPprelim}
Let $S$ be a Sylow $p$-subgroup of a finite group $A$, with $p>2$ and
$S$ noncyclic. Suppose that $S$ is a TI subgroup, that $S$ is not
normal in $A$, and that
\[
 I=N_A(S)=S\rtimes C
\]
with $C$ cyclic of order prime to $p$. Put $M=Z(I)$. Then
\[
 M=O_{p'}(A),
\]
$A/M$ is almost simple, and its socle is isomorphic to one of
\[
 \PSL(2,p^f)\ (f\ge2),\qquad
 \PSU(3,p^f)\ (p^f>2),\qquad
 {}^2G_2(3^{2f_0+1})'\ (p=3,\ f_0\ge0).
\]
Moreover, $A$ acts doubly transitively on its Sylow $p$-subgroups.
\end{theorem}

\begin{remark}[The nonnormality hypothesis in GMP]
\label{rem:GMP-nonnormal}
The printed statement of \cite[Theorem~3.16]{GMP} assumes only
$S\ne A$, whereas its proof uses the stronger condition
$N_A(S)\ne A$. This distinction matters: for odd $p$, the group
$(C_p\times C_p)\rtimes C_2$, with the involution acting by
inversion, satisfies the printed hypotheses with
$S=C_p\times C_p$ normal; here $Z(A)=O_{p'}(A)=1$, but $A$ is
solvable and hence $A/Z(A)$ is not almost simple.

For completeness, the strengthened assumption supplies both
properties needed by the cited proof. If $I=N_A(S)<A$, then
for $a\notin I$ every $p$-element of $I\cap I^a$ would lie in
$S\cap S^a=1$. Thus $I$ is a proper strongly
$p$-embedded subgroup of $A$, and in particular $O_p(A)=1$.
Moreover, Lemma~3.15 of \cite{GMP}, applied to
$R=O_{p'}(A)$, implies $R\le I$; normality and coprimality
then give $[R,S]=1$, and the cyclicity of $I/S$ yields
$R\le Z(I)$. The quotient $I/R$ remains a \emph{proper}
strongly $p$-embedded subgroup of $A/R$. The classification
step cited in \cite[Theorem~3.16]{GMP} applies to this proper
subgroup and gives the stated almost simple possibilities and
$Z(I/R)=1$, hence $Z(I)=R$.
The geometric formulation \cite[Theorem~3.19]{GMP} already
assumes $A\ne N_A(S)$. The argument above supplies the nonnormality
hypothesis required by the cited proof.
\end{remark}

\paragraph{The standard rank-one actions.}
We briefly recall the permutation data required in the proofs.
The linear groups act naturally on $\PP^1(\F_q)$, of cardinality
$q+1$; the unitary groups act on the $q^3+1$ isotropic points of
the Hermitian unital, and the Ree groups act on their natural
$q^3+1$-point ovoids. All three actions are doubly transitive.
A Sylow $p$-subgroup fixes a point and acts regularly on its
complement.
In the linear case, for example, its action is the translation action
$x\mapsto x+b$ on $\F_q\subset\PP^1(\F_q)$.
Thus, if $S$ is the Sylow $p$-subgroup in the natural action, and
$n$ is the degree, then
\[
 n=|S|+1.
\]
The stabilizer of two distinct points is a cyclic, prime-to-$p$
complement to $S$ in a point stabilizer. A central $p'$-kernel in
$G$ can enlarge this complement: the table records the faithful
orders in the quotient, not necessarily $m=|G_P|/|S|$ for $G$.
The small Ree boundary must be kept separate:
${}^2G_2(3)'\cong\PSL(2,8)$ has cyclic Sylow $3$-subgroups,
whereas its field-extension group
${}^2G_2(3)\cong\operatorname{P\Gamma L}(2,8)$ has a noncyclic
Sylow $3$-subgroup and a $28$-point action.

The orders of the Sylow subgroups and faithful two-point stabilizers
in the standard actions are
\begin{equation}\label{eq:rankonedata}
\begin{array}{c|c|c}
 A&|S|&\text{two-point stabilizer}\\ \hline
 \PSL(2,q),\ \PGL(2,q)&q&(q-1)/2,\ q-1\rule{0pt}{2.5ex}\\
 \PSU(3,q),\ \PGU(3,q)&q^3&(q^2-1)/\mu,\ q^2-1\\
 \Ree(q)&q^3&q-1
\end{array}
\end{equation}
where $\mu=(3,q+1)$. Thus the natural degrees are $q+1$ and
$q^3+1$. These are the data for the named groups themselves; see
\cite[Section~5]{GKmany} for the linear and unitary normalizers and
\cite[Section~3]{GMQZ} for the Ree group.

More precisely, for a simple socle
$L\in\{\PSL(2,q),\PSU(3,q),\Ree(q)\}$ in the ranges
above, let $U\in\Syl_p(L)$, write $u=|U|$, and let $b_L$ be
respectively $(q-1)/2$, $(q^2-1)/\gcd(3,q+1)$, or $q-1$.
The standard order and Sylow-normalizer formulas give
\[
 |N_L(U)|=u b_L,\qquad
 |L|=u(u+1)b_L,\qquad
 |\Syl_p(L)|=u+1.
\]
Here $u=q$ for $\PSL(2,q)$ and $u=q^3$ for the other two
families. We will use these formulas only inside the simple socle;
the outer index and the central kernel will remain explicit in
Lemma~\ref{lem:quadratic-data}.

The necessary orders $e>1$ of cyclic semiregular
subgroups are as follows:
\[
\begin{array}{c|l}
\PSL(2,q),\ \PGL(2,q)&e\mid(q+1)/2,\quad e\mid q+1,
  \text{ respectively};\\
\PSU(3,q),\ \PGU(3,q)&e\mid(q+1)/2\quad\text{or}\quad
 e\mid(q^2-q+1)/\delta;\\
\Ree(q)&e\mid(q+1)/4\quad\text{or}\quad e\mid q\pm3q_0+1.
\end{array}
\]
Here $\delta=(3,q+1)$ in the special unitary case and $\delta=1$
in the projective unitary case. Involutions in the unitary action fix $q+1$ isotropic points
\cite[Lemma~2.2]{MZ}; Ree involutions also fix $q+1$ points
\cite[Section~5]{GKsemi}. The semiregular
orders follow from \cite[Propositions~3.2 and~5.2]{GKsemi}
and the unitary element types in \cite[Lemma~2.2]{MZ}. Further exclusions using the order bound
are proved in Section~\ref{sec:rankone}.
The groups $\operatorname{SL}_2(q)$ for $q\ge5$,
$\operatorname{SU}_3(q)$ for $q\ge3$, and $\Ree(q)$ for $q\ge27$
are perfect. For the first two, the centers consist of scalar matrices with
$c^2=1$ and with $c^3=c^{q+1}=1$, respectively, hence have orders
$2$ and $(3,q+1)$. See \cite[Section~5]{GKmany} and
\cite[Section~2.2]{GMP} for the classical group descriptions.

\begin{fact}[Schur--Zassenhaus]\label{fact:SZ}
If $N\triangleleft B$ is a finite $p$-group and $B/N$ has order
prime to $p$, then $B=N\rtimes C$ for a complement $C$, and any
two complements are conjugate by an element of $N$; see
\cite[Theorem~1 and Remark~5]{ConradSZ}.
\end{fact}

\begin{fact}[Dickson]\label{fact:Dickson}
For finite subgroups of $\PGL_2(K)$, the following classification holds. A tame subgroup is cyclic, dihedral, $A_4$, $S_4$,
or $A_5$. A subgroup whose order is divisible by $p$ either
fixes a point, is of linear type $\PSL(2,q)$ or $\PGL(2,q)$,
or is the exceptional $A_5$ in characteristic $3$; see
\cite[Theorem~A.8]{HKT}. In particular, a subgroup of
$\PGL(2,p)$ containing a $p$-element either fixes a point of
$\PP^1(\F_p)$ or contains $\PSL(2,p)$.
\end{fact}

\subsection{The Hermitian and Ree curves}\label{subsec:standardcurves}
The Hermitian curve $H_q:z^q+z=y^{q+1}$ has genus
$q(q-1)/2$, full geometric automorphism group $\PGU(3,q)$,
and $q^3+1$ points over $\F_{q^2}$. At a rational point its
Weierstrass semigroup is generated by $q,q+1$.
Its Sylow group is a Heisenberg group of order $q^3$ and exponent $p$;
its center
and commutator subgroup coincide and have order $q$. The two
positive lower jumps for its one-point Sylow cover are
$1,q+1$. The tame action on the quotient and center is by
$\alpha$ and $\alpha^{q+1}$, respectively, with
$\alpha\in\F_{q^2}^{\times}$; see \cite{HKT,GMP}.

For $q=3q_0^2\ge27$, the Ree curve has equations
\[
 y^q-y=v^{q_0}(v^q-v),\qquad
 z^q-z=v^{2q_0}(v^q-v).
\]
Its genus is $\frac32q_0(q-1)(q+q_0+1)$, its full geometric
group is $\Ree(q)$, and it has $q^3+1$ points over $\F_q$.
A Sylow group $S$ has orders
$|S|=q^3$, $|S'|=q^2$, and $|Z(S)|=q$.
It has exponent $9$, its derived subgroup $S'$ is elementary abelian,
and $[S,S']=Z(S)$ and $S^3\le Z(S)$. In particular,
$\Phi(S)=S'$; see \cite[Proposition~49 and the discussion preceding it]{GMQZ}.
On the successive factors the torus acts with weights
$1,q_0+1,2q_0+1$. The three positive lower jumps are
$1,3q_0+1,q+3q_0+1$; see \cite{Skabelund,GMQZ}.
The cyclic cover of degree $q-3q_0+1$ constructed in
\cite{Skabelund} admits the full lift of $\Ree(q)$.
The full group of that cover is
$C_{q-3q_0+1}\times\Ree(q)$ \cite{GMQZ}.

\paragraph{Standing convention.}
Unless a result states otherwise, the curve under classification has
zero $p$-rank. Order bounds and the absence of a common fixed point
are imposed explicitly where needed. Auxiliary covers and model
families are considered with the assumptions stated for them.

\section{Orbit structure and quadratic reduction}\label{sec:two}

We begin by determining the possible short orbits of $G$. The first
reduction is uniform in the Sylow $p$-subgroup: a suitable order bound
forces a rational quotient and exactly two short orbits, one wild and
one tame. For noncyclic Sylow subgroups outside the small Ree case,
the same configuration follows from a sharper quadratic bound, which also
forces the Sylow quotient to be rational. Both arguments use the same
ramification notation, but the second does not rely on the stronger
order hypothesis of the first.

\subsection{The uniform two-orbit reduction}

For the ramification calculations, let $S\in\Syl_p(G)$ be nontrivial,
put $s=|S|$, and let $P=P_S$ be its unique fixed point.
By Lemma~\ref{lem:sylow}, we may write
\[
 G_P=S\rtimes H,\qquad |H|=m,
\]
where $H$ is cyclic of order prime to $p$. Let
$\Omega=G\cdot P$ be the orbit of Sylow fixed points and put
$n=|\Omega|$. Set
\begin{equation}\label{eq:Tdef}
 \mathcal D_P=\sum_{i\ge1}\bigl(|(G_P)_i|-1\bigr),
 \qquad x=2g-2,
 \qquad h=g(\cX/S).
\end{equation}
Since $\gamma(\cX)=0$, Lemma~\ref{lem:sylow} shows that $S$
fixes precisely $P$ and acts freely on $\cX\setminus\{P\}$.
Thus $P$ is the unique ramification point of the covering
$\cX\to\cX/S$. Moreover, the lower ramification groups of
$S$ at $P$ satisfy
\[
 S_0=S,\qquad S_i=(G_P)_i\quad\text{for }i\ge1,
\]
where the second equality follows from the intersection property
for lower ramification groups. By Hilbert's different formula,
the different exponent at $P$ is therefore
\[
 d_P=\sum_{i\ge0}(|S_i|-1)=(s-1)+\mathcal D_P.
\]
The Riemann--Hurwitz formula gives
\begin{equation}\label{eq:hurwitzS}
 x=s(2h-2)+(s-1)+\mathcal D_P
  =2sh-s-1+\mathcal D_P.
\end{equation}
We shall compare this identity with the Riemann--Hurwitz formula
for $\cX\to\cX/G$ to exclude the other possible short-orbit
configurations.

\begin{theorem}[Uniform two-orbit reduction]\label{thm:two}
Assume $\gamma(\cX)=0$. Suppose that $G$ fixes no point and
\begin{equation}\label{eq:weaklarge}
 |G|>\max\{84(g-1),8(g-1)^2\}.
\end{equation}
Then $\cX/G$ is rational and $G$ has exactly two short orbits: the wild
orbit $\Omega$ and one tame orbit $\Delta$.  No assumption on the
structure of the Sylow $p$-subgroups is required.

In particular, if $g\ge4$ and
\[
 |G|>24g(g-1),
\]
then the conclusion holds.
\end{theorem}

\begin{proof}
Since $|G|>84(g-1)$, the classical tame Hurwitz bound implies that
$p$ divides $|G|$; in particular the Sylow subgroup $S$ and its
fixed point $P$ introduced above exist. The large-group orbit theorem
\cite[Theorem~11.56]{HKT} gives $\cX/G\cong\PP^1$ and leaves four
possibilities for the short $G$-orbits: one wild orbit; two wild
orbits; one wild and two tame orbits; or one wild and one tame orbit.

We first identify the wild orbit.  If $Q$ is any wild point, a
nontrivial $p$-subgroup of $G_Q$ lies in some Sylow $p$-subgroup $S'$.
Since this subgroup fixes $Q$, and every nontrivial $p$-subgroup
fixes a unique point by Lemma~\ref{lem:sylow}, we have $Q=P_{S'}$.
All Sylow $p$-subgroups are conjugate, and their fixed points form
the single orbit $\Omega$ by the same lemma.  Thus every wild
point belongs to $\Omega$, excluding the possibility of two wild
orbits.

By orbit--stabilizer, $|G|=nsm$.  The group $S$
fixes $P$ and acts freely on $\Omega\setminus\{P\}$, so
$s\mid(n-1)$.  Since $G$ fixes no point, $n>1$; hence
\[
 n\ge s+1,\qquad s<n.
\]
For the $G$-cover, the different exponent at $P$ is
$(sm-1)+\mathcal D_P$.  We now eliminate the two remaining
unwanted orbit configurations.

\emph{Only the wild orbit $\Omega$.}
If $\Omega$ is the only short orbit, it contains all the ramified
points of $\cX\to\cX/G$.  Riemann--Hurwitz therefore yields
\[
 \begin{aligned}
 x&=-2|G|+n(sm-1+\mathcal D_P)\\
  &=n(\mathcal D_P-sm-1).
 \end{aligned}
\]
Since $x>0$, the integer $\mathcal D_P-sm-1$ is positive, and
therefore $n\le x$.  Eliminating $\mathcal D_P$ between this
identity and \eqref{eq:hurwitzS}, we obtain
\[
 sm=s(1-2h)+x\left(1-\frac1n\right)<s+x,
\]
where $h\ge0$.  Using $s<n\le x$ gives
\[
 |G|=nsm<n(s+x)<2x^2=8(g-1)^2,
\]
contrary to the assumed lower bound.

\emph{One wild orbit and two tame orbits.}
Let $e_1,e_2\ge2$ be the orders of the stabilizers on the tame
orbits.  A tame orbit with stabilizer order $e_i$ contributes
$|G|(1-1/e_i)$ to the different.  Thus
\[
 \frac{x}{|G|}
 =1+\frac{\mathcal D_P-1}{sm}
   -\frac1{e_1}-\frac1{e_2}.
\]
Since $(G_P)_1=S$ and $s\ge3$, we have
$\mathcal D_P\ge s-1\ge2$.  If at least one $e_i$ is at least $3$,
then $1/e_1+1/e_2\le1/3+1/2$, and
\[
 \frac{x}{|G|}>\frac16.
\]
Consequently $|G|<6x=12(g-1)$, contradicting
$|G|>84(g-1)$.  Hence $e_1=e_2=2$.

The two tame orbits now contribute $|G|$ in total.  Applying
Riemann--Hurwitz once more,
\[
 x=-2|G|+n(sm-1+\mathcal D_P)+|G|
  =n(\mathcal D_P-1).
\]
In particular $n\le x$.  The local tame-complement bound,
Lemma~\ref{lem:local}, gives $m\le\mathcal D_P$.  Hence
\[
 \begin{aligned}
 |G|=nsm
 &\le ns\mathcal D_P
   =s(x+n)\\
 &<2x^2=8(g-1)^2,
 \end{aligned}
\]
again a contradiction.

The only remaining configuration is one wild orbit $\Omega$ and one
tame orbit $\Delta$.  Finally, when $g\ge4$,
\[
 84(g-1)\le24g(g-1),\qquad
 8(g-1)^2<24g(g-1),
\]
which proves the last assertion.
\end{proof}

\begin{remark}[Comparison with the two-orbit literature]\label{rem:montanucci}
Montanucci \cite{Montanucci} treated
arbitrary automorphism subgroups in odd characteristic. Her
results imply, when $\gamma(\cX)=0$, that $|G|\ge60g^2$
forces exactly two short orbits, one wild and one tame: the
alternative of a single wild short orbit would force positive
$p$-rank. Theorem~\ref{thm:two} supplies the same conclusion
under the smaller threshold $\max\{84(g-1),8(g-1)^2\}$.
We use this reduction because it also applies at the threshold
$|G|>24g(g-1)$ relevant to the geometric classification.
The one-orbit calculation extends the identity used in the proof of
\cite{Montanucci} to a Sylow quotient of arbitrary genus.
The improvement here concerns the numerical bound, not the passage
from the full automorphism group to subgroups.
\end{remark}

\begin{corollary}[Exact relations in the remaining configuration]
\label{cor:relations}
Assume $\gamma(\cX)=0$ and the hypotheses of Theorem~\ref{thm:two}.
Let $R\in\Delta$ and put
$e=|G_R|$.  Then $e>1$, $p\nmid e$, and
\begin{align}
 \frac{x}{n}&=\mathcal D_P-1-\frac{sm}{e},\label{eq:general1}\\
 2g&=2sh+\mathcal D_P-s+1,\label{eq:general2}\\
 n&\equiv1\pmod{s},\qquad
 n\ge s+1,\qquad
 m\le \mathcal D_P.\label{eq:general3}
\end{align}
\end{corollary}

\begin{proof}
Only $\Omega$ and $\Delta$ are short, and their different
contributions are respectively $n(sm-1+\mathcal D_P)$ and
$|G|(1-1/e)$.  Since $\cX/G\cong\PP^1$ and $|G|=nsm$,
Riemann--Hurwitz gives
\[
 x=-2|G|+n(sm-1+\mathcal D_P)
   +|G|\left(1-\frac1e\right)
  =n\left(\mathcal D_P-1-\frac{sm}{e}\right),
\]
which is \eqref{eq:general1}.
Equation~\eqref{eq:general2} is a rearrangement of
\eqref{eq:hurwitzS}; \eqref{eq:general3} follows from
Lemma~\ref{lem:sylow} and the local bound
Lemma~\ref{lem:local}.
\end{proof}

\subsection{The quadratic refinement outside the small Ree case}
\label{subsec:quadratic-data}

Theorem~\ref{thm:two} applies without restricting the Sylow
$p$-subgroups, but its quadratic order bound is not available under
the weaker hypothesis of Theorem~\ref{thm:quadratic-reduction}.
Lemma~\ref{lem:quadratic-data} provides precisely the extra information
needed here: $n=s+1$, the lower bound $m\ge4$, and the short list
of numerical possibilities when $m\le8$. We use these data to
exclude the unwanted orbit configurations directly and then prove
that $\cX/S$ is rational. The identification of the whole group
$G/M$ in Proposition~\ref{prop:outer} is not used in this proof.

\paragraph{Numerical Sylow-orbit data.}

We now extract the group-theoretic information needed in
Theorem~\ref{thm:quadratic-reduction}. We keep the index of
the socle explicit: identifying the whole almost simple group is not
needed for this numerical step.

\begin{lemma}[Numerical Sylow-orbit data]
\label{lem:quadratic-data}
Assume that $\gamma(\cX)=0$, that $G\le\Aut(\cX)$ has no common
fixed point, and that $S\in\Syl_p(G)$ is noncyclic, with $p>2$.
Let $P$ be the fixed point of $S$, and put
\[
 \Omega=G\cdot P,\qquad n=|\Omega|,\qquad s=|S|,\qquad
 m=|G_P|/s.
\]
Set $M=Z(N_G(S))$, $\bar G=G/M$, and $r=|M|$; by
Theorem~\ref{thm:GMPprelim}, applied in the proof,
$M=O_{p'}(G)$. Suppose that $\bar G\not\cong{}^2G_2(3)$.
Write $L=\operatorname{soc}(\bar G)$ and $\iota=[\bar G:L]$.
Then the following statements hold.
\begin{enumerate}[label=\textup{(\roman*)},leftmargin=*]
\item $M$ is the kernel of the action on $\Omega$, and
$\bar G$ acts faithfully and doubly transitively on $\Omega$.
The image $\bar S$ of $S$ is a Sylow $p$-subgroup of $L$.
In particular,
\begin{equation}\label{eq:quadratic-data-degree}
 p\nmid\iota,\qquad n=s+1,\qquad s\ge p^2\ge9.
\end{equation}

\item For $P'\in\Omega\setminus\{P\}$, the pointwise stabilizer
$D=G_{P,P'}$ is cyclic, $|D|=m$, and $G_P=S\rtimes D$.
Moreover,
\begin{equation}\label{eq:quadratic-data-order}
 m=r\iota b_L,\qquad
 b_L=\begin{cases}
 (q-1)/2,&L\cong\PSL(2,q),\quad q=p^f,\ f\ge2,\\
 (q^2-1)/\gcd(3,q+1),&L\cong\PSU(3,q),\quad q=p^f\ge3,\\
 q-1,&L\cong\Ree(q),\quad q=3^{2f_0+1}\ge27.
 \end{cases}
\end{equation}
Here $s=q$ in the linear case and $s=q^3$ in the other two
cases. Consequently $m\ge4$.

\item If $m\le8$, the numerical possibilities are contained in
\begin{equation}\label{eq:quadratic-small-data}
\begin{array}{c|r|r|r}
 L&s&m&r\iota\\ \hline
 \PSL(2,9)&9&4&1\rule{0pt}{2.5ex}\\
 \PSL(2,9)&9&8&2\\
 \PSU(3,3)&27&8&1\\
 \PSU(3,5)&125&8&1
\end{array}.
\end{equation}
In the unitary rows, $r=\iota=1$ and
$[S,S]=Z(S)$ has order $q$.
No assertion about which index-two extension occurs in the second
row is needed here.
\end{enumerate}
\end{lemma}

\begin{proof}
Lemma~\ref{lem:sylow} gives $G_P=N_G(S)$, the TI property, and
the free action of $S$ away from $P$. If $S$ were normal in $G$,
its unique fixed point would be fixed by $G$. Thus
$N_G(S)<G$, and Theorem~\ref{thm:GMPprelim} applies; in
particular $M=O_{p'}(G)$.

Since $M\triangleleft G$ and $M\le G_P$, it fixes $\Omega$
pointwise. Let $\mathcal K$ be the kernel of this action. A nontrivial
$p$-element of $\mathcal K$ would fix the distinct points of $\Omega$,
contrary to Lemma~\ref{lem:sylow}. Hence $\mathcal K$ is a normal
$p'$-subgroup, so $\mathcal K\le O_{p'}(G)=M$. Therefore $\mathcal K=M$.
Theorem~\ref{thm:GMPprelim}, together with the equivariant
Sylow-fixed-point bijection of Lemma~\ref{lem:sylow}, gives
double transitivity.

The projection to $\bar G$ is injective on every Sylow
$p$-subgroup. Its image has the same fixed point and acts freely
elsewhere on $\Omega$. Each point stabilizer in $\bar G$ has a
unique Sylow $p$-subgroup, the image of the one in $G_P$.
Consequently the points of $\Omega$ are also in equivariant
bijection with $\Syl_p(\bar G)$, and distinct Sylow
$p$-subgroups of $\bar G$ have trivial intersection.

Since $L\triangleleft\bar G$, for every
$\bar T\in\Syl_p(\bar G)$ the intersection $\bar T\cap L$
is a Sylow $p$-subgroup of $L$. The map
\[
 \Syl_p(\bar G)\longrightarrow\Syl_p(L),\qquad
 \bar T\longmapsto\bar T\cap L
\]
is onto, because every Sylow subgroup of $L$ is contained in one
of $\bar G$. It is injective by the TI property, since $p$ divides
$|L|$. Thus
\begin{equation}\label{eq:socle-sylow-bijection}
 n=|\Syl_p(L)|.
\end{equation}
The bijection is $\bar G$-equivariant for the conjugation action
on $\Syl_p(L)$.

The small socle $L\cong{}^2G_2(3)'\cong\PSL(2,8)$ cannot occur:
its Sylow $3$-subgroups are cyclic of order $9$, while
$\Aut(L)/L\cong C_3$ and
$\Aut(L)\cong{}^2G_2(3)$; see \cite{AtlasL28}.
Thus $L\le\bar G\le\Aut(L)$ would force either
$\bar G=L$, contrary to the noncyclicity of $S$, or
$\bar G\cong{}^2G_2(3)$, excluded by hypothesis.

For each remaining socle in Theorem~\ref{thm:GMPprelim}, take
$U=\bar S\cap L\in\Syl_p(L)$ and write $u=|U|$. The standard Sylow-normalizer
and order data recalled in \eqref{eq:rankonedata} give
\[
 |N_L(U)|=u b_L,\qquad |L|=u(u+1)b_L,
 \qquad |\Syl_p(L)|=u+1,
\]
where $u=q$ in the linear case and $u=q^3$ in the unitary and
Ree cases; see \cite[Section~5]{GKmany} and
\cite[Section~3]{GMQZ}. These are calculations inside $L$;
no identification of the overgroup $\bar G$ is being made.
By \eqref{eq:socle-sylow-bijection}, $n=u+1$.
On the other hand, $U$ has order dividing $s$,
whereas the free action of $S$ on $\Omega\setminus\{P\}$ gives
$s\mid n-1=u$. Hence $u=s$, so $\bar S\le L$, $p\nmid\iota$,
and $n=s+1$.

It follows that $S$ acts regularly on $\Omega\setminus\{P\}$.
Since $G_P$ is transitive on that set, orbit--stabilizer gives
$|D|=|G_P|/s=m$. Also $D\cap S=1$, and the injection
$D\hookrightarrow G_P/S$ shows that $D$ is cyclic. Thus
$G_P=S\rtimes D$. Comparing orders now gives
\[
 nsm=|G|=r|\bar G|=r\iota|L|=r\iota\,s(s+1)b_L,
\]
and therefore $m=r\iota b_L$.

In the linear case $q\ge p^2\ge9$, so $b_L\ge4$.
For the unitary groups $b_L=8$ at $q=3,5$, and
$b_L\ge(q^2-1)/3\ge16$ for $q\ge7$.
For the Ree groups $b_L=q-1\ge26$. This proves $m\ge4$.
If $m\le8$, the linear case has $q\le17$; the only odd
prime power $q=p^f$ with $f\ge2$ in this range is $9$.
Then $m=4r\iota$ gives the first two rows of
\eqref{eq:quadratic-small-data}. The unitary case gives only
$q=3,5$, and $m=8r\iota\le8$ forces $r=\iota=1$.
The Ree case is impossible. Finally, in either unitary row
$S$ is isomorphic to a Sylow $p$-subgroup of $\PSU(3,q)$.
In the upper-unitriangular description of \cite[Section~2.2]{GMP},
commutators have central coordinate $b^q c-c^q b$. Taking $b=1$
shows that these coordinates range over the trace-zero subspace
of $\F_{q^2}$, which is the center. Thus $[S,S]=Z(S)$ has
order $q$.
\end{proof}

The lemma uses the almost-simple reduction in GMP and the Sylow
data of its possible socles, but does not exclude outer
automorphism extensions. The exact identification of $G/M$ is
deferred to Proposition~\ref{prop:outer} in
Section~\ref{sec:noncyclic}; it is not an input to the proof of
Theorem~\ref{thm:quadratic-reduction}.

\begin{proof}[Proof of Theorem~\ref{thm:quadratic-reduction}]
\emph{The group-theoretic input.}
Retain the notation $s,n,m,x,h$ and $\mathcal D_P$ from the
preceding subsection. The assumptions of
Lemma~\ref{lem:quadratic-data} are exactly the structural
assumptions of this theorem. It gives
\begin{equation}\label{eq:quadratic-n}
 n=s+1,\qquad s\ge9,
\end{equation}
and
\begin{equation}\label{eq:quadratic-m}
 m\ge4.
\end{equation}
We shall use the small-value list
\eqref{eq:quadratic-small-data} only in the final rationality step.

\emph{The short-orbit configuration.}
The hypothesis $|G|>84(g-1)$ gives $\cX/G\cong\PP^1$ and leaves the
usual short-orbit possibilities. By Lemma~\ref{lem:sylow}, all wild
points belong to $\Omega$, so two wild short orbits are impossible.

Suppose first that $\Omega$ is the only short orbit. The Hurwitz
formula for $G$ gives
\[
 x=n(\mathcal D_P-sm-1),
\]
while Hurwitz for the $S$-cover gives
\[
 x=2sh-s-1+\mathcal D_P.
\]
Using \eqref{eq:quadratic-n} and eliminating $\mathcal D_P$ yields
\begin{equation}\label{eq:quadratic-oneorbit}
 \frac{x}{n}=m+2h-1=:a.
\end{equation}
By \eqref{eq:quadratic-m}, $a\ge m-1\ge3$. For $m\ge4$,
\[
 \frac{m}{(m-1)^2}\le\frac49,
\]
with equality only for $m=4$. Therefore
\[
 |G|=n(n-1)m<n^2m
 \le\frac49 n^2a^2
 =\frac49 x^2
 =\frac{16}{9}(g-1)^2<2g^2,
\]
a contradiction. In particular, the orbit reduction itself does
not require the quadratic constant $2$.

Suppose next that there are two tame short orbits, with stabilizer
orders $e_1,e_2\ge2$. If one of them is at least $3$, then, as
in the proof of Theorem~\ref{thm:two},
\[
 \frac{x}{|G|}=1+\frac{\mathcal D_P-1}{sm}
     -\frac1{e_1}-\frac1{e_2}>\frac16.
\]
Hence $|G|<6x=12(g-1)<84(g-1)$, a contradiction.
Thus both tame stabilizers have order $2$. Their total tame different
contribution is $|G|$, so
\[
 x=n(\mathcal D_P-1).
\]
Combining this with Hurwitz for $S$ and \eqref{eq:quadratic-n} gives
\[
 \frac{x}{n}=2h-1=:a,\qquad \mathcal D_P=a+1.
\]
Since $\mathcal D_P\ge s-1$, we have $a\ge s-2=n-3$. A noncyclic
$p$-group in odd characteristic has $s\ge p^2\ge9$, hence $a\ge7$.
By Lemma~\ref{lem:local}, $m\le\mathcal D_P=a+1$. Consequently
\[
 |G|=n(n-1)m\le n(n-1)(a+1)<n^2(a+1).
\]
Since $4(a+1)<a^2$ for $a\ge7$,
\[
 |G|<\frac{n^2a^2}{4}<g^2,
\]
again impossible. We have proved that there is exactly one wild and
one tame short orbit.

\emph{Rationality of the Sylow quotient.}
It remains to show that $\cY=\cX/S$ is rational. Let
$D=G_{P,P'}$, where $P'\in\Omega\setminus\{P\}$. By
Lemma~\ref{lem:quadratic-data}(ii), $D$ is a cyclic
complement to $S$ in $G_P$ and has order $m$. It normalizes
$S$ and hence acts on $\cY$. The kernel of
$G_P\to\Aut(\cY)$ is exactly $S$, since
$K(\cX)/K(\cX)^S$ is a Galois extension with group $S$.
As $D\cap S=1$, the induced action of $D$ on $\cY$ is faithful.
Moreover, $\{P\}$ and $\Omega\setminus\{P\}$ are two distinct
$S$-orbits: their images in $\cY$ are distinct fixed points of
$D$. Assume that $h=g(\cY)>0$.

Since $\mathcal D_P\ge s-1$, equation~\eqref{eq:hurwitzS} gives
$2g-2\ge2sh-2$, hence $g\ge sh$.
If $g(\cY/D)\ge1$, tame Hurwitz at the two fixed points gives
$2h-2\ge2(m-1)$, so $h\ge m$. Hence
\[
 \frac{|G|}{g^2}
 \le\frac{nsm}{s^2h^2}
 \le\frac{n}{sh}
 \le\frac{1+1/s}{m}<2,
\]
contrary to $|G|>2g^2$.

Assume therefore that $\cY/D\cong\PP^1$. If the two fixed points
were the only ramified points of this cyclic cover, tame Hurwitz
would give $2h-2=-2m+2(m-1)=-2$, contrary to $h>0$.
There is therefore another short orbit, whose contribution is at
least $m/2$. Consequently
\[
 2h-2\ge-2m+2(m-1)+m/2=m/2-2,
\]
so $h\ge m/4$. Thus
\begin{equation}\label{eq:quadratic-rational-bound}
 \frac{|G|}{g^2}
 \le \frac{16n}{sm}
 =\frac{16}{m}\left(1+\frac1s\right).
\end{equation}
If $m\ge9$, recall independently that $s\ge p^2\ge9$, so
the right-hand side is at most
$160/81<2$, a contradiction.

It remains to inspect $m\le8$. Lemma~\ref{lem:quadratic-data}(iii)
restricts the possibilities to the four numerical pairs
\[
 (s,m)\in\{(9,4),(9,8),(27,8),(125,8)\}.
\]
The first two pairs come from the linear socle $\PSL(2,9)$.
The last two come from $\PSU(3,q)$ with $q=3,5$ and satisfy
$[S,S]=Z(S)$ of order $q$. We do not need to distinguish
index-two extensions of the linear socle.
For the linear cases $s=9$, $n=10$. If $m=4$, then
$h\ge1$, hence $g\ge9$ and
$|G|=360<84(9-1)\le84(g-1)$. If $m=8$, then
$h\ge2$, hence $g\ge18$ and
$|G|=720<84(18-1)\le84(g-1)$.
Both contradict the order hypothesis.

In the unitary rows one has $m=8$, $s=q^3$, and
$[S,S]=Z(S)$ of order $q$. By \eqref{eq:equal-commutator},
$[S,S]\le S_3\le S_2$. Therefore
\[
 \mathcal D_P\ge(s-1)+2(q-1),\qquad g\ge sh+q-1.
\]
Since $h\ge2$, for $q=3$ we obtain $g\ge56$ and
\[
 |G|=28\cdot27\cdot8=6048<2\cdot56^2,
\]
while for $q=5$ we obtain $g\ge254$ and
\[
 |G|=126\cdot125\cdot8=126000<2\cdot254^2.
\]
Both contradict $|G|>2g^2$. Therefore $h=0$, as required.

Finally, if $g\ge41$, then
\[
 2g^2-84(g-1)=2\bigl((g-41)(g-1)+1\bigr)>0.
\]
Thus in this genus range the condition $|G|>2g^2$ already implies
the other order inequality in the theorem.
\end{proof}

The two reductions above provide complementary starting points:
Theorem~\ref{thm:two} applies under the uniform bound needed for the
main classification, while Theorem~\ref{thm:quadratic-reduction}
gives a sharper result for noncyclic Sylow subgroups outside the
small Ree case. Its proof uses only the numerical socle data of
Lemma~\ref{lem:quadratic-data}. We now identify the possible
almost simple groups and continue the noncyclic classification under
the hypotheses of the main theorem.

\section{The non-cyclic Sylow case}\label{sec:noncyclic}

Throughout this section, assume
\begin{equation}\label{eq:mainhyp}
 p>2,\qquad \gamma(\cX)=0,\qquad g\ge4,\qquad
 |G|>24g(g-1),
\end{equation}
that $G$ has no common fixed point, and that
$S\in\Syl_p(G)$ is noncyclic. By Section~\ref{sec:two},
$G$ has one wild short orbit $\Omega$ and one tame short orbit
$\Delta$, and $\cX/G$ is rational.

We first identify the action on $\Omega$ and derive a genus formula.
We then study the central $p'$-kernel, including the small Ree case,
and determine the possible tame inertia orders. These reductions
prepare the geometric reconstruction; compare
\cite[Section~5]{GKzero} for the orbit-and-genus argument.

\subsection{The Sylow action and the genus formula}
\label{subsec:sylow-genus}

\begin{proposition}[Reduction to the standard actions]
\label{prop:outer}
Assume $\gamma(\cX)=0$, that $G$ has no common fixed point,
and that $S\in\Syl_p(G)$ is noncyclic. Put
$M=Z(N_G(S))$; by Lemma~\ref{lem:sylow} and
Theorem~\ref{thm:GMPprelim}, $M=O_{p'}(G)$. If
$G/M\not\cong{}^2G_2(3)$, then the induced faithful group on
the Sylow-fixed-point orbit $\Omega$ is
\[
 \PSL(2,q),\quad\PGL(2,q),\quad
 \PSU(3,q),\quad\PGU(3,q),\quad\Ree(q)\ (q\ge27)
\]
in its standard rank-one action. In particular, with
$s=|S|$, $n=|\Omega|$, $m=|G_P|/s$ and $r=|M|$, we have
\[
 n=s+1,\qquad m=r\bar m,
\]
where $\bar m$ is the corresponding two-point stabilizer order
in~\eqref{eq:rankonedata}.
\end{proposition}

\begin{proof}
Put $\bar G=G/M$ and $L=\operatorname{soc}(\bar G)$.
Lemma~\ref{lem:quadratic-data} identifies $\bar G$ with the
faithful group on $\Omega$ and gives $n=s+1$. By the
$\bar G$-equivariant bijection $\Omega\cong\Syl_p(L)$
\eqref{eq:socle-sylow-bijection} in the proof of that lemma, the
socle has its natural linear, unitary or Ree action, and
$\bar G_{P,P'}$ is cyclic for distinct $P,P'\in\Omega$.
Write $T=L_{P,P'}$; its order is the corresponding entry
for $L$ in \eqref{eq:rankonedata}.

For the linear and unitary groups, we use the semilinear descriptions
and two-point stabilizer calculation in
\cite[Section~5, especially Remark~5.2]{GKmany}.
For the nonsmall Ree groups, every automorphism is a product of
inner and field automorphisms; see \cite[Section~3, p.~226]{Bleher}.
Choose the standard pair $P,P'$, fixed by the field automorphisms.
If $\bar G$ contains a nontrivial field component, double
transitivity of $L$ allows us to multiply its representative
by an element of $L$ so that it fixes both points, without
changing that component. In the semilinear description this
representative is $d\phi^j$, where $\phi$ raises field
coordinates to their $p$th powers and $d$ belongs to the
projective two-point torus (to $T$ in the Ree case).
Thus $d$ centralizes $T$, and conjugation on $T$ is
$t\mapsto t^{p^j}$.

This action is nontrivial. For $L=\PSL(2,q)$, $q=p^f$,
we have $1\le j<f$ and
\[
 0<p^j-1\le q/p-1<(q-1)/2=|T|.
\]
For $L=\PSU(3,q)$, $q=p^f$, we have $1\le j<2f$ and
\[
 0<p^j-1\le q^2/p-1<(q^2-1)/\mu=|T|,
 \qquad \mu=\gcd(3,q+1)<p.
\]
The field components here include the graph involution.
For $L=\Ree(q)$, write $q=3^f$, where $f\ge3$ is odd.
Then $1\le j<f$ and $0<3^j-1<q-1=|T|$.
In every case $|T|\nmid p^j-1$, contradicting the cyclicity
of $\bar G_{P,P'}$. This also covers $q=9$, where the
nontrivial field component inverts $T\cong C_4$.

Only diagonal extensions remain. Their quotient orders are
$2$ in the linear case and $\mu\in\{1,3\}$ in the unitary
case, and there is none in the Ree case. Hence $\bar G$ is
one of the five groups in the statement. Finally,
$M\le G_{P,P'}$, and the image of this stabilizer in
$\bar G$ has order $\bar m$ from \eqref{eq:rankonedata}.
Therefore $m=r\bar m$.
\end{proof}

For the remaining candidate
$\bar G=G/M\cong{}^2G_2(3)\cong\operatorname{P\Gamma L}(2,8)$,
apply Theorem~\ref{thm:GMPprelim} directly, rather than
Lemma~\ref{lem:quadratic-data}. Put $I=N_G(S)=G_P$.
Then $M=Z(I)=O_{p'}(G)$, and $G$ is doubly transitive on
$\Omega$. The inverse image of $\bar S=SM/M$ is
$SM=S\times M$, whose Sylow $p$-subgroup is unique. Hence
$I/M=N_{\bar G}(\bar S)$.
Here $p=3$ and $s=27$. Choose a conjugate $J\cong9{:}6$ of
the maximal subgroup in \cite{AtlasL28} whose normal Sylow
$3$-subgroup is $\bar S$. Then
$J\le N_{\bar G}(\bar S)<\bar G$, since $\bar S$ is not normal
in $\bar G$. Maximality gives $N_{\bar G}(\bar S)=J$, of order
$54$. Thus $n=[\bar G:N_{\bar G}(\bar S)]=1512/54=28$.
Together with Proposition~\ref{prop:outer}, this gives in all cases
\begin{equation}\label{eq:n=s+1}
 n=|\Omega|=s+1.
\end{equation}
By Lemma~\ref{lem:sylow}, $S$ therefore acts regularly on
$\Omega\setminus\{P\}$. For $P'\ne P$, the group
$D=G_{P,P'}$ has order $m$, intersects $S$ trivially and
embeds in the cyclic group $G_P/S$. Consequently
$G_P=S\rtimes D$ with $D$ cyclic, also in the small Ree case.
We retain that case until the central-kernel calculation.

\begin{lemma}\label{lem:pointbound}
One has
\[
 |G_P|>4g.
\]
\end{lemma}

\begin{proof}
If $|G_P|\le4g$, then $s\le4g$ and
$n=s+1\le4g+1$.  Hence
\[
 |G|=n|G_P|\le4g(4g+1)<24g(g-1),
\]
for $g\ge4$, a contradiction.
\end{proof}

\Needspace{9\baselineskip}
The next lemma is the odd-characteristic counterpart of
\cite[Lemma~5.3]{GKzero}. Outside the small Ree case it follows
from Theorem~\ref{thm:quadratic-reduction}; the following direct
argument covers that case as well.

\begin{lemma}\label{lem:rational}
The quotient $\cX/S$ is rational.
\end{lemma}

\begin{proof}
Put $\cY=\cX/S$ and let $g_1=g(\cY)$.  Since $S$ fixes only $P$,
Hurwitz gives
\begin{equation}\label{eq:gsg1}
 g\ge sg_1.
\end{equation}
The cyclic complement $D=G_{P,P'}$ constructed above acts
faithfully on $\cY$, since the kernel of the action of $G_P$
is $S$. It fixes the two distinct images of $P$ and
$\Omega\setminus\{P\}$.

Assume $g_1>0$.  If $g(\cY/D)\ge1$, tame Hurwitz at the two fixed points
gives $g_1\ge|D|$, and therefore
\[
 g\ge s|D|=|G_P|>4g,
\]
a contradiction.  If $\cY/D$ is rational, the positive genus of $\cY$
forces another short $D$-orbit.  Since $D$ is cyclic, its length is at
most $|D|/2$, and tame Hurwitz gives
\[
 g_1\ge\frac{|D|}{4}.
\]
Thus
\[
 g\ge\frac{s|D|}{4}
 =\frac{|G_P|}{4}>g,
\]
again a contradiction.  Hence $g_1=0$.
\end{proof}

The following intersection argument and the genus identity that follows
should be compared with \cite[Lemma~5.4 and Eq.~(16)]{GKzero} and
\cite{Montanucci}. We give
the proofs in the present notation because they feed into the central
kernel and quotient-genus calculation.

\begin{lemma}\label{lem:intersection}
Let $R\in\Delta$.  Then
\[
 G_P\cap G_R=1.
\]
Consequently $G_R$ is cyclic and acts semiregularly on $\Omega$.
\end{lemma}

\begin{proof}
Since $\Delta$ is tame, any element of $G_R$ has order prime to $p$.
Suppose $1\ne\alpha\in G_P\cap G_R$.  Inside the solvable group
$G_P=S\rtimes D$, a $p'$-element lies in a conjugate of the cyclic
complement.  Hence $\alpha$ fixes, besides $P$, another point of
$\Omega\setminus\{P\}$.  It also fixes $R$.

The element $\alpha$ normalizes $S$, hence induces an automorphism of
$\cX/S\cong\PP^1$.  It fixes the images of the three distinct
$S$-orbits
\[
 \{P\},\qquad \Omega\setminus\{P\},\qquad S(R).
\]
Therefore the induced automorphism of $\PP^1$ is trivial.  The kernel of
the action of $G_P$ on $\cX/S$ is $S$, so $\alpha\in S$.  Since
$\alpha$ has order prime to $p$, this is impossible.

Thus $G_P\cap G_R=1$.  A tame point stabilizer is cyclic, and the
intersection statement shows that no nontrivial element of $G_R$ fixes a
point of $\Omega$.
\end{proof}

\begin{proposition}[The genus formula]\label{prop:master}
Let $R\in\Delta$ and put $e=|G_R|$. Then
\begin{equation}\label{eq:master}
  2g-2=n\left(\frac{m}{e}-1\right).
\end{equation}
\end{proposition}

\begin{proof}
By Corollary~\ref{cor:relations},
\[
 \frac{x}{n}=\mathcal D_P-1-\frac{sm}{e},
 \qquad x=2g-2.
\]
Since $\cX/S$ is rational, \eqref{eq:hurwitzS} gives
$\mathcal D_P=x+s+1$.  Therefore
\[
 x\left(1-\frac1n\right)
   =s\left(\frac{m}{e}-1\right).
\]
Now $n=s+1$, so $x=n(m/e-1)$, as claimed.
\end{proof}

\subsection{The central kernel}\label{sec:kernel}

\paragraph{Cyclic covers and divisor notation.}
For a cyclic cover of degree $r$ prime to $p$, Kummer theory gives
$K(\cX)=K(\cY)(z)$ with $z^r=f\in K(\cY)$.
At a point $Q\in\cY$, the ramification index is
$r/\gcd(r,v_Q(f))$ (with $\gcd(r,0)=r$).
For a divisor $\mathcal D$ on a smooth curve $Y$, we use the notation
\[
 L(\mathcal D)=\{0\}\cup\{f\in K(Y)^\times:\operatorname{div}(f)+\mathcal D\ge0\},
 \qquad \ell(\mathcal D)=\dim_K L(\mathcal D).
\]
The Weierstrass semigroup at $P$ is
\[
 H(P)=\{n\ge0:\text{there is a function with pole divisor }nP\}.
\]
Both Riemann--Roch spaces and their dimensions commute with
extension of the constant field by flat base change; see
\cite[Tag~02KH]{Stacks}.

We retain the notation $M$, $r$ and $m=r\bar m$ of
Proposition~\ref{prop:outer}; in all cases, including the small Ree
case, $\bar m$ is the two-point stabilizer order in $G/M$.

Compare the cyclic kernel in \cite[Theorem~6.2]{GKzero}.
Here we also establish centrality in $G$, the quotient-genus
formula, the ramification description and $r\mid n$,
including the small Ree candidate.

\begin{theorem}[Central kernel and quotient genus]\label{thm:kernel}
The group $M$ is cyclic, is the kernel of the action of $G$ on $\Omega$,
and is contained in $Z(G)$.  If
\[
 \widehat g=g(\cX/M),
\]
then
\begin{align}
 2g-2
 &=n\left(\frac{r\bar m}{e}-1\right),\label{eq:masterkernel}\\
 2\widehat g-2
 &=n\left(\frac{\bar m}{e}-1\right),\label{eq:descent}\\
 2g-2
 &=r(2\widehat g-2)+n(r-1).\label{eq:tamecover}
\end{align}
If $r>1$, the cover $\cX\to\cX/M$ is totally ramified at every point of
$\Omega$ and unramified elsewhere.  Moreover,
\begin{equation}\label{eq:rdivn}
 r\mid n.
\end{equation}
\end{theorem}

\begin{proof}
Since $M\le Z(G_P)$ is a $p'$-group and $G_P/S$ is cyclic, $M$ is cyclic.
Normality shows that $M$ fixes every point of $\Omega$.

Let $\mathcal K$ be the kernel of the action on $\Omega$.
No nontrivial $p$-element fixes two points of $\Omega$, so
$\mathcal K$ is a normal $p'$-subgroup of $G$.
Hence $\mathcal K\le O_{p'}(G)=M$, and the reverse inclusion
gives $\mathcal K=M$.

The action on $\Omega$ is doubly transitive, so $G_P$ is maximal in $G$.
The normal subgroup $C_G(M)$ contains $G_P$.  It cannot equal $G_P$,
because a point stabilizer in a nontrivial doubly transitive action is not
normal.  Hence $C_G(M)=G$, and $M\le Z(G)$.

By Lemma~\ref{lem:intersection}, $G_R\cap M=1$.  Therefore $M$ acts
freely outside $\Omega$, while every point of $\Omega$ is totally
ramified in $\cX\to\cX/M$.  Tame Hurwitz gives
\eqref{eq:tamecover}.  Equation \eqref{eq:masterkernel} is
Proposition~\ref{prop:master}.  Combining it with
\eqref{eq:tamecover} gives \eqref{eq:descent}.

Finally, assume $r>1$.  By Kummer theory write
\[
 K(\cX)=K(\cX/M)(z),\qquad z^r=f,
\]
and let a generator $\sigma$ of $M$ satisfy $\sigma(z)=\zeta z$,
where $\zeta$ is a primitive $r$-th root of unity.  Let $P_i\in\Omega$,
$B_i$ its image in $\cX/M$, and $a_i=v_{B_i}(f)$.  Total ramification
gives $\gcd(a_i,r)=1$.  If $\sigma$ acts on a local parameter at $P_i$
with tangent multiplier $\lambda_i$, comparing leading terms of $z$
gives $\lambda_i^{a_i}=\zeta$.  Since $M\le Z(G)$ and $G$ is transitive
on $\Omega$, all $\lambda_i$ coincide.  Each has order $r$, so
$a_i\equiv a_1\pmod r$ for every $i$.  At unramified points the
valuation of $f$ is divisible by $r$.  Taking the degree of
$\operatorname{div}(f)$ gives $na_1\equiv0\pmod r$; since
$\gcd(a_1,r)=1$, we obtain $r\mid n$.
\end{proof}

\begin{remark}
The quotient-genus formula \eqref{eq:descent} is independent of $r$.
Thus enlarging the central cyclic kernel cannot repair a negative value of
$\widehat g$.  This observation eliminates some numerical branches which
would survive if one considered only the genus formula for $\cX$.
\end{remark}

\begin{remark}[The small Ree group]\label{rem:smallree}
For $\bar G\cong{}^2G_2(3)$ we have $n=28$, $s=27$,
and $\bar m=2$. All involutions lie in $\PSL(2,8)$,
since this subgroup has index $3$, and form one conjugacy
class; see \cite{AtlasL28}. The Sylow normalizer of order
$54$ contains an involution, so every involution fixes a
point of $\Omega$.
By Lemma~\ref{lem:intersection}, $G_R\cap M=1$ and the
image of $G_R$ in $\bar G$ acts semiregularly on $\Omega$.
It therefore has odd order $e$. Also $e\mid28$ and $e>1$,
so $e=7$. Equation~\eqref{eq:descent} now gives
$2\widehat g-2=28(2/7-1)=-20$, a contradiction.
\end{remark}

\Needspace{8\baselineskip}
\subsection{The possible tame inertia orders}\label{sec:rankone}

\subsubsection*{Linear groups.}

Assume that the induced group is $\PSL(2,q)$, with $q$ odd.  Then
\[
 n=q+1,\qquad s=q,\qquad
 \bar m=\frac{q-1}{2}.
\]
A cyclic subgroup acting semiregularly on the natural orbit has order
dividing $(q+1)/2$ \cite{GKsemi}.  If
\[
 e=\frac{q+1}{2t},
\]
then \eqref{eq:masterkernel} gives
\begin{equation}\label{eq:PSL}
 2g=(q-1)(rt-1).
\end{equation}
For $\PGL(2,q)$ one has $\bar m=q-1$ and $e\mid q+1$;
writing $e=(q+1)/t$ gives the same formula, now with $t\mid q+1$.
For $\PSL(2,q)$ the condition is $t\mid(q+1)/2$.

\subsubsection*{Unitary groups.}

Let $G/M$ be either $\PSU(3,q)$ or $\PGU(3,q)$ in its natural
action, with $q\ge3$ odd. Put $\mu=(3,q+1)$ and let $\delta=\mu$
or $1$, respectively. Then
\[
 n=q^3+1,\qquad s=q^3,\qquad \bar m=(q^2-1)/\delta.
\]
For $\PSU(3,q)$ the semiregular cyclic orders are given by
\cite[Proposition~3.2]{GKsemi}. For $\PGU(3,q)$, the semisimple
element types in \cite[Lemma~2.2]{MZ} give the following alternatives for a
cyclic group without fixed points on the natural unital: its
order divides $q+1$, or it lies in a Singer torus of order
$q^2-q+1$. The other semisimple types fix isotropic points.
Since an involution fixes isotropic points, the first order is
odd and hence divides $(q+1)/2$. Thus in both cases
\[
 e\mid(q+1)/2\qquad\text{or}\qquad
 e\mid(q^2-q+1)/\delta.
\]

\begin{proposition}\label{prop:unitary}
Under the large-group hypothesis \eqref{eq:mainhyp}, the branch
\[
 e\mid\frac{q+1}{2}
\]
is impossible.  Hence only
\[
 e\mid\frac{q^2-q+1}{\delta}
\]
can occur.
\end{proposition}

\begin{proof}
For $q=3$ the branch is empty, since $e>1$ is odd and
would divide $2$. Thus $q\ge5$. Write
\[
 e=\frac{q+1}{2t}.
\]
Equation \eqref{eq:masterkernel} gives
\[
 2g
 =
 (q-1)\left(
 \frac{2rt(q^3+1)}{\delta}-(q^2+q+1)
 \right).
\]
Since $q\ge5$ and $\delta\le3$, we have
\[
 \frac{q^3+1}{\delta}\ge q^2+q+1.
\]
Hence
\[
 g\ge
 \frac{r(q-1)(q^3+1)}{2\delta}.
\]
Since
\[
 |G|
 =
 \frac{r(q^3+1)q^3(q^2-1)}{\delta},
\]
we obtain
\[
 \frac{|G|}{g^2}
 \le
 \frac{4\delta q^3(q+1)}
 {r(q-1)(q^3+1)}
 <\frac{4\delta(q+1)}{r(q-1)}
 \le18.
\]
For $g\ge4$,
\[
 18g^2\le24g(g-1),
\]
contradicting \eqref{eq:mainhyp}.
\end{proof}

For the surviving Singer-divisor branch, writing
\[
 e=\frac{q^2-q+1}{\delta t}
\]
gives the necessary genus relation
\begin{equation}\label{eq:PSUsinger}
 2g
 =
 (q-1)\bigl(rt(q+1)^2-(q^2+q+1)\bigr).
\end{equation}
The corresponding equations are determined in Section~\ref{sec:unitarycomplete}.

\subsubsection*{Ree groups.}

Assume
\[
 G/M\cong{}^2G_2(q),
 \qquad
 q=3q_0^2,\qquad
 q_0=3^{f_0},\ f_0\ge1.
\]
Then
\[
 n=q^3+1,\qquad
 s=q^3,\qquad
 \bar m=q-1.
\]
By \cite[Proposition~5.2]{GKsemi}, a cyclic semiregular subgroup has
order dividing
\[
 \frac{q+1}{2},
 \qquad
 q-3q_0+1,
 \qquad
 q+3q_0+1.
\]
Every involution fixes points of the natural orbit, so the order $e$ is
odd.  Since $q\equiv3\pmod8$, the first possibility improves to
\[
 e\mid\frac{q+1}{4}.
\]

\begin{proposition}\label{prop:ree}
The value
\[
 e=q+3q_0+1
\]
is impossible.
\end{proposition}

\begin{proof}
Put
\[
 e_+=q+3q_0+1,\qquad e_-=q-3q_0+1.
\]
Since
\[
 e_+e_-=q^2-q+1
\]
and $n=(q+1)(q^2-q+1)$, the descent formula
\eqref{eq:descent}, with $e=e_+$, gives
\[
 2\widehat g-2
 =
 -(q+1)e_-(3q_0+2)<-2,
\]
which is impossible.
\end{proof}

\section{Cyclic Sylow subgroups and the structural alternatives}\label{sec:cyclic}
Retain \eqref{eq:mainhyp} and the absence of a global fixed point.
The two-orbit theorem applies. We use the normalized form of
\eqref{eq:general1}:
\begin{equation}\label{eq:twoRH}
 \frac{x}{|G|}=\frac{\mathcal D_P-1}{sm}-\frac1e.
\end{equation}

\subsection{Higher cyclic Sylow subgroups}

\begin{proposition}\label{prop:highcyclic}
Under \eqref{eq:mainhyp}, a cyclic Sylow $p$-subgroup of $G$ has order $p$.
\end{proposition}
\begin{proof}
Suppose $S=C_{p^a}$, $a\ge2$, and let $j_1<j_2$ be its first two
lower jumps. The upper jumps of this equal-characteristic cyclic $p$-extension
satisfy $u_2\ge p u_1$; see
\cite[Theorem~1.1(c), with tame part trivial]{OP}.
Converting to lower numbering gives
\[
 j_2=j_1+p(u_2-u_1)\ge(p^2-p+1)j_1.
\]
The lower groups for $S$ are obtained by intersection with those for
$I$; in particular the first graded quotient has order $p$.
Thus $m\le j_1(p-1)$ and
\begin{align*}
 \mathcal D_P&\ge j_1(s-1)+(j_2-j_1)(s/p-1)\\
  &\ge j_1(ps-p^2+p-1).
\end{align*}
Consequently
\[
 \mathcal D_P-sm\ge j_1(s-p^2+p-1)\ge p-1\ge2.
\]
Equation~\eqref{eq:twoRH} now gives
\[
 \frac{x}{|G|}>1-\frac1e\ge\frac12,
\]
so $|G|<4(g-1)$, a contradiction.
\end{proof}

\subsection{Prime-order Sylow data}
Assume $S=C_p$. Its sole lower jump $j$ is positive and prime to $p$.
Put $J=j(p-1)$. Equations~\eqref{eq:hurwitzS}--\eqref{eq:twoRH} become
\begin{equation}\label{eq:prime}
 g=ph+\frac{(p-1)(j-1)}2,\qquad m\mid J,\qquad
 \frac{|G|}{x}=\frac{pm e}{e(J-1)-pm}.
\end{equation}
The denominator is positive. The large-group hypothesis is exactly
$|G|/x>12g$.

\subsection{Positive quotient genus: a finite reduction}
\begin{fact}[Cyclic tame actions and elliptic rank]\label{fact:cyclic-elliptic}
A tame cyclic group fixing a point on a curve of genus $h>0$ has
order at most $4h+2$; see \cite{StichtenothI,StichtenothII} and
\cite[Theorem~11.60]{HKT}.
This includes $h=1$. In characteristics $5$ and $7$, an elliptic
curve with a point-fixing automorphism of order $4$ or $6$ has
$j$-invariant $1728$ or $0$, respectively.
For an elliptic curve $E:y^2=f(x)$ with $f$ a squarefree cubic
over $K$ and $p>2$, put
\[
 a_p=[x^{p-1}]f(x)^{(p-1)/2}.
\]
Then $E$ is ordinary, equivalently $\gamma(E)=1$, if and only if
$a_p\ne0$; it is supersingular, equivalently $\gamma(E)=0$, if
and only if $a_p=0$. This is the Hasse-invariant criterion
\cite[Chapter~V, Theorem~4.1]{Silverman}. It also follows directly
from the Cartier operator on $dx/y$, whose scalar is $a_p^{1/p}$.

\end{fact}

\begin{lemma}\label{lem:positiveh}
If $h\ge1$, the only possible numerical data are
\[
 (p,h,j,m,e,g,|G|)=(5,2,1,4,7,10,2520).
\]
\end{lemma}
\begin{proof}
The complement $C_m$ acts faithfully on $\cX/S$ and fixes the image of
$P$. Fact~\ref{fact:cyclic-elliptic} gives
$m\le4h+2$, including the elliptic case; we use this bound throughout.

First suppose $j\ge2$. Write $J=vm$ and $a=ve-p$. Positivity in
\eqref{eq:prime} implies $ma-e>0$, and hence $a\ge1$. If $J>p+1$,
the function of $a$ below is decreasing, so
\begin{equation}\label{eq:Rbound}
 \frac{|G|}{x}
 =\frac{pm(p+a)}{a(J-1)-p}
 \le\frac{pm(p+1)}{J-p-1}.
\end{equation}
For $p\ge7$, $j\ge2$, the last quantity is at most
$2pm\le2p(4h+2)<12g$.
For $p=5$, $j\ge3$, it is at most $pm\le p(4h+2)<12g$;
the same estimate holds for $p=3$, $j\ge4$
(the value $j=3$ is inadmissible).
For $p=5,j=2$, we have $m\mid8$ and the bound is $15m$.
If $m\le4$, it is at most $60<12g$; if $m=8$, then $h\ge2$
and it is at most $120<12g$. Finally, if $p=3,j=2$, then
$m\mid4$, positivity gives $e\ge m+1$, and the decreasing function
in \eqref{eq:prime} is at most $m(m+1)\le20<12g$.
Therefore $j=1$.

Now $g=ph$ and $m\mid p-1$. The expression in
\eqref{eq:prime} decreases with $e$, so use the least integer
$e>pm/(p-2)$ to obtain an upper bound, even if that integer is not
prime to $p$. There are three possibilities for $m$.
\begin{enumerate}[label=\textup{(\roman*)},leftmargin=*]
\item If $m<(p-1)/2$, then $m\le(p-1)/3$.
For $m\ge2$ the least integer is $m+1$, and
\[
 \frac{|G|/x}{g}
 \le\frac{m(m+1)}{h(p-2-2m)}
 \le\frac{m+2+2/(m-1)}h\le10.
\]
The last inequality uses $m\le4h+2$ and
$m+2+2/(m-1)\le4h+6\le10h$.
For $m=1$, $p\ge5$, the bound is $2/[h(p-4)]\le2$.
\item If $m=(p-1)/2$ and $p\ge5$, the least integer is $(p+3)/2$,
and
\[
 \frac{|G|/x}{g}\le
 \frac{p+5+12/(p-3)}{4h}\le2+\frac4h\le6,
\]
using $m=(p-1)/2\le4h+2$, hence $p\le8h+5$.
For $p=3,m=1$, the bound is $4/h$.
\item If $m=p-1$, $p\ge5$, the least integer is $p+2$, and
\[
 \frac{|G|/x}{g}\le
 \frac{p+5+18/(p-4)}h.
\]
Here $m=p-1\le4h+2$ gives $p\le4h+3$.
For $p\ge7,h\ge2$ this is at most
$4+14/h\le11$. For $h=1$, only $p=5,7$ remain. For $p=5$
the bound is $28/h$, leaving only $h=1,2$.
For $p=3,m=2$ the bound is $14/h$, and $h=1$ would give
$g=3$, outside our assumptions.
\end{enumerate}
It follows that only $(p,h,m)=(5,1,4),(5,2,4),(7,1,6)$ remain.
For $p=5,m=4$, increasing $e$ from $7$ to $8$ already gives
$|G|/x\le40<12g$; for $p=7,m=6$, increasing $e$ from $9$ to
$10$ gives $|G|/x\le105/2<12g$. Thus the respective orders are
$e=7,7,9$.

The two elliptic cases are impossible. We use the Hasse-invariant
criterion in Fact~\ref{fact:cyclic-elliptic}. Deuring--Shafarevich gives
$\gamma(\cX/S)=0$. An elliptic curve in characteristic $5$ with a
point-fixing automorphism of order $4$ has model $v^2=u^3-u$;
its Hasse invariant is the nonzero coefficient $-2$ of $u^4$ in
$(u^3-u)^2$. Similarly in characteristic $7$, order $6$ gives
$v^2=u^3-1$, with nonzero coefficient $-3$ of $u^6$ in
$(u^3-1)^3$. Both are ordinary, a contradiction.
The remaining case gives $g=10$, $x=18$ and
$|G|=18\cdot20\cdot7/(21-20)=2520$.
\end{proof}

\subsection{Identifying the exceptional group}
\begin{lemma}\label{lem:A7}
In the remaining case of Lemma~\ref{lem:positiveh}, $G\cong A_7$.
\end{lemma}
\begin{proof}
At the wild point, $I=C_5\rtimes C_4$ with faithful conjugation:
the lower jump is $1$, so its tame tangent character acts faithfully
on the first ramification quotient. Hence $S\le G'$ and $Z(I)=1$.
Every central element of $G$ preserves the unique fixed point of $S$,
so $Z(G)\le Z(I)=1$.

The abelian quotient $G/G'$ is prime to $5$. The associated quotient
of the $G$-cover of $\PP^1$ is tame and branched at most at its two
branch points. By Fact~\ref{fact:tame-two}, such a nontrivial cover is cyclic
and totally ramified at both branch points.
Its degree would divide both $4$ and $7$. Thus $G=G'$.

Choose a maximal proper normal subgroup $N$ of $G$. Since $G$ is
perfect, $G/N$ is nonabelian simple. The low-order list of
nonabelian simple groups leaves
\[
 A_5,\quad \PSL(2,7),\quad A_6,\quad \PSL(2,8),\quad A_7
\]
as the possibilities of order dividing $2520$; their orders are
recorded in \cite{Atlas}. In the first four cases, $|N|$ is
respectively $42,15,7,5$. Each such $N$ has a characteristic subgroup
$V$ of prime order: take its unique Sylow $7$-subgroup when $|N|=42$,
its unique Sylow $5$-subgroup when $|N|=15$, and $V=N$ otherwise.
Then $V\triangleleft G$. Since $G$ is perfect and $\Aut(V)$ is
abelian, conjugation on $V$ is trivial. Thus $1\ne V\le Z(G)$,
a contradiction. Hence $G/N\cong A_7$, and comparison of orders
gives $N=1$ and $G\cong A_7$.
\end{proof}

Orbit--stabilizer now gives the orbit lengths in case~\ref{case:A7}.
The corresponding Hermitian example is identified in
Proposition~\ref{prop:A7rigidity}.

\subsection{Rational quotient: the explicit family}
\label{subsec:cyclic-rational}
\begin{fact}[The one-point characterization]\label{fact:one-point}
Let $\cX$ be the smooth projective model of $y^p-y=f(x)$, where
$f\in K[x]$ has degree prime to $p$, and assume $g\ge2$. The result
of Stichtenoth \cite{StichtenothII}, in the convenient form
\cite[Theorem~3.1 and Corollary~3.4]{LM}, says that an automorphism
of $\cX$ can move the unique point over infinity only in the
following cases, up to $K$-isomorphism:
\begin{equation}\label{eq:Stichtenothprelim}
 y^p-y=x^d,\quad d<p,\ d\mid p+1;
 \qquad y^p-y=x^{p+1}.
\end{equation}
For the first family the stabilizer at infinity has order
$pd(p-1)$ and its orbit has $p+1$ points.
\end{fact}

If $h=0$, the cover $\cX\to \cX/S$ is an Artin--Schreier cover branched
only at infinity, so it has an equation $y^p-y=f(x)$ with
$\deg f=j$, $(j,p)=1$. Since $G$ moves the point over infinity,
the one-point characterization \eqref{eq:Stichtenothprelim} gives, up to isomorphism,
\[
 f(x)=x^d,\quad 2\le d<p,\ d\mid p+1,
 \qquad\text{or}\qquad f(x)=x^{p+1}.
\]
The second possibility must be excluded numerically, since $S$ is a
Sylow subgroup of $G$, not necessarily of $\Aut(\cX)$.
Here $J=p^2-1$ and $g=p(p-1)/2$. The derivation of
\eqref{eq:Rbound} does not use $h>0$; together with $m\le J$, it gives
\[
 \frac{|G|}{x}\le\frac{p(p^2-1)}{p-2}
 <6p(p-1)=12g,
\]
a contradiction. Thus the first family is forced.

By \cite[Corollary~3.4]{LM}, the stabilizer at infinity consists of
\[
 (x,y)\longmapsto(ax,a^d y+b),\qquad
 a^d\in\F_p^\times,\quad b\in\F_p.
\]
It has order $pd(p-1)$, and \cite[Theorem~3.1(a)]{LM} gives an
orbit of length $p+1$. Hence $|A_d|=dp(p^2-1)$.
Together with the lift
$(x,y)\mapsto(x/y^{(p+1)/d},-1/y)$, these transformations induce
$\PGL(2,p)$ on the $y$-line and have central kernel
$\{x\mapsto\zeta x:\zeta^d=1\}$.
Their group has order $|A_d|$, proving the asserted exact sequence.
Dickson's subgroup classification shows that a subgroup of
$\PGL(2,p)$ containing a $p$-element either fixes its unique point
in $\PP^1(\F_p)$ or contains $\PSL(2,p)$. The former would make
$G$ fix the unique point over that branch point. Hence the image
of $G$ is $\PSL(2,p)$ or $\PGL(2,p)$, proving
case~\ref{case:AS}.

\subsection{The exceptional alternating action}\label{sec:A7rigidity}

\begin{proposition}\label{prop:A7rigidity}
Let $\cX$ and $G$ be as in Lemmas~\ref{lem:positiveh} and~\ref{lem:A7}.
Then $\cX$ is isomorphic over $K$ to the Hermitian curve $H_5$ of
degree $6$.
\end{proposition}
\begin{proof}
Take $U=A_6\le G=A_7$, a point stabilizer in the natural degree-$7$
action. Every subgroup $5{:}4$ is the normalizer of its Sylow
$5$-subgroup, so all such subgroups are conjugate.
In particular, $I=N_G(S)$ is generated, up to conjugacy, by
$(1\,2\,3\,4\,5)$ and $(2\,3\,5\,4)(6\,7)$, so its orbits on
$G/U$ have lengths $5$ and $2$. By Fact~\ref{fact:intermediate}, the corresponding $U$-orbits
above the wild branch point have stabilizers $C_4$ and $5{:}2$,
and lengths $90$ and $36$.
The first contributes $90(4-1)=270$ to the different of
$\cX\to\cX/U$.  In the second, lower ramification is $5{:}2,C_5,1$,
so its contribution is $36((10-1)+(5-1))=468$.
The tame inertia $C_7$ acts as a $7$-cycle on $G/U$ and
intersects $U$ trivially, so it contributes zero to this different. Thus
\[
 18=360\bigl(2g(\cX/U)-2\bigr)+738,
\]
and $\cX/U\cong\PP^1$.

The resulting degree-$7$ map
$\cX/U\longrightarrow\cX/G$ has fibers of multiplicities $5+2$
and $7$ at its two branch points.  Place the latter point at infinity
on both lines, and the two zeros at $0,1$ on the source.  Up to a
nonzero scalar on the target, the map is necessarily
\begin{equation}\label{eq:A7map}
 u=z^5(z-1)^2.
\end{equation}
In characteristic $5$ one has $u'=2z^5(z-1)$, so the different
exponents at $z=0$ and $z=1$ are $5$ and $1$; at infinity the index is
$7$ and the exponent is $6$. Thus $-2=7\cdot(-2)+5+1+6$, and there is
no further branch point.
The core of $A_6$ in $A_7$ is trivial; hence $K(\cX)$ is the Galois
closure of $K(z)/K(u)$, uniquely determined up to $K(u)$-isomorphism.

The Hermitian curve $H_5$, of genus $10$ and zero $5$-rank, admits
$G_0\cong A_7$ by \cite[Remark~6.15]{GKmany}; see also the maximal
subgroups of $\PSU(3,5)$ in \cite{AtlasU35}. This action has no
common fixed point, since point stabilizers are solvable. Take
$S_0\in\Syl_5(G_0)$, put $h_0=g(H_5/S_0)$, and let $j_0$ be its
lower jump. The faithful conjugation action of
$N_{G_0}(S_0)\cong5{:}4$ on $S_0$ and
\eqref{eq:gradedprelim} give $\gcd(4,j_0)=1$. Hurwitz gives
$10=5h_0+2(j_0-1)$, whence $(h_0,j_0)=(2,1)$.
Since $2520>24\cdot10\cdot9$, Theorem~\ref{thm:two} applies.
Now $n_0=2520/20=126$, and \eqref{eq:general1} gives
$18/126=3-20/e_0$, where $e_0$ is the tame inertia order. Thus
$e_0=7$. Choose $U_0\cong A_6$ as a natural point stabilizer in
$G_0\cong A_7$. The conjugacy of the $5{:}4$ normalizers and the
$7$-cycle action give the same fibers and the same degree-$7$ map.
Its Galois closure is $K(H_5)$, so
$\cX\cong H_5$.
\end{proof}

\subsection{The structural alternatives}

\begin{theorem}[Structural reduction]\label{thm:main}
Let $\cX$ be a smooth projective irreducible curve over an
algebraically closed field $K$ of characteristic $p>2$. Suppose
\[
 g=g(\cX)\ge4,\qquad \gamma(\cX)=0,\qquad
 |G|>24g(g-1),\qquad G\le\Aut(\cX)
\]
and $G$ has no common fixed point. Let $S\in\Syl_p(G)$.
Then $\cX/G\cong\PP^1$ and there are exactly two short $G$-orbits,
one wild, $\Omega$, and one tame, $\Delta$. Exactly one of the
following three alternatives holds.
\begin{enumerate}[label=\textup{(\Alph*)},leftmargin=*]
\item\label{case:AS} $S\cong C_p$ and $\cX/S$ is rational. Up to
$K$-isomorphism,
\[
 \cX:\quad y^p-y=x^d,\qquad 2\le d<p,\quad d\mid p+1,
 \qquad g=\frac{(p-1)(d-1)}2.
\]
Its full automorphism group $A_d$ has a central exact sequence
\[
 1\longrightarrow C_d\longrightarrow A_d
 \longrightarrow\PGL(2,p)\longrightarrow1,
 \qquad |A_d|=d\,p(p^2-1).
\]
The image of $G$ is $\PSL(2,p)$ or $\PGL(2,p)$. The size
inequality must still be imposed on $G$.
\item\label{case:A7} $p=5$, $g=10$, $G\cong A_7$, and $S\cong C_5$.
Here
\[
 g(\cX/S)=2,\quad |\Omega|=126,\quad |\Delta|=360,
 \quad G_P\cong C_5\rtimes C_4,\quad G_R\cong C_7,
\]
for $P\in\Omega$, $R\in\Delta$, and the lower jump of $S$ is $1$.
By Proposition~\ref{prop:A7rigidity}, $\cX$ is the Hermitian curve
$H_5$ of degree $6$.
\item\label{case:noncyc} $S$ is noncyclic. The kernel $M$ of $G$ on
$\Omega$ is a central cyclic $p'$-subgroup. Put $r=|M|$ and $\bar G=G/M$.
The possibilities and necessary tame stabilizer orders $e=|G_R|>1$
are given in Table~\ref{tab:main}. In this table $n=|\Omega|=|S|+1$
and $b=|\bar G|/[n(n-1)]$, so $b=\bar m$ in the notation of
Section~\ref{sec:kernel}.
\end{enumerate}
In case \ref{case:noncyc}, $\cX/S$ is rational, $G_R$ is semiregular
on $\Omega$, and
\begin{equation}\label{eq:classification-master}
 \quad |G|=r n(n-1)b,\qquad
 2g-2=n\left(\frac{rb}{e}-1\right),\qquad r\mid n.\quad
\end{equation}
The additional quotient-genus condition is
\begin{equation}\label{eq:qgenus}
 2g(\cX/M)-2=n\left(\frac b e-1\right)\in\{-2,0,2,4,\ldots\}.
\end{equation}
In particular no cyclic Sylow subgroup $C_{p^a}$ with $a\ge2$ occurs.
\end{theorem}

\begin{table}[ht]
\centering\small
\renewcommand{\arraystretch}{1.35}
\begin{tabular}{@{}llll@{}}
\toprule
$\bar G$ & $n$ & $b$ & Necessary condition on $e>1$\\
\midrule
$\PSL(2,q)$ & $q+1$ & $(q-1)/2$ & $e\mid(q+1)/2$\\
$\PGL(2,q)$ & $q+1$ & $q-1$ & $e\mid q+1$\\
$\PSU(3,q)$ & $q^3+1$ & $(q^2-1)/\mu$ & $e\mid(q^2-q+1)/\mu$\\
$\PGU(3,q)$ & $q^3+1$ & $q^2-1$ & $e\mid q^2-q+1$\\
$\Ree(q)$ & $q^3+1$ & $q-1$ &
 \begin{tabular}[t]{@{}l@{}}
 $e\mid(q+1)/4$ or $e\mid e_-$,\\
 or $e\mid e_+$ with $e<e_+$
 \end{tabular}\\
\bottomrule
\end{tabular}
\caption{In the linear rows $q=p^f$, $f\ge2$; in the unitary rows
$q=p^f\ge3$ and $\mu=(3,q+1)$. Coincident rows are identified.
For Ree, $p=3$, $q=3q_0^2\ge27$, $q_0=3^{f_0}$, and
$e_\pm=q\pm3q_0+1$. The equality $e=e_+$ is excluded by
\eqref{eq:qgenus}. The branch $e\mid e_+$, $e<e_+$, and the
$(q+1)/4$ branch are eliminated later in
Proposition~\ref{prop:reemodel}.}
\label{tab:main}
\end{table}

For the linear rows write $\nu=2,1$, respectively, and
$e=(q+1)/(\nu t)$. For the unitary rows write $\delta=\mu,1$,
respectively, and $e=(q^2-q+1)/(\delta t)$. The genus formulas become
\begin{align}
 2g&=(q-1)(rt-1) &&\text{(linear)},\label{eq:linear}\\
 2g&=(q-1)\bigl(rt(q+1)^2-(q^2+q+1)\bigr)
 &&\text{(unitary)}.\label{eq:unitary}
\end{align}
For the Ree row, equation~\eqref{eq:classification-master} is the uniform genus formula.
For trivial central kernel, the linear and unitary formulas have the
same formal shape as the corresponding genus formulas in
\cite[Theorem~1.2 and Section~5]{GKzero}; we do not claim these numerical
identities as new. The new task pursued in
Sections~\ref{sec:rigidity}--\ref{sec:fullgroups} is to determine the
actual odd-characteristic curves, the central lifts, their full
automorphism groups, and the precise large subgroups that occur.
The equations and the sufficient realization conditions are determined
in Sections~\ref{sec:equations} and~\ref{sec:fullgroups}.

\begin{proof}[Proof of Theorem~\ref{thm:main}]
Theorem~\ref{thm:two} gives rationality of $\cX/G$ and the two short
orbits. If $S$ is cyclic, Proposition~\ref{prop:highcyclic} gives
$S\cong C_p$. Section~\ref{subsec:cyclic-rational} gives
case~\ref{case:AS} when $\cX/S$ is rational. Otherwise,
Lemmas~\ref{lem:positiveh} and~\ref{lem:A7}, together with
Proposition~\ref{prop:A7rigidity}, give case~\ref{case:A7}.
For a noncyclic Sylow subgroup, Theorem~\ref{thm:GMPprelim} and Proposition
\ref{prop:outer} give the groups in Table~\ref{tab:main}; the small
Ree case is excluded in Remark~\ref{rem:smallree}. Lemmas
\ref{lem:rational} and~\ref{lem:intersection}, followed by
Theorem~\ref{thm:kernel}, give rationality, semiregularity,
centrality, the genus formulas and $r\mid n$. The semiregular
orders and the exclusions in Section~\ref{sec:rankone} complete
the table. The three alternatives are disjoint by the Sylow
structure and the genus of the Sylow quotient.
\end{proof}

\section{Equivariant one-point covers}\label{sec:rigidity}

We determine how a tame action restricts one-point covers, first
for elementary-abelian groups and then for central extensions.
All covers are over $K$, and ramification numbers refer to infinity
on the rational base. We use the Swan-conductor convention of
Section~\ref{subsec:quotientprelim} and the reduced-polynomial
notation of Section~\ref{subsec:coversprelim}.

\subsection{Marked covers and equivariance}\label{subsec:markedprelim}
Fix the geometric point $\bar 0$ of $\mathbb A^1_x$ at $0$ and put
$\Gamma=\pi_1^{\rm et}(\mathbb A^1_x,\bar 0)$.
A \emph{marked $P$-cover} means a connected finite \'etale Galois
cover of this affine line, an identification of its deck group with
the finite group $P$, and a chosen geometric point above $0$.
Isomorphisms preserve the base, group identification and chosen
point. Such covers correspond to continuous epimorphisms
$\rho:\Gamma\to P$; the point removes the ambiguity of inner
conjugation. For $E\triangleleft P$, the marked quotient corresponds
to $\pi\circ\rho$, where $\pi:P\to P/E$. Thus two marked
quotients are isomorphic exactly when these homomorphisms agree.
See \cite[Expos\'e~V]{SGA1} and
\cite[Tags~0BND and~03SF]{Stacks}.

Conjugations are taken in the automorphism groups of function fields.
For a lift $\widetilde\tau$ of $x\mapsto\zeta x$ fixing the marked
point, we use the inverse-conjugation action
\[
 \theta(\tau)(s)=\widetilde\tau^{-1}s\widetilde\tau
 \qquad(s\in P).
\]
Since $T=\langle\tau\rangle$ is cyclic, this defines a homomorphism
$T\to\Aut(P)$. Let $\tau_*$ be the pointed map on $\Gamma$ induced
by the base morphism whose pullback is $x\mapsto\zeta x$.
In a compatible separable closure it is computed by inverse
conjugation by a lift of the field substitution. The marked
monodromy homomorphism therefore satisfies
\[
 \rho(\tau_*\sigma)=\theta(\tau)(\rho(\sigma)).
\]
On Artin--Schreier classes, precomposition by $\tau_*$ is pullback:
$[f(x)]\mapsto[f(\zeta x)]$. If $E\le Z(P)$ and two marked
homomorphisms have the same quotient modulo $E$, their difference
$\chi(\sigma)=\rho_2(\sigma)\rho_1(\sigma)^{-1}$ is a homomorphism
to $E$. If $E=(\F_Q,+)$ and
$\theta(\tau)(\sigma_c)=\sigma_{\beta c}$, then
$[f(\zeta x)]=[\beta f(x)]$ for its torsor class.
Thus $\beta$ is the inverse-conjugation multiplier; the multiplier
for $\widetilde\tau\sigma_c\widetilde\tau^{-1}$ is $\beta^{-1}$.
Lemma~\ref{lem:twists} gives the corresponding polynomial and
conductor restrictions. A pointed isomorphism of connected Galois
covers, if it exists, is unique and is $T$-equivariant whenever the
actions preserve the markings.

\subsection{Equivariant reconstruction}

\begin{lemma}[Equivariant elementary-abelian covers]\label{lem:ASrigid}
Let $Q=p^a$, and let $V\cong(\F_Q,+)$ act on a curve $Y$ with
$Y/V=\PP^1_x$, with just one branch point, at infinity, and with
just one positive lower jump $D$.
Suppose a tame cyclic group $T=\langle\tau\rangle$ of order $m$
normalizes $V$, fixes a point above $x=0$, and acts by
$\tau(x)=\zeta x$, where $\zeta$ has order $m$.
Suppose its conjugation action on $V$ is irreducible over $\F_p$,
with kernel of order $r$. Put $b=m/r$. Then
\[
 D=rt,\qquad (t,b)=1,\qquad p\nmid D.
\]
If there are no integers $h,k$ satisfying
\begin{equation}\label{eq:ASobstruction}
 0<h<t,\qquad p\nmid h,\qquad 0\le k<a,\qquad
 b\mid tp^k-h,
\end{equation}
then, after scaling $x$ and marking the additive group,
\[
 K(Y)=K(x,y),\qquad y^Q-y=x^D.
\]
The specified $T$-fixed point can be taken to be $(0,0)$, and then
$\tau(y)=\zeta^D y$.
\end{lemma}

\begin{proof}
Put $E=K(Y)$ and $F_0=K(x)$, and let $P_0$ be the specified
$T$-fixed point above $0$. We use automorphisms of function fields,
so that $\tau(f(x))=f(\zeta x)$.
The inertia group at infinity is $V$: since $V$ is abelian, a
proper inertia group would give a nontrivial unramified quotient
cover of $\PP^1$, contrary to Hurwitz. Thus the sole upper jump is
also $D$. By compatibility of upper numbering with quotients,
every nontrivial degree-$p$ subextension of $E/F_0$ has upper jump
$D$; see \cite[Chapter~IV]{SerreLocal}.

Write $\wp_p(h)=h^p-h$. The cover is connected and finite \'etale
over $\mathbb A^1_x$. Its monodromy epimorphism induces an injection
\[
 V^\vee=\operatorname{Hom}_{\F_p}(V,\F_p)
 \hookrightarrow H^1_{\rm et}(\mathbb A^1_x,\F_p)
 =K[x]/\wp_p(K[x]);
\]
see Section~\ref{subsec:coversprelim} and
\cite[Tags~0A3J and~03P2]{Stacks}.
Let $W$ consist of the $p$-reduced representatives with zero
constant term of this image. Thus every character has a polynomial
representative and $\dim_{\F_p}W=a$. For $f\in W$, choose
$z_f\in E$ with $z_f^p-z_f=f$; this is possible because the class of
$f$ is a character of $\Gamma$ factoring through $V$, so the
Artin--Schreier extension it defines is the fixed field in $E$ of the
kernel of that character. The pairing
\[
 \langle\sigma,f\rangle=\sigma(z_f)-z_f\in\F_p
 \qquad(\sigma\in V)
\]
is independent of the root chosen and is $\F_p$-bilinear.
If it vanishes for every $\sigma$, then $z_f\in F_0$. Being integral
over $K[x]$, it lies in $K[x]$, so the class of $f$ is zero. Hence $W\to V^\vee$ is injective, and equality of
dimensions makes it an isomorphism. Every nonzero member of $W$
has degree $D$, by the degree-$p$ ramification formula, so $p\nmid D$.

Since $\tau$ normalizes $V$, substitution $f(x)\mapsto f(\zeta x)$
preserves $W$. Reducedness and the zero constant term are preserved
as well. As $\tau(z_f)$ is a root for $f(\zeta x)$, and
$(\tau^{-1}\sigma\tau)(z_f)-z_f
=\tau^{-1}\bigl(\sigma(\tau z_f)-\tau z_f\bigr)$, where the bracket
lies in $\F_p$ and is fixed by $\tau^{-1}$, the pairing gives
\[
 \langle\sigma,f(\zeta x)\rangle
 =\langle\tau^{-1}\sigma\tau,f(x)\rangle.
\]
Thus $W$ is the dual of the conjugation module $V$ and has the same
kernel, of order $r$. It is irreducible: the annihilator of a proper
nonzero submodule of $V^\vee$ would be a proper nonzero submodule of $V$.

The leading-coefficient map
\[
 W\longrightarrow K,\qquad f\longmapsto[x^D]f
\]
is injective, since every nonzero $f\in W$ has degree $D$.
Let $C$ be its image and put $\alpha=\zeta^D$.
The action on $C$ is multiplication by $\alpha$, whose order is
therefore $b=m/r$. Consequently
\[
 \frac{m}{\gcd(m,D)}=b=\frac mr,
\]
so $\gcd(m,D)=r$. Writing $D=rt$ gives $\gcd(t,b)=1$.

The finite field $F=\F_p(\alpha)$ acts on $C$. If $\dim_F C>1$,
an $F$-line would be a proper nonzero $T$-invariant
$\F_p$-subspace, contrary to irreducibility. Hence $\dim_F C=1$.
Since $\dim_{\F_p}C=a$, we obtain $F=\F_Q$ and $C=c\F_Q$ for
some $c\in K^\times$. Taking $\eta x$ as the new coordinate $x$,
where $\eta^D=c$, makes $C=\F_Q$ and preserves $\tau(x)=\zeta x$.

For each $\lambda\in\F_Q$, there is now a unique polynomial
\[
 f_\lambda=\lambda x^D+
 \sum_{\substack{0<i<D\\p\nmid i}}c_i(\lambda)x^i
 \quad\text{in }W,
\]
where $c_i:\F_Q\to K$ is $\F_p$-linear. Each such map has a unique
expression
\[
 c_i(\lambda)=\sum_{k=0}^{a-1}c_{ik}\lambda^{p^k},
 \qquad c_{ik}\in K.
\]
To justify this, $\operatorname{Hom}_{\F_p}(\F_Q,K)$ has dimension
$a$ over $K$. The $a$ maps $\lambda\mapsto\lambda^{p^k}$ are
$K$-linearly independent: a relation would be a polynomial of
degree at most $p^{a-1}<Q$ vanishing at all $Q$ elements of $\F_Q$.
They therefore form a basis.

The leading coefficient of $f_\lambda(\zeta x)$ is
$\alpha\lambda$, so uniqueness gives
$f_\lambda(\zeta x)=f_{\alpha\lambda}(x)$. Comparing the coefficient
of $x^i$ yields
\[
 \sum_{k=0}^{a-1}c_{ik}
       (\alpha^{p^k}-\zeta^i)\lambda^{p^k}=0
 \qquad(\lambda\in\F_Q).
\]
By the same independence, a nonzero $c_{ik}$ would imply
\[
 \zeta^{Dp^k}=\zeta^i,\qquad m\mid Dp^k-i.
\]
Since $r\mid m,D$, this forces $i=rh$. Now $0<h<t$ and
$p\nmid h$, since $p\nmid i$, and division by $r$ gives
$b\mid tp^k-h$. This contradicts \eqref{eq:ASobstruction}.
Hence
\[
 W=\{\lambda x^D:\lambda\in\F_Q\}.
\]

It remains to recover $E/F_0$ and the marked action. In a common
separable closure of $F_0$, choose $y$ with $y^Q-y=x^D$ and put
$E_0=F_0(y)$. At a place of $E_0$ over infinity, with ramification
index $e_0$ and normalized valuation $v$, one has
$Qv(y)=-e_0D$. Since $p\nmid D$, it follows that $Q\mid e_0$,
and thus $[E_0:F_0]=Q$. Its Galois group is the full translation
group $\{y\mapsto y+c:c\in\F_Q\}$.
For $\lambda\in\F_Q$, put
\[
 z_\lambda=\sum_{j=0}^{a-1}(\lambda y)^{p^j},
 \qquad
 z_\lambda^p-z_\lambda=\lambda x^D.
\]
The classes $[\lambda x^D]$ form an $a$-dimensional space, since
a nonconstant $p$-reduced polynomial cannot be an Artin--Schreier
coboundary. Thus $E$ and $E_0$ have the same degree-$p$
subextensions. Each is their compositum: the intersection of the
kernels of all characters of an elementary-abelian group is
trivial. Consequently $E=E_0$ in the chosen separable closure.
This gives the required equation and identifies $V$ with the
translation group.

Finally, $y$ is integral over $K[x]$ and is regular at $P_0$.
Its equation gives $y(P_0)\in\F_Q$. Replacing $y$ by $y-y(P_0)$
preserves the equation and gives $y(P_0)=0$. Since
$\alpha=\zeta^D\in\F_Q$, we then have
\[
 (\tau(y)-\alpha y)^Q-(\tau(y)-\alpha y)=0.
\]
Thus $\tau(y)=\alpha y+c_0$ for some $c_0\in\F_Q$.
Now $\tau$ fixes $P_0$, so evaluation gives $c_0=0$.
Hence $\tau(y)=\zeta^D y$, as asserted.
With the function-field convention used above, if
$\sigma_c(y)=y+c$, then
$\tau\sigma_c\tau^{-1}=\sigma_{\alpha^{-1}c}$, in agreement
with the dual action on $W$.
\end{proof}

\begin{lemma}[Central twists with a conductor bound]\label{lem:twists}
Let $P$ be a finite $p$-group and $E\le Z(P)$ an elementary-abelian
subgroup identified with $(\F_Q,+)$, where $Q=p^a$; write
$\sigma_c\in E$ for the element corresponding to $c\in\F_Q$.
Consider two marked $P$-covers $Y_i\to\mathbb A^1_x$, $i=1,2$.
Suppose $T=\langle\tau\rangle=C_m$, with $p\nmid m$, acts on each
cover, fixes its marked point above $0$, and induces the same base
substitution $\tau_i(x)=\zeta x$. Under the group markings, assume
that the inverse-conjugation actions
\[
 \theta_i(\tau)(s)=\tau_i^{-1}s\tau_i\qquad(s\in P)
\]
coincide, preserve $E$, and restrict to
$\theta_i(\tau)(\sigma_c)=\sigma_{\beta c}$ for a common
$\beta\in\F_Q^\times$. Suppose their marked $P/E$-quotients are
isomorphic over $\mathbb A^1_x$, and both covers have largest upper
jump at infinity at most $U$.
Then their difference is represented by a $Q$-reduced polynomial
$f(x)$ with zero constant term satisfying
\begin{equation}\label{eq:twistweight}
 f(\zeta x)=\beta f(x).
\end{equation}
For each nonzero monomial $c_i x^i$ of $f$, write $i=p^k h$ with
$p\nmid h$. Then
\begin{equation}\label{eq:twistconductor}
 0\le k<a,\qquad h\le U.
\end{equation}
If no monomial satisfies these weight and conductor conditions
(the weight condition being $\zeta^i=\beta$),
the two marked $P$-covers are isomorphic, compatibly with $T$.
\end{lemma}

\begin{proof}
Put $\Gamma=\pi_1^{\rm et}(\mathbb A^1_x,\bar0)$ and use the
function-field and monodromy convention of
Section~\ref{subsec:markedprelim}. The marked covers give epimorphisms
$\rho_i:\Gamma\to P$. Their quotient markings agree, so, for
$\pi:P\to P/E$, one has $\pi\rho_1=\pi\rho_2$. Consequently
\[
 \chi(\sigma)=\rho_2(\sigma)\rho_1(\sigma)^{-1}\in E.
\]
Centrality gives, for $\sigma,\eta\in\Gamma$,
\[
 \begin{aligned}
 \chi(\sigma\eta)
 &=\rho_2(\sigma)\rho_2(\eta)
                 \rho_1(\eta)^{-1}\rho_1(\sigma)^{-1}\\
 &=\rho_2(\sigma)\chi(\eta)\rho_1(\sigma)^{-1}
  =\chi(\sigma)\chi(\eta).
 \end{aligned}
\]
Thus $\chi$ is a continuous homomorphism to $E$, not necessarily
surjective; its torsor need not be connected.

Write $\theta=\theta_1=\theta_2$. In the common pointed realization
used to define $\tau_*$, choose the lift $\widehat\tau$ of
$x\mapsto\zeta x$ fixing the point above $0$. Its restriction to
$K(Y_i)$ and $\tau_i$ lift the same base substitution and fix the
same marked point. Their quotient is a deck transformation fixing
a point of an \'etale fiber, hence is the identity; see
\cite[Tag~03SF]{Stacks}. Restricting
$\widehat\tau^{-1}\sigma\widehat\tau$ to $K(Y_i)$ therefore gives
$\rho_i(\tau_*\sigma)=\theta(\tau)(\rho_i(\sigma))$ exactly, not just
up to inner conjugacy. Hence, writing $E$ additively,
\[
 \chi(\tau_*\sigma)=\beta\chi(\sigma).
\]
Pullback on Artin--Schreier classes is substitution by $\zeta x$:
if $w^Q-w=f$ for a polynomial $f$, then
$\widehat\tau(w)^Q-\widehat\tau(w)=f(\zeta x)$ and
$(\widehat\tau^{-1}\sigma\widehat\tau)(w)-w
=\widehat\tau^{-1}\bigl(\sigma(\widehat\tau w)-\widehat\tau w\bigr)
=\sigma(\widehat\tau w)-\widehat\tau w$, since the bracket lies
in $\F_Q$.
The inverse conjugation in the definition of $\theta$ is essential
here: if the direct conjugation multiplier were $\eta$, the
polynomial multiplier would instead be $\eta^{-1}$.

Let $\wp_Q(h)=h^Q-h$. The Artin--Schreier sequence, with
$F^a-1$ in place of $F-1$, and the vanishing of
$H^1_{\rm et}(\mathbb A^1_x,\mathbf G_a)$ on the affine line give
\[
 H^1_{\rm et}(\mathbb A^1_x,\F_Q)=K[x]/\wp_Q(K[x]);
\]
see \cite[Tags~0A3J and~03P2]{Stacks}.
Each class has a unique $Q$-reduced representative with zero
constant term; see Section~\ref{subsec:coversprelim}.
Let $f$ represent $\chi$ in this form. The preceding equivariance
identity gives $[f(\zeta x)]=[\beta f(x)]$.
Since $\beta\in\F_Q$, multiplication by $\beta$ preserves
$\wp_Q(K[x])$. Both representatives are reduced with zero constant
term, so they are equal. This proves \eqref{eq:twistweight}.

We next transfer the ramification bound to $\chi$. Choose a single
local fundamental-group map
$\iota:G_{K((x^{-1}))}\to\Gamma$ at infinity, with compatible
base points. By compatibility of upper numbering with quotients,
\[
 (\rho_i\circ\iota)(G_{K((x^{-1}))}^{\,u})=1
 \qquad(u>U,\ i=1,2);
\]
see \cite[Chapter~IV]{SerreLocal}. The same equality holds for
$\chi\circ\iota$ and every degree-$p$ character obtained from it.
All these characters therefore have Swan conductor at most $U$.
This uses the common upper filtration, not lower numbering or
separate choices of local identifications.

As $f$ is $Q$-reduced, it has a unique expression
\[
 f(x)=\sum_{\substack{h>0\\p\nmid h}}
             \sum_{k=0}^{a-1}c_{h,k}x^{hp^k}.
\]
For $\lambda\in\F_Q$, the character
$c\mapsto\operatorname{Tr}_{\F_Q/\F_p}(\lambda c)$ gives the
Artin--Schreier class $[\lambda f]$ modulo $\wp_p$.
Indeed, if $w^Q-w=f$, then
$z_\lambda=\sum_{j=0}^{a-1}(\lambda w)^{p^j}$ satisfies
$z_\lambda^p-z_\lambda=\lambda f$ and transforms by this trace.
Since $c\,x^{hp^k}\equiv c^{1/p^k}x^h$ modulo $\wp_p(K[x])$, its
$p$-reduced representative is
\[
 \sum_{\substack{h>0\\p\nmid h}}A_h(\lambda)x^h,
 \qquad
 A_h(\lambda)=\sum_{k=0}^{a-1}(\lambda c_{h,k})^{1/p^k}.
\]
The roots are unique since $K$ is perfect. If some $c_{h,k}\ne0$,
then $A_h$ is not identically zero on $\F_Q$, because
\[
 A_h(\lambda)^{p^{a-1}}
 =\sum_{k=0}^{a-1}
  c_{h,k}^{p^{a-1-k}}\lambda^{p^{a-1-k}}
\]
is a nonzero polynomial of degree at most $p^{a-1}<Q$.
The exponents are distinct, so its nonzero terms cannot cancel.

If $f\ne0$, choose the largest $h=h_{\max}$ with some
$c_{h,k}\ne0$, and then choose $\lambda\in\F_Q$ with
$A_{h_{\max}}(\lambda)\ne0$. The resulting character has a
$p$-reduced polynomial of degree exactly $h_{\max}$, hence Swan
conductor $h_{\max}$. Thus $h_{\max}\le U$, which proves
\eqref{eq:twistconductor} for every monomial. In particular, one
bounds the prime-to-$p$ part $h$ of the exponent, not $\deg f$.

Finally, if no monomial is allowed, then $f=0$. Since $E$ is
abelian, $H^1(\Gamma,E)=\operatorname{Hom}_{\rm cont}(\Gamma,E)$,
so $\chi=0$ and $\rho_1=\rho_2$. The covers are isomorphic as
marked $P$-covers. The pointed isomorphism is unique, since a
deck transformation of a connected Galois cover fixes no point
in an \'etale fiber unless it is the identity. Conjugating this
isomorphism by the two $T$-actions preserves the group marking
and marked point, because $\theta_1=\theta_2$. Uniqueness therefore
makes the isomorphism $T$-equivariant.
\end{proof}

\begin{fact}[Frattini quotients and tame actions]\label{fact:frattini}
For a finite $p$-group $S$, its Frattini subgroup is the intersection
of its maximal proper subgroups, and the Burnside basis theorem gives
\begin{equation}\label{eq:Frattiniprelim}
 \Phi(S)=S^p[S,S],\qquad
 B\Phi(S)=S\ \Longrightarrow\ B=S\quad(B\le S),
\end{equation}
where $S^p$ is generated by the $p$th powers. Indeed, every maximal
subgroup has index $p$ and contains $S^p[S,S]$; conversely, the
hyperplanes of the elementary-abelian quotient $S/S^p[S,S]$
have trivial intersection. The generation assertion follows by
putting a proper $B$ inside a maximal subgroup.
Moreover,
\begin{equation}\label{eq:Frattinikernelprelim}
 B_0=\ker\bigl(\Aut(S)\longrightarrow\Aut(S/\Phi(S))\bigr)
 \quad\text{is a $p$-group}.
\end{equation}
To see this, fix a basis of $S/\Phi(S)$ of size $d$. Its ordered
lifts form a set of size $|\Phi(S)|^d$, all of whose members
generate $S$ by \eqref{eq:Frattiniprelim}. The group $B_0$ acts
freely on this set, so its order divides a power of $p$.

If two homomorphisms from a $p'$-group $T$ to $\Aut(S)$ induce
the same homomorphism on $S/\Phi(S)$, their images are complements
to $B_0$ in the inverse image of that common image. By
Fact~\ref{fact:SZ} they are conjugate by an element of $B_0$.
Projection is injective on each complement, so the conjugation
identifies the homomorphisms, not just their images.
\end{fact}

\begin{lemma}[Compatibility of the markings]\label{lem:markings}
Let $K$ be algebraically closed of characteristic $p>2$.
Let $S$ be a Sylow $p$-subgroup of $\PSU(3,q)$, with
$q=p^f\ge3$, or of $\Ree(q)$, with $p=3$ and
$q=3^{2f_0+1}\ge27$. With the identifications fixed in the proof, put
\[
 A=S/\Phi(S)\cong(F,+),\qquad
 F=\F_{q^2}\ \text{or}\ \F_q,
\]
respectively. For $i=1,2$, let $Y_i\to\mathbb A^1_x$ be a connected
finite \'etale $S$-cover with a chosen point $P_i$ above $0$.
Suppose $T=\langle\tau\rangle=C_m$, with $p\nmid m$, acts on both
covers, fixes $P_i$, and has function-field generators $\tau_i$
satisfying $\tau_i(x)=\zeta x$, where $\zeta^m=1$.
Under initial group markings, write
\[
 \theta_i(\tau)(s)=\tau_i^{-1}s\tau_i\qquad(s\in S).
\]
Assume that the induced action on $A$ is multiplication by
$\alpha_i\in F^\times$, with $\F_p(\alpha_i)=F$.
Suppose there is a pointed $T$-equivariant isomorphism
\[
 \Psi:Y_1/\Phi(S)\longrightarrow Y_2/\Phi(S)
\]
over $\mathbb A^1_x$. Then the full $S$-markings can be changed,
without changing $P_i$ or $\Psi$, so that $\Psi$ preserves the
induced $A$-markings and $\theta_1=\theta_2$.
\end{lemma}

\begin{proof}
For a unitary Sylow group, $\Phi(S)=S'=Z(S)$. For a Ree Sylow
group, $S'$ is elementary abelian and the cubes lie in
$Z(S)\le S'$, so $\Phi(S)=S'$; see
Section~\ref{subsec:standardcurves} and \cite[Proposition~49]{GMQZ}.
The isomorphism $\Psi$ induces an $\F_p$-linear isomorphism
$\psi:A\to A$ with $\psi(\alpha_1v)=\alpha_2\psi(v)$.
The scalars have the same minimal polynomial over $\F_p$, so
$\alpha_2=\alpha_1^{p^j}$ for some $j$. Since $\F_p(\alpha_1)=F$,
putting $c=\psi(1)\ne0$ gives
\[
 \psi(v)=c\,v^{p^j}\qquad(v\in F).
\]
Indeed, this follows first for $v$ polynomial in $\alpha_1$,
and these are all elements of $F$.

This particular isomorphism lifts to $S$; an arbitrary element
of $\Aut(A)$ need not. In the unitary presentation
\[
 S=\{(u,v)\in\F_{q^2}^2:v^q+v=u^{q+1}\},\qquad
 (u,v)(w,z)=(u+w,v+z+u^qw),
\]
the maps $(u,v)\mapsto(cu,c^{q+1}v)$ and coordinatewise
$p$-Frobenius are automorphisms. In the Ree case, take $S$ to be
the Sylow subgroup fixing $P_\infty$ on the model $R_q$ of
Section~\ref{subsec:standardcurves}, where $q_0=3^{f_0}$; it is the
normal Sylow subgroup of the stabilizer of $P_\infty$. For
$\lambda\in\F_q^\times$, the automorphism $d_\lambda$ multiplying
$v,y,z$ by $\lambda,\lambda^{q_0+1},\lambda^{2q_0+1}$ fixes
$P_\infty$, and these automorphisms act on $A=S/S'$ with weight $1$.
Hence $A$ is identified with $(\F_q,+)$, and conjugations by the
$d_\lambda$ give all scalings of $A$. Since $R_q$ and $P_\infty$ are
defined over $\F_3$, cubing the coordinates is an endomorphism
$\mathrm{Fr}$ of $R_q$ fixing $P_\infty$. For $g\in\Aut(R_q)$, let
$g^{(3)}$ be obtained by cubing the coefficients of $g$; then
$\mathrm{Fr}\circ g=g^{(3)}\circ\mathrm{Fr}$, and $g\mapsto g^{(3)}$
is an automorphism of $\Aut(R_q)$. It preserves the stabilizer of
$P_\infty$, hence $S$, and $d_\lambda^{(3)}=d_{\lambda^3}$. The
additive bijection $\phi$ of $A$ that it induces therefore satisfies
$\phi(\mu a)=\mu^3\phi(a)$ for $\mu\in\F_q^\times$, so
$\phi(a)=\phi(1)a^3$. Composing with a scaling lifts $a\mapsto a^3$.
These maps lift $\psi$ and are used
only to change the abstract group marking. Adjusting one full
marking by such a lift makes $\Psi$ preserve the induced
$A$-markings.

Since $\Psi$ is $T$-equivariant and now preserves the induced
$A$-markings, the homomorphisms $\theta_i$ induce the same
homomorphism $T\to\Aut(A)$. By the last paragraph of
Fact~\ref{fact:frattini}, they are conjugate by an element of
$B_0=\ker(\Aut(S)\to\Aut(A))$; that argument also allows a
nontrivial common kernel of the $T$-actions. Changing one full
marking by this element gives $\theta_1=\theta_2$. Since $B_0$ acts
trivially on $A$, the quotient markings are unchanged. Throughout, only the
abstract group markings have changed, not $P_i$ or $\Psi$.
\end{proof}

\section{Equations of the curves}\label{sec:equations}

The cyclic Sylow cases were settled in Section~\ref{sec:cyclic}.
We now determine the equations in the noncyclic alternatives of
Theorem~\ref{thm:main}, treating the linear, unitary and Ree groups in turn.
Lemma~\ref{lem:ASrigid} identifies the first additive quotient; in the
unitary and Ree cases, Lemmas~\ref{lem:twists} and~\ref{lem:markings}
then determine the central extensions. We retain \eqref{eq:mainhyp},
assume that $G$ has no common fixed point, and use the notation of
Theorem~\ref{thm:main}. The exponent $D$ in each one-point equation
will be determined from the ramification data.

\subsection{Linear groups}\label{sec:linearcomplete}

\begin{proposition}\label{prop:linearcomplete}
In a noncyclic linear row of Table~\ref{tab:main}, put $D=rt$.
Then
\[
 \cX:\quad y^q-y=x^D,\qquad 2\le D<q,\qquad D\mid q+1.
\]
\end{proposition}
\begin{proof}
Write $q=p^f$, $f\ge2$, and put $\nu=2$ for $\PSL(2,q)$ and
$\nu=1$ for $\PGL(2,q)$. The cyclic complement has order $m=rb$,
where $b=(q-1)/\nu$, and its conjugation action on $S=(\F_q,+)$
has kernel $M$. Its image is the scalar group of order $b$.
A generator is in no proper subfield, since
$b\ge(q-1)/2>\sqrt q-1$; thus this action is irreducible.
The lower ramification groups are invariant subspaces, so there is
just one positive jump $j$. Hurwitz for $\cX\to\cX/S\cong\PP^1$
gives $2g=(q-1)(j-1)$; comparison with \eqref{eq:linear} gives $j=D=rt$.

Choose $P'\in\Omega\setminus\{P_\infty\}$ and the complement
$T=G_{P_\infty,P'}$. On $\cX/S$, place the images of $P_\infty$
and $P'$ at $\infty$ and $0$. The action of $T$ on this line is
faithful, since its kernel lies in $T\cap S=1$. Hence a generator
satisfies $\tau(x)=\zeta x$ with $\zeta$ of order $m$, and $P'$ is
a $T$-fixed point above $0$. These are the geometric hypotheses of
Lemma~\ref{lem:ASrigid}. Here $t=(q+1)/(\nu e)$ and $e>1$.

There is no obstruction \eqref{eq:ASobstruction}.  For $k=0$,
$0<t-h<t<b$; the last inequality holds for $q\ge9$.
For $1\le k<f$, multiply the congruence by $\nu e$ and use
$\nu et=q+1$.  It gives
\[
 \nu eh\equiv2p^k\pmod{q-1}.
\]
The left side is strictly between $0$ and $q+1$; the right side is
between $2p$ and $2q/p<q-1$.  Adding $q-1$ to the right side would
exceed $q+1$.  Equality is the only possibility, but it is impossible
because $p\nmid\nu eh$ and $p\mid2p^k$.
Thus the curve has the stated monomial equation.

By the choice of $0$ and of the marked point, the original wild orbit is
\[
 \Omega=\{x=0,\ y\in\F_q\}\cup\{P_\infty\}.
\]
The group $C_D$, acting by $x\mapsto\xi x$ with $\xi^D=1$,
fixes it pointwise. Put $\widehat G=\langle G,C_D\rangle$, a finite
group since $g\ge4$. It preserves $\Omega$, has no common fixed point,
and satisfies $|\widehat G|\ge|G|>24g(g-1)$. A Sylow $p$-subgroup
$\widetilde S$ of $\widehat G$ containing $S$ fixes $P_\infty$ and acts freely
on the other $q$ points, by Lemma~\ref{lem:sylow}. Thus
$q\le|\widetilde S|\mid q$, so $\widetilde S=S$.
Its fixed-point orbit is $\Omega$, as $G$ is already transitive there.
Theorem~\ref{thm:kernel} therefore applies to $\widehat G$ and gives
$C_D\le M_{\widehat G}\le Z(\widehat G)$ and
$D\mid|M_{\widehat G}|\mid q+1$, where $M_{\widehat G}$
is the kernel on $\Omega$.

For $\mu\in M\le T$, the marked action from
Lemma~\ref{lem:ASrigid} has
$\mu(x)=\xi x$ and $\mu(y)=\xi^D y$.
Since $\mu$ centralizes every translation $y\mapsto y+c$,
$c\in\F_q$, we have $\xi^D=1$. Thus $M\le C_D$.
Conversely, $G\cap C_D$ fixes $\Omega$ pointwise, so
\[
 G\cap C_D=M.
\]

If $D=q+1$, then $g=q(q-1)/2$. Since $D=rt$ and $t\ge1$,
$r\le D$, and the structural order $|G|=rq(q^2-1)/\nu$ gives
\[
 |G|\le Dq(q^2-1)=q(q-1)(q+1)^2
       <24g(g-1)
\]
for $q\ge3$.  This is impossible.  A proper divisor of $q+1$ is less
than $q$, and $D\ge2$ follows from positive genus.
\end{proof}

\subsection{Unitary groups}\label{sec:unitarycomplete}
Write $q=p^f\ge3$ and put
\[
 A=q^2-q+1,\qquad B=q^2+q+1,\qquad D=rt,
 \qquad b=(q^2-1)/\delta,\qquad m=rb,
\]
where $\delta=(3,q+1)$ for the special unitary row and $\delta=1$ for
the projective unitary row.

\begin{lemma}\label{lem:unitarybounds}
The large-group inequality implies
\[
 t<q,\qquad tq<b,\qquad \delta t<q-1.
\]
\end{lemma}
\begin{proof}
The genus and order are
\[
 g=\frac{q-1}{2}\bigl(D(q+1)^2-B\bigr),\qquad
 |G|=\frac{r q^3(q^3+1)(q^2-1)}\delta.
\]
If $t\ge q$, then $D\ge3$ and
$g>D(q-1)(q+1)^2/3$.  Consequently
\[
 \frac{|G|}{g^2}
 <\frac{9q^3}{\delta r t^2(q-1)}
 \le\frac{9q}{\delta r(q-1)}\le\frac{27}{2}<18.
\]
This contradicts $|G|>24g(g-1)\ge18g^2$, since $g\ge4$.
For $\delta=1$, $t$ divides the odd number $A$; hence $t<q$ implies
$t\le q-2$, proving the other two assertions.

Suppose $\delta=3$.  For $q=5$, the conditions $t\mid A/3=7$ and
$e>1$ give $t=1$.  Otherwise $q\ge11$ and $q\equiv2\pmod3$.
If $t\ge q/3$, then $t-1\ge(q-2)/3$ and
\[
 g>\frac{r(q-1)(t-1)(q+1)^2}{2},\qquad
 \frac{|G|}{g^2}
 <\frac{12q^3}{r(q-1)(q-2)^2}
 \le\frac{12\cdot11^3}{10\cdot9^2}<20.
\]
The rational function in the middle is decreasing for $q\ge11$.
Here $g\ge6$, so $24g(g-1)\ge20g^2$, again a contradiction.
Thus $3t\le q-2$, which proves the lemma.
\end{proof}

\begin{proposition}\label{prop:unitaryjumps}
The Sylow group $S$ has precisely two positive lower jumps, namely
\[
 j_1=D,\qquad j_2=D(q+1).
\]
Its center $Z=[S,S]$ has order $q$, and
\[
 \cX/Z:\quad y^{q^2}-y=x^D.
\]
\end{proposition}
\begin{proof}
The standard unitary Sylow group is a Heisenberg group with center
$Z$ of order $q$. The complement has kernel $M$ on $S/Z$ and
scalar image of order $b>q-1$, which generates $\F_{q^2}$ over
$\F_p$. Its image on $Z$ has order $q-1$ and generates $\F_q$.
Both actions are therefore irreducible. A complement-invariant
normal subgroup not contained in $Z$ has nonzero image in $S/Z$;
by irreducibility that image is all of $S/Z$. It then contains
$[S,S]=Z$ and is $S$. Inside $Z$, irreducibility leaves no
intermediate subgroup. Thus the only proper nontrivial such
subgroup is $Z$.
The nonabelian group $S$ cannot be a ramification graded quotient.
Thus the ramification series is $S,Z,1$.

Let $j_1<j_2$ be its jumps. The kernel on $S/Z$ has order $r$, so
\eqref{eq:gradedprelim} gives $\gcd(m,j_1)=r$ and $j_1=rv$ with
$(v,b)=1$. Since
$\mathcal D_P=j_1(q^3-1)+(j_2-j_1)(q-1)$, Hurwitz and
\eqref{eq:unitary} give
\begin{equation}\label{eq:unitaryjumpidentity}
 q(q+1)j_1+j_2=rt(q+1)^2.
\end{equation}
Fix a noncentral element $u\in S$ and put $A_u=C_S(u)$.
We verify that $A_u$ is abelian of order $q^2$ and contains $Z$.
In the standard Heisenberg coordinates, the elements of $S$ are
\[
 (a,c),\qquad a,c\in\F_{q^2},\qquad c^q+c=a^{q+1},
\]
with multiplication
\[
 (a,c)(b,e)=(a+b,c+e+a^q b).
\]
Two such elements commute precisely when $a^q b=b^q a$.
If $u$ has first coordinate $a\ne0$, this condition is equivalent
to $b\in a\F_q$. Each such $b$ allows $q$ choices of the second
coordinate, so $|A_u|=q^2$. Any two elements with first coordinates
in $a\F_q$ commute, and the elements with first coordinate zero
form $Z$. Thus $A_u$ has the required properties.

Lower ramification groups are obtained by intersection with $S$.
For $A_u$ their successive nontrivial orders are therefore $q^2,q$,
with lower jumps $j_1,j_2$. Its upper jumps are
\[
 j_1,\qquad j_1+\frac{j_2-j_1}{q}.
\]
Hasse--Arf \cite[Chapter~IV]{SerreLocal} therefore gives
$q\mid j_2-j_1$. Reducing \eqref{eq:unitaryjumpidentity} modulo $q$
and using $j_1=rv$ and $(r,q)=1$ gives $v\equiv t\pmod q$.
The strict inequality $j_2>j_1$ in
\eqref{eq:unitaryjumpidentity} gives
$0<v<t(1+1/q)<t+1$, using Lemma~\ref{lem:unitarybounds}.
Hence $1\le v\le t<q$; together with $v\equiv t\pmod q$, this gives
$v=t$ and the asserted jumps.

Apply Lemma~\ref{lem:ASrigid} to the quotient by $Z$, with additive
group of order $q^2=p^{2f}$.  Its sole jump is $D$, by compatibility of upper numbering
with quotients. Choose $x=0$ as the image of
$\Omega\setminus\{P_\infty\}$, and mark a point in this orbit fixed
by the tame complement. Thus all the hypotheses of the lemma hold.
For $0\le k\le f$, a purported obstruction satisfies
$0<tp^k-h<tq<b$, which is impossible.
For $f<k<2f$, multiply the congruence by $p^{2f-k}$; since
$b\mid p^{2f}-1$, it gives
$b\mid t-hp^{2f-k}$.  This integer has absolute value less than $b$,
so it is zero.  But $p\nmid t$ and $p\mid hp^{2f-k}$.
Thus no obstruction exists and the quotient equation follows.
\end{proof}

\begin{proposition}\label{prop:unitarymodel}
The curve has equations
\begin{equation}\label{eq:Umodel}
 y^{q^2}-y=x^D,\qquad z^q+z=y^{q+1},
 \qquad D\mid q^2-q+1.
\end{equation}
\end{proposition}
\begin{proof}
On $H_q:z^q+z=y^{q+1}$, the Sylow quotient has coordinate
$u=y^{q^2}-y$. Its pullback by $u=x^D$ is the model in
\eqref{eq:Umodel}, before imposing the divisibility on $D$.
The degrees $q^3$ and $D$ are coprime, so the function fields are
linearly disjoint and the pullback is connected. Its lower jumps
are $D,D(q+1)$.
One may compute the base jumps directly with the uniformizer $y/z$
on $z^q+z=y^{q+1}$: its difference under a noncentral Heisenberg
translation has valuation $2$, and under a nonidentity central
translation has valuation $q+2$. Tame base change multiplies both
lower jumps by $D$. Indeed, locally at $P_\infty$ let $L$ be the
completed function field of the base cover, $K_0=K((u^{-1}))$, and
$L'=L(x)$. Then $\operatorname{Gal}(L'/K_0)=S\times C_D$, and $L'/L$
is tame and totally ramified of degree $D$, so its Herbrand function
is $w\mapsto w/D$ for $w\ge0$. For $w>0$ the ramification group
$\operatorname{Gal}(L'/K_0)_w$ is a $p$-group. Hence it lies in
$S=\operatorname{Gal}(L'/K_0(x))$ and is the $w$th ramification group
of $S$, and by Herbrand's theorem it maps isomorphically onto
$\operatorname{Gal}(L/K_0)_{w/D}$; see
\cite[Chapter~IV, \S\S1 and~3]{SerreLocal}.
Write the torus action on the marked quotient as
\[
 x\longmapsto\zeta x,\qquad y\longmapsto\alpha y,
 \qquad \alpha=\zeta^D\in\F_{q^2}^\times.
\]
The standard lift has $\tau(z)=\beta z$, where
$\beta=\alpha^{q+1}\in\F_q^\times$. For a central translation
$\sigma_c(z)=z+c$, $c^q+c=0$, one has
$\tau^{-1}\sigma_c\tau=\sigma_{\beta c}$.
Choosing a nonzero trace-zero parameter identifies the center with
$(\F_q,+)$, so $\beta$ is precisely the multiplier in
Lemma~\ref{lem:twists}.
The standard transformations
$(y,z)\mapsto(y+a,z+a^qy+c)$, with $a\in\F_{q^2}$ and
$c^q+c=a^{q+1}$, together with the diagonal transformations, give
the standard Sylow and tame actions. Take the origin as the marked
point on the standard cover. Lemma~\ref{lem:ASrigid} gives a
pointed $T$-equivariant identification of the first quotients.
To apply Lemma~\ref{lem:markings}, identify $S$ with its image
$\bar S$ in $\bar G$. The given tame complement projects to a
complement of $\bar S$ in $N_{\bar G}(\bar S)$.
By Schur--Zassenhaus it is conjugate to the standard torus by
an element of $\bar S$. Such inner conjugation is trivial on
$\bar S/\Phi(\bar S)$. Thus both induced actions are scalar and
irreducible, as in Proposition~\ref{prop:unitaryjumps}.
Lemma~\ref{lem:markings} now identifies the full inverse-conjugation
actions without changing the pointed first-quotient identification.
The marked $S/Z$-quotients therefore agree.

Both covers have largest upper jump
\[
 U=D+\frac{D(q+1)-D}{q^2}=D(1+1/q).
\]
Apply Lemma~\ref{lem:twists} to their restrictions over
$\mathbb A^1_x$, with $P=S$, $E=Z$, and $Q=q$.  A nonzero $q$-reduced twist monomial has exponent
$i=p^kh$, with $0\le k<f$ and $h\le U$, and must satisfy
\[
 i\equiv W:=D(q+1)\pmod{m},\qquad m=r(q^2-1)/\delta.
\]
But Lemma~\ref{lem:unitarybounds} gives $0<W<m$, while
\[
 0<i\le p^{f-1}U=\frac{W}{p}<W.
\]
This is impossible.  The twist is zero, proving the two equations.

It remains to prove the divisibility; assume $D>1$.
The group $C_D$, acting on $x$, fixes pointwise the original orbit
$\Omega$, identified with the $q^3+1$ points above $H_q(\F_{q^2})$.
The finite group $\widehat G=\langle G,C_D\rangle$ preserves this set and
inherits the size and no-fixed-point hypotheses. A Sylow $p$-subgroup
containing $S$ is free on its $q^3$ points other than $P_\infty$;
its order is therefore $q^3$. As in
Proposition~\ref{prop:linearcomplete}, Theorem~\ref{thm:kernel}
gives $C_D\le Z(\widehat G)$ and $D\mid q^3+1$.
As in Proposition~\ref{prop:linearcomplete}, an element of $M$ acts
by $x\mapsto\xi x$ and $y\mapsto\xi^Dy$, and it centralizes the
translations of $y$, so $\xi^D=1$. Thus the marked diagonal action
identifies $M$ with the order-$r$ subgroup of $C_D$, and
$G\cap C_D=M$.

The following argument is of the type used in
\cite{TafazolianSkabelund}. The quotient $\cX/C_D$ is $H_q$. Denote
by $Q_\infty,Q_0$ the
images of $P_\infty$ and the marked origin. Double transitivity of
$G$ on $\Omega$ gives $\sigma\in G$ interchanging these points.
Since $\sigma$ centralizes $C_D$, the ratio $\varphi=\sigma(x)/x$
belongs to $K(H_q)$. On this quotient $u=y^{q^2}-y$ satisfies
$x^D=u$ and
\[
 \operatorname{div}(u)=\sum_{Q\in H_q(\F_{q^2})}Q-(q^3+1)Q_\infty.
\]
The induced action preserves the sum; hence $\varphi^D=\sigma(u)/u$
gives
\[
 \operatorname{div}(\varphi)=\frac{q^3+1}{D}(Q_\infty-Q_0).
\]
The class $[Q_0-Q_\infty]$ has exact order $q+1$; compare the lifting
argument in \cite[Lemma~3.9]{MTZ}. Indeed, a tangent coordinate has
divisor $(q+1)(Q_0-Q_\infty)$, whereas the least positive nongap at
$Q_\infty$ is $q$. A proper divisor of $q+1$ is smaller than $q$,
so no smaller order is possible. Consequently
$q+1\mid(q^3+1)/D$, or $D\mid q^2-q+1$.
\end{proof}

\subsection{Ree groups}\label{sec:reecomplete}
Write
\[
 q=3q_0^2=3^{2f_0+1},\quad q_0=3^{f_0}\ge3,\quad
 n=q^3+1,\quad m=r(q-1),\quad m_-=q-3q_0+1.
\]

\begin{lemma}\label{lem:reebound}
The tame stabilizer order in the Ree row satisfies $e>3q_0$.
The largest upper jump $U$ of the cover $\cX\to\cX/S$ satisfies
$U<rq_0$. Moreover, the distinct groups occurring in the lower
ramification filtration of $S$ at its fixed point are
$S\supset S'\supset Z(S)\supset1$; this lists the groups, not the
jump numbers. The cyclic complement acts irreducibly on $S/S'$,
$S'/Z(S)$ and $Z(S)$.
\end{lemma}
\begin{proof}
Assume $e\le3q_0$.  Put $g_1=1+n(q-1-e)/(2e)$, the genus from the
genus formula with $r=1$.  Here $g_1>1$, and
\[
 g=rg_1+\frac{(r-1)(n-2)}2\ge rg_1.
\]
Since $g(g-1)\ge r g_1(g_1-1)$ and
$|G|=r n q^3(q-1)$, it suffices to disprove the size inequality for $r=1$.
For $q_0\ge9$, put $c=q-1-3q_0$.  Then
$c\ge\frac45(q-1)$ and $2c^2>q(q-1)$.
Since $g_1\ge1+nc/(6q_0)$ and $n>q^3$,
\[
 24g_1(g_1-1)>
 \frac{2n^2c^2}{q}>nq^3(q-1).
\]
For $q_0=3$, Table~\ref{tab:main} leaves only $e=7$ under
$1<e\le9$: the three torus orders are the primes $7,19,37$.
This gives $g_1=26715$ and
\[
 |\Ree(27)|=10073444472
 <17127948240=24g_1(g_1-1).
\]
Thus $e>3q_0$ in all cases.

We next determine the ramification filtration.
Write $Z=Z(S)$. The Ree Sylow group has $|S|=q^3$, $|S'|=q^2$,
$|Z|=q$, $\Phi(S)=S'$, and $[S,S']=Z$; $S'$ is elementary abelian.
The scalar image of the complement acts on $S/S'$, $S'/Z$ and $Z$ by
\[
 \lambda,\qquad\lambda^{q_0+1},\qquad\lambda^{2q_0+1},
 \qquad \lambda\in\F_q^\times.
\]
The group structure and these weights follow from the standard
transformations
\begin{align*}
 v&\mapsto Av+B,\\
 y&\mapsto A^{q_0+1}y+AB^{q_0}v+C,\\
 z&\mapsto A^{2q_0+1}z-A^{q_0+1}B^{q_0}y+AB^{2q_0}v+E
\end{align*}
with $A\in\F_q^\times$, $B,C,E\in\F_q$; see
\cite[proof of Lemma~4.2]{Skabelund} and \cite[Proposition~49 and its preamble]{GMQZ}.
For instance, commutators of elements with first parameter $B$ and
of elements in $S'$ with second parameter $C$ have central parameter
$\pm B^{q_0}C$, so $[S,S']=Z$.
Since $(q-1,q_0+1)=2$ and $(q-1,2q_0+1)=1$, the scalar orders
are $q-1,(q-1)/2,q-1$. Each exceeds the multiplicative order of
any proper subfield. The three modules are irreducible, and the last
two are nonisomorphic because their scalar orders differ.

The only proper nontrivial torus-invariant normal subgroups are $S'$
and $Z$. Indeed, an invariant normal subgroup not contained in
$S'$ has nonzero image in the irreducible module $S/S'$ and hence
surjects onto it. The Frattini property then makes it $S$.
If a normal subgroup $N\le S'$ maps nontrivially
to $S'/Z$, then $NZ=S'$ by irreducibility, whence
$Z=[S,S']=[S,N]\le N$ and $N=S'$. Inside $Z$, irreducibility applies.
The first two factors cannot merge, because $S/Z$ is nonabelian and
a ramification graded quotient is elementary abelian. The last two
cannot merge either: the tame action on $S'$ is semisimple with two
nonisomorphic constituents, whereas a ramification graded quotient
is isotypic by \eqref{eq:gradedprelim}.
Thus the three positive lower jumps are
$j_1<j_2<j_3$, with groups $S,S',Z(S)$.
Here
$\mathcal D_P=(q-1)(q^2j_1+qj_2+j_3)$, so Hurwitz and the genus
formula give
\[
 q^2j_1+qj_2+j_3=\frac{rn}{e}.
\]
Consequently
\[
 U=j_1+\frac{j_2-j_1}{q}+\frac{j_3-j_2}{q^2}
  =\frac{rn}{e q^2}-\frac{j_1}{q}-\frac{j_2}{q^2}
  <\frac{rn}{e q^2}<rq_0.
\]
For the last strict inequality use $e\ge3q_0+1$ and
$n=q^3+1<q^2q_0(3q_0+1)$.
\end{proof}

\begin{proposition}\label{prop:reemodel}
In the Ree row, $t=m_-/e$ is a positive integer. With $D=rt$,
one has
\begin{equation}\label{eq:Rmodel}
 \begin{split}
 v^q-v&=x^D,\\
 y^q-y&=v^{q_0}(v^q-v),\\
 z^q-z&=v^{2q_0}(v^q-v),
 \end{split}
 \qquad D\mid m_-,\qquad e=\frac{m_-}{t}.
\end{equation}
In particular the other two torus branches in Table~\ref{tab:main}
do not occur.
\end{proposition}
\begin{proof}
The complement has kernel $M$ on the first ramification quotient.
Hence $\gcd(m,j_1)=r$, and we may write $j_1=r\ell$ with
$(\ell,q-1)=1$. Since $j_1\le U<rq_0$,
$1\le\ell<q_0$. The $S/S'$-cover has just one upper jump,
namely $j_1$; its degree-$3$ quotients have this jump, so
$3\nmid j_1$ and $3\nmid\ell$.
Choose $x=0$ and a point above it fixed by the cyclic complement,
as in the linear case. Lemma~\ref{lem:ASrigid} applies to
$\cX/S'$ once its obstruction is excluded.
For a purported obstruction
$(q-1)\mid\ell3^k-h$, with $0<h<\ell$, $3\nmid h$ and
$0\le k<2f_0+1$, if $k\le f_0$, the absolute size is
less than $q-1$.  If $k\ge f_0+1$, multiply by
$3^{2f_0+1-k}\le q_0$ to get
$(q-1)\mid\ell-h3^{2f_0+1-k}$; again its absolute size is less than
$q-1$, and equality would contradict $3\nmid\ell$.
Thus
\[
 \cX/S':\quad v^q-v=x^{r\ell}.
\]

Set $D=r\ell$ temporarily. The standard model in \eqref{eq:Rmodel}
is the pullback by $u=x^D$ of the Ree Sylow cover with quotient
coordinate $u=v^q-v$. It is connected since $(D,q^3)=1$.
Its lower jumps are
\[
 D,\quad D(3q_0+1),\quad D(q+3q_0+1).
\]
Indeed, consider first the Ree curve itself. By the standard
transformations in the proof of Lemma~\ref{lem:reebound}, the elements
with $A=1$ and $B=0$ form a normal subgroup of order $q^2$ with
abelian quotient, hence equal to $S'$; they are the translations
$(y,z)\mapsto(y+c_1,z+c_2)$ with $c_1,c_2\in\F_q$, and
$R_q/S'=\PP^1_v$. Put
$f_1=v^{q_0}(v^q-v)$ and $f_2=v^{2q_0}(v^q-v)$. For
$\lambda,\mu\in\F_q$, one has
$(\lambda y+\mu z)^q-(\lambda y+\mu z)=\lambda f_1+\mu f_2$, so by the
trace construction of Section~\ref{subsec:coversprelim} the
corresponding degree-$3$ character of $S'$ has class
$[\lambda f_1+\mu f_2]$ modulo $\wp_3(K[v])$. Since
$q+q_0=q_0(3q_0+1)$ and $q+2q_0=q_0(3q_0+2)$, its $3$-reduced
representative is
\[
 \mu^{1/q_0}v^{3q_0+2}+\lambda^{1/q_0}v^{3q_0+1}
 -\lambda v^{q_0+1}-\mu v^{2q_0+1}.
\]
Thus the upper jumps of $S'$ over $K(v)$ are $3q_0+1$, attained by
the characters with $\mu=0$, and $3q_0+2$; the second step is the
common kernel $Z(S)$ of the characters with $\mu=0$. The lower jumps
of $S'$ are therefore $3q_0+1$ and $3q_0+1+q$. The $S/S'$-cover
$v^q-v=u$ has jump $1$. Since lower numbering is compatible with
subgroups and upper numbering with quotients, $S_i=S$ for $i\le1$,
$S_i=S'$ for $2\le i\le3q_0+1$, $S_i=Z(S)$ for
$3q_0+2\le i\le q+3q_0+1$, and $S_i=1$ afterwards. These are the
jumps recorded in Section~\ref{subsec:standardcurves}. Tame base
change multiplies them by $D$, as in the proof of
Proposition~\ref{prop:unitarymodel}.
Its last upper jump is $D(1+1/q_0+1/q)<rq_0$, since
$(q_0-1)(1+1/q_0+1/q)<q_0$.
The given cover satisfies the same strict upper bound by
Lemma~\ref{lem:reebound}. Let $U_*$ be the larger of the two largest
upper jumps, so $U_*<rq_0$. By compatibility of upper numbering
with quotients, this is also a common bound for their $S/Z(S)$-covers.
The first quotients are pointed and $T$-equivariantly identified
by Lemma~\ref{lem:ASrigid}. The normalizer argument used in
Proposition~\ref{prop:unitarymodel} gives scalar actions on
$S/\Phi(S)=S/S'$, and Lemma~\ref{lem:reebound} gives their
irreducibility. Apply Lemma~\ref{lem:markings} to the two full
$S$-covers and retain the resulting markings throughout.

First apply Lemma~\ref{lem:twists} with $P=S/Z(S)$ and
$E=S'/Z(S)\cong(\F_q,+)$, over the common marked quotient with
group $S/S'$. The torus order is $r(q-1)$, and its inverse-conjugation
multiplier on this kernel is $\zeta^{D(q_0+1)}$. A twist monomial would have exponent
$i=3^kh$, with $0\le k<2f_0+1$, $3\nmid h$, $h\le U_*<rq_0$, and
\[
 i\equiv r\ell(q_0+1)\pmod{r(q-1)}.
\]
Thus $r\mid h$, and, on writing $h=rh_0$,
\begin{equation}\label{eq:reedigits}
 3^kh_0\equiv\ell(q_0+1)\pmod{q-1},\qquad
 0<h_0<q_0.
\end{equation}
Multiplication by $3$ modulo $3^{2f_0+1}-1$ cyclically rotates
$2f_0+1$ base-$3$ digits.  The nonzero digits of $h_0<3^{f_0}$ fit in a
cyclic interval of $f_0$ positions.  Those of
$\ell(q_0+1)=\ell+\ell3^{f_0}$ do not: positions $0$ and $f_0$ are both
nonzero, because $3\nmid\ell$, and no such cyclic interval contains
both positions.  This disproves \eqref{eq:reedigits}.
The first twist is zero. Lemma~\ref{lem:twists} therefore gives a
pointed $T$-equivariant isomorphism of the marked $S/Z(S)$-covers.
We use this same isomorphism as the common quotient in the second
comparison, without changing the full group markings or chosen points.
Apply Lemma~\ref{lem:twists} again, now with $P=S$,
$E=Z(S)\cong(\F_q,+)$, and the same bound $U_*$.
The inverse-conjugation multiplier on the kernel is
$\zeta^{D(2q_0+1)}$. The identical argument replaces the right side of
\eqref{eq:reedigits} by $\ell(2q_0+1)$.
This integer is less than $q-1$, and its digits in positions $0$ and
$f_0$ are again nonzero; there is no carry from the lower block
$\ell<3^{f_0}$.  The same cyclic-interval argument excludes the twist.
This proves the three equations.

As a degree-$D$ Kummer cover of the Ree curve $R_q$, this model
is totally ramified at its $n$ rational points: $v^q-v$ has simple
finite zeros and its only pole has order $q^3$, coprime to $D$.
Thus Hurwitz, with the known Ree genus, gives
\[
 2g-2=D(2g(R_q)-2)+n(D-1)
      =n\left(\frac{D(q-1)}{m_-}-1\right).
\]
Comparing with \eqref{eq:classification-master} gives
$e=m_-/\ell$. Hence $\ell\mid m_-$ and $t=\ell$.
The other two torus branches are excluded, since $e>1$ and
$m_-$ is coprime to both $(q+1)/4$ and $q+3q_0+1$.

Finally assume $D>1$. The deck group $C_D$ fixes pointwise
$\Omega$, identified by the marked model with the $n$ points above
$R_q(\F_q)$. As in Proposition~\ref{prop:linearcomplete}, the finite
group $\widehat G=\langle G,C_D\rangle$ preserves $\Omega$ and still
has Sylow order $q^3$: a Sylow subgroup containing $S$ acts freely on
the other $q^3$ points. Theorem~\ref{thm:kernel} applies to
$\widehat G$, so $C_D\le Z(\widehat G)$ and $D\mid n$.
As in Proposition~\ref{prop:linearcomplete}, an element of $M$ acts
by $x\mapsto\xi x$ and $v\mapsto\xi^Dv$, and it centralizes the
translations of $v$, so $\xi^D=1$. Thus the marked diagonal action
gives $M\le C_D$, and $G\cap C_D=M$.

As in Proposition~\ref{prop:unitarymodel}, compare
\cite{TafazolianSkabelund}. Write $Q_\infty,Q_0$ for the images of
$P_\infty$ and the marked origin on $R_q=\cX/C_D$. Double transitivity gives $\sigma\in G$
interchanging these points. Centrality of $C_D$ implies
$\varphi=\sigma(x)/x\in K(R_q)$. Since
\[
 \operatorname{div}(v^q-v)=\sum_{Q\in R_q(\F_q)}Q-nQ_\infty,
\]
comparison of $(\sigma(x)/x)^D$ with $\sigma(v^q-v)/(v^q-v)$ gives
\[
 \operatorname{div}(\varphi)=\frac nD(Q_\infty-Q_0).
\]
Put $L=n/m_-=(q+1)(q+3q_0+1)$. The class
$[Q_0-Q_\infty]$ has exact order $L$. To see that its order divides
$L$, use the function $w_8$ of \cite[Eq.~(3)]{Skabelund}, with his
$x$ replaced by $v$. On his degree-$m_-$ cover, the involution
interchanging these points sends the Kummer coordinate $w$ to
$w/w_8$; see \cite[proof of Lemma~4.2]{Skabelund}. Since
$\operatorname{div}(w)=\sum_{P\in\widetilde\Omega}P-n\widetilde P_\infty$
on that cover, taking divisors and descending gives
\[
 \operatorname{div}_{R_q}(w_8)=L(Q_0-Q_\infty).
\]
The least positive nongap at $Q_\infty$ is $q^2$. The coordinate
$v$ has this pole order. Conversely, let $f$ be a nonconstant
$\F_q$-rational function whose only pole is $Q_\infty$, of order $d$.
Then $f$ has degree $d$, it takes values in $\F_q$ at the $q^3$
affine rational points, and each value in $\F_q$ has at most $d$
preimages. Hence $q^3\le qd$, that is, $d\ge q^2$. Riemann--Roch base change shows
that any geometric nongap also has such an $\F_q$-rational
representative. Since $L<2q^2$, no proper divisor of $L$ can be
a positive nongap. Thus the divisor class has order $L$, and
$L\mid n/D$ gives $D\mid m_-$.
\end{proof}

\Needspace{8\baselineskip}
\section{Known full groups and the large subgroups}\label{sec:fullgroups}

Section~\ref{sec:equations} determines the possible curves.
We now determine their full automorphism groups and classify the
subgroups satisfying the large-group hypothesis. The one-pole
calculation below controls the intermediate Ree covers; the
central lift models then give the exact subgroup conditions.
Write $d$ for the Kummer degree and $r$ for the central kernel
of $G$, so $r\mid d$ and $t=d/r$.

\subsection{Full automorphism groups}\label{subsec:fullgroups}

\begin{lemma}[One-pole spaces on the Kummer models]\label{lem:semigroup}
Let $Y$ be a smooth projective curve over $K$, and let $X$ be the
smooth model of $K(Y)(w)$, where $w^d=f$ and $p\nmid d$.
Suppose $(f)_\infty=NP_\infty$, all finite zeros of $f$ are simple,
and $(d,N)=1$. There is a unique point $Q_\infty$ above $P_\infty$,
and
\[
 H_X(Q_\infty)=
 \bigcup_{j=0}^{d-1}\bigl(dH_Y(P_\infty)+jN\bigr).
\]
More precisely, for every integer $s\ge0$,
\[
 L_X(sQ_\infty)=\bigoplus_{j=0}^{d-1}
 w^jL_Y\!\left(\left\lfloor\frac{s-jN}{d}\right\rfloor P_\infty\right),
\]
where a summand with negative divisor degree is zero.
\end{lemma}
\begin{proof}
At a point $Q$ above $P_\infty$ with ramification index $e_Q$,
$d\,v_Q(w)=-e_QN$. Coprimality gives $d\mid e_Q$, so
$[K(X):K(Y)]=e_Q=d$. Hence $Q=Q_\infty$ is unique and
$(w)_\infty=NQ_\infty$.

Put $A=\Gamma(Y\setminus\{P_\infty\},\mathcal O_Y)$.
The algebra $A[T]/(T^d-f)$, with $w$ the image of $T$, is normal. Away from the zeros of $f$,
its defining polynomial has invertible derivative $dw^{d-1}$,
so the cover is \'etale. At a simple zero, $f$ is a local parameter
and $w^d=f$ makes the local maximal ideal upstairs generated by
$w$; that local ring is a discrete valuation ring. Thus this
algebra is the integral closure of $A$ and equals
$\Gamma(X\setminus\{Q_\infty\},\mathcal O_X)$.
It is free over $A$ with basis $1,w,\ldots,w^{d-1}$.

For $0\ne h_j\in A$, the pole order of $h_jw^j$ is
$d\operatorname{pole}_{P_\infty}(h_j)+jN$.
For distinct $j$ these orders are distinct modulo $d$; hence the
largest pole cannot cancel in a sum. Requiring every order to be
at most $s$ proves the displayed decomposition of $L_X(sQ_\infty)$.
It also proves one inclusion of semigroups, while a function
$hw^j$ realizes every order $d\operatorname{pole}(h)+jN$.
\end{proof}

The full groups of the linear and unitary model families are already
known: the linear case is contained in Bonini--Montanucci--Zini
\cite{BMZ}, and the order in the unitary case is the $n=3$ case of
Montanucci--Tizziotti--Zini \cite[Theorem~1.1]{MTZ}. The top Ree cover
was treated in \cite{GMQZ}. We record these groups because their
explicit lift structure is needed to classify the subgroups $G$ that
occur in our converse theorem. The argument below is retained for the
intermediate Ree covers and to keep the subgroup calculation
self-contained.

Before recording the groups, note that every displayed model
has zero $p$-rank, independently of the classification argument.
For $L(q,d)$ the translations of $y$ give a group of order $q$
with quotient $\PP^1_x$, ramified only and totally at infinity.
For $U(q,d)$ and $R(q,d)$ use the standard Sylow covers of their
base curves, followed by the tame base change $u=x^d$.
The base change is linearly disjoint from the Sylow cover because
their degrees are coprime. It preserves the full wild inertia at
infinity and creates no finite ramification in the Sylow cover.
Thus these natural $p$-groups have orders $q^3$, rational quotient, and
exactly one short orbit, a fixed point.
Deuring--Shafarevich gives $\gamma=0$ in all three cases.
The natural lift groups used below are given explicitly in
\eqref{eq:lineargroup} and \eqref{eq:unitarygroup}, and by
\cite[Lemma~4.2]{Skabelund} in the Ree case; their construction does
not assume that they are the full automorphism groups.

\begin{proposition}\label{prop:fullorders}
For each displayed model of genus at least two, its full geometric
automorphism group $A_X$ is as follows. The linear row is a special
case of \cite{BMZ}, the order in the unitary row for $d>1$ is taken
from \cite[Theorem~1.1]{MTZ}, and the top Ree case is
\cite[Theorem~2]{GMQZ}:
\begin{center}
\begin{tabular}{@{}lll@{}}
\toprule
Model & Central quotient & $|A_X|$\\
\midrule
$L(q,d):\ y^q-y=x^d$ & $\PGL(2,q)$ & $d q(q^2-1)$\\
$U(q,d)$ in \eqref{eq:Umodel} & $\PGU(3,q)$ &
$d q^3(q^3+1)(q^2-1)$\\
$R(q,d)$ in \eqref{eq:Rmodel} & $\Ree(q)$ &
$d q^3(q^3+1)(q-1)$\\
\bottomrule
\end{tabular}
\end{center}
The Kummer deck group is the kernel of the displayed quotient action.
For the Ree model the extension is $C_d\times\Ree(q)$.
\end{proposition}
\begin{proof}
For $L(q,d)$, the curve $y^q-y=x^d$ has $d<q$ and $d\mid q+1$;
\cite[Lemma~2.3(1)]{BMZ} gives its cyclic kernel and full
projective quotient. For $U(q,1)=H_q$, the assertion is the classical
full-group theorem for the Hermitian curve recalled in
Section~\ref{subsec:standardcurves}. For $U(q,d)$ with $d>1$, put
$m=q^2-q+1$ and $s=m/d$. Then the model is the $n=3$ curve
$\mathcal Y_{3,s}$ of \cite{TTH}. From Theorem~1.1 of \cite{MTZ} we
take only the order in the table. The lift group
\eqref{eq:unitarygroup} acts on the model and has this order, so it is
the full group; its structure is determined from that lift below, see
Remark~\ref{rem:MTZstructure}.

For the Ree family, $d=1$ gives the Ree curve itself, whose full group
is recalled in Section~\ref{subsec:standardcurves}. The case
$d=m_-=q-3q_0+1$ is \cite[Theorem~2]{GMQZ}. Suppose now that
$1<d<m_-$ and $d\mid m_-$; for the Suzuki analogue see
\cite[Theorem~8]{AFTV}. Let $S$ be the natural group of order
$q^3$, fixing the distinguished point $P_\infty$ and having rational
quotient. Choose a Sylow $p$-subgroup $\widetilde S$ of $A_X$
containing $S$. Since $X$ has zero $p$-rank, $\widetilde S$ fixes
$P_\infty$ by Lemma~\ref{lem:sylow}.

On the Ree quotient the least positive nongap at infinity is $q^2$,
as shown in the proof of Proposition~\ref{prop:reemodel}. The function
$v^q-v$ has a single pole of order $q^3$ and simple finite zeros.
Since $d<q$ and $(d,q)=1$, Lemma~\ref{lem:semigroup} gives
\[
 L(dq^2P_\infty)=\langle1,v\rangle,
\]
and $dq^2$ is the least positive nongap on $X$ at $P_\infty$.
The group $\widetilde S$ preserves this space and hence acts on the
$v$-line by affine transformations. Its multiplicative image in
$K^\times$ is a $p$-group and is therefore trivial; thus the induced
transformations are translations.

The finite branch locus of $X\to\PP^1_v$ is exactly $\F_q$.
Indeed, the Ree cover of the $v$-line is \'etale over the affine line,
and the Kummer cover $x^d=v^q-v$ ramifies precisely over the simple
zeros $v\in\F_q$, since $d>1$. A translation preserving this branch
locus has its parameter in $\F_q$, so the image of $\widetilde S$
on the $v$-line has order at most $q$. The kernel fixes $v$, and
\[
 [K(X):K(v)]=dq^2.
\]
Its order therefore divides $dq^2$ and, being a $p$-power, is at most
$q^2$. Hence $|\widetilde S|\le q^3=|S|$, proving that $S$ is
already a Sylow subgroup of $A_X$.

Apply Theorem~\ref{thm:GMPprelim} to $A_X$. Zero $p$-rank makes $S$
a TI subgroup, and its normalizer has cyclic tame quotient. The known
lifted Ree group moves $P_\infty$, so $S$ is not normal in $A_X$.
The small Ree quotient is impossible, since its Sylow $3$-subgroup
has order $27$, whereas $|S|=q^3$ with $q\ge27$.
The group $A_X$ has no common fixed point, since the lifted Ree group
has none. Lemma~\ref{lem:quadratic-data}, which has no largeness
hypothesis, therefore applies. Among its socles, the nonabelian exponent-$9$
Sylow group forces the Ree case with the same parameter $q$.
The Sylow fixed-point orbit has size $q^3+1$ and contains the natural distinguished orbit of the lifted
Ree group, also of size $q^3+1$; the two sets therefore coincide.
The first part of the proof of Theorem~\ref{thm:kernel} uses only
this doubly transitive action and the cyclic tame complement. It
shows that the kernel on this orbit is cyclic and central in $A_X$.
Since $C_d$ fixes the orbit pointwise, it is central as well.
Consequently $A_X/C_d$ embeds in the full automorphism group
$\Ree(q)$ of the quotient Ree curve.

Conversely, Skabelund's lift on the $m_-$-cover descends to the
quotient by $C_{m_-/d}$. The full group of the top cover is
$C_{m_-}\times\Ree(q)$ by \cite[Theorem~2]{GMQZ}, so its induced
group on $X$ is $C_d\times\Ree(q)$. Together with the preceding
upper bound, this proves $A_X=C_d\times\Ree(q)$.
\end{proof}

\Needspace{6\baselineskip}
\subsection{The large subgroups}\label{subsec:large-subgroups}

\paragraph{The linear lift.}
The full group in this case is known from \cite{BMZ}; we record the
following explicit lift because it makes the large subgroups transparent.
Set $a_L=(q+1)/d$.  The lift group is
\begin{equation}\label{eq:lineargroup}
 A_L(q,d)=
 \frac{\{(B,\kappa):B\in\operatorname{GL}_2(\F_q),\ \kappa\in K^\times,\ 
                         \kappa^d=\det B\}}
      {\{(cI,c^{a_L}):c\in\F_q^\times\}}.
\end{equation}
For $B=\left(\begin{smallmatrix}A&B_0\\C&E\end{smallmatrix}\right)$
its action is
\[
 y\longmapsto\frac{Ay+B_0}{Cy+E},\qquad
 x\longmapsto\frac{\kappa x}{(Cy+E)^{a_L}}.
\]
The identity
$\bigl((Ay+B_0)/(Cy+E)\bigr)^q-\bigl((Ay+B_0)/(Cy+E)\bigr)
=\det(B)(y^q-y)/(Cy+E)^{q+1}$ verifies the action and the scalar
kernel in \eqref{eq:lineargroup}.
Let $L$ be the image of $\operatorname{SL}_2(q)$ under
$B\mapsto[(B,1)]$, and put
\begin{equation}\label{eq:epsilonL}
 \epsilon_L=\frac{2}{(2,a_L)}.
\end{equation}
Then
\[
 A_L'=L,\quad |L\cap C_d|=\epsilon_L,\quad
 L\cong\operatorname{SL}_2(q)/C_{2/\epsilon_L},\quad
 A_L/L\cong C_{2d/\epsilon_L}.
\]
Indeed the determinant pairs form a cyclic group of order $d(q-1)$;
modding out the scalar image gives the displayed cyclic quotient.
The group $\operatorname{SL}_2(q)$ is perfect for the values $q\ge5$
occurring here, which proves the derived-group assertion.

\paragraph{The unitary lift.}
The order of the full group is supplied by
\cite[Theorem~1.1]{MTZ}. We determine its central extension
structure directly from the following lift model; no splitting
of the natural central exact sequence is assumed.
Put $A=q^2-q+1$, $a_U=A/d$, and $\mu=(3,q+1)$.
Let $\operatorname U_3(q)$ denote the unitary isometry group of
$Z^qW+ZW^q-Y^{q+1}$.  Then
\begin{equation}\label{eq:unitarygroup}
 A_U(q,d)=
 \frac{\{(B,\kappa):B\in\operatorname U_3(q),\ \kappa\in K^\times,\ 
                         \kappa^d=\det B\}}
      {\{(cI,c^{a_U}):c^{q+1}=1\}}.
\end{equation}
If $\ell$ is the projective denominator of $B$ on the affine Hermitian
curve, the lift sends $x$ to $\kappa x/\ell^{a_U}$.
The identity to be checked is
\[
 (y^{q^2}-y)\circ B
       =\det(B)\,(y^{q^2}-y)/\ell^{q^2-q+1}.
\]
It holds for the Heisenberg translations, for the diagonal unitary
matrices, and for the interchange of $Z,W$; these generate the
unitary group. Writing $\ell_B$ for the denominator of $B$, the identity
is compatible with composition, since
$\ell_{B_1B_2}=\ell_{B_2}\,(\ell_{B_1}\circ B_2)$; hence it suffices to
check it on generators. For the diagonal matrices one can use
$\operatorname{diag}(\lambda^q,1,\lambda^{-1})$, with
$\lambda\in\F_{q^2}^{\times}$ and coordinates ordered as $(Z,Y,W)$:
the affine action is $(z,y)\mapsto(\lambda^{q+1}z,\lambda y)$,
the determinant is $\lambda^{q-1}$, and $\ell=\lambda^{-1}$.
The interchange gives the factor
$-z^{-(q^2-q+1)}$, using $z^q+z=y^{q+1}$.
The scalar denominator in \eqref{eq:unitarygroup} is valid because
$a_U d=q^2-q+1\equiv3\pmod{q+1}$.

For $L$ the image of $\operatorname{SU}_3(q)$ under
$B\mapsto[(B,1)]$, put
\begin{equation}\label{eq:epsilonU}
 \epsilon_U=\frac{\mu}{(\mu,A/d)}.
\end{equation}
Then
\[
 A_U'=L,\quad |L\cap C_d|=\epsilon_U,\quad
 L\cong\operatorname{SU}_3(q)/C_{\mu/\epsilon_U},\quad
 A_U/L\cong C_{\mu d/\epsilon_U}.
\]
This follows just as in the linear case from the determinant pairs and
the scalar quotient, using the perfectness of $\operatorname{SU}_3(q)$.

\begin{remark}\label{rem:MTZstructure}
We use \cite[Theorem~1.1]{MTZ} only for the order of $A_U(q,d)$.
That theorem describes the group as
$\PGU(3,q)\rtimes C_{m/s}$, where $m/s=d$ in our notation. If $\mu=3$ and $3\nmid a_U$, then
$\epsilon_U=3$ and $A_U'=L\cong\operatorname{SU}_3(q)$. A subgroup
$H\cong\PGU(3,q)$ of $A_U$ would give $H'\cong\PSU(3,q)$ as a subgroup
of index $3$ in $\operatorname{SU}_3(q)$. The action on its three
cosets would then give a nontrivial homomorphism from the perfect
group $\operatorname{SU}_3(q)$ to the solvable group $S_3$, which is
impossible. Hence in this case $A_U$ has no subgroup isomorphic to
$\PGU(3,q)$, and the central extension by $C_d$ does not split.
For the Giulietti--Korchm\'aros curve, where $d=q^2-q+1$ and $a_U=1$,
and when $3\mid q+1$, this agrees with \cite{GKoriginal}: its
automorphism group has a subgroup of index $3$ isomorphic to
$\operatorname{SU}(3,q)\times C_{(q^2-q+1)/3}$.
\end{remark}

\begin{theorem}[All large subgroups in the three families]\label{thm:subgroups}
Apart from the exceptional $A_7\le\Aut(H_5)$, every group in the
hypotheses with $g\ge4$ is one of the following groups on its indicated
model.  In every row $r\mid d$, $t=d/r$, and the strict
inequality $|G|>24g(g-1)$ is imposed. In the full projective rows,
$A[t]$ denotes the full inverse image in $A$ of the unique index-$t$
subgroup of the cyclic quotient $A/L$.
\begin{enumerate}[label=\textup{(\roman*)},leftmargin=*]
\item On $L(q,d)$, with $L$ and $\epsilon_L$ as above:
\[
\begin{array}{c|c|c}
 G/C_r& G&\text{existence condition}\\ \hline
 \PSL(2,q)&LC_r&\epsilon_L\mid r\\
 \PGL(2,q)&A_L[t]&t\text{ odd}.
\end{array}
\]
The orders are $rq(q^2-1)/2$ and $rq(q^2-1)$.
\item On $U(q,d)$, with $L$ and $\epsilon_U$ as above:
\[
\begin{array}{c|c|c}
 G/C_r& G&\text{existence condition}\\ \hline
 \PSU(3,q)&LC_r&\epsilon_U\mid r\\
 \PGU(3,q)&A_U[t]&(t,\mu)=1.
\end{array}
\]
The orders are $rq^3(q^3+1)(q^2-1)/\mu$ and
$rq^3(q^3+1)(q^2-1)$.  When $\mu=1$ the rows coincide.
\item On $R(q,d)$, the group is exactly $C_r\times\Ree(q)$, of order
$rq^3(q^3+1)(q-1)$.
\end{enumerate}
All the listed groups occur on the indicated models. They act without
a common fixed point, and the models have zero $p$-rank.
\end{theorem}
\begin{proof}
The preceding geometric propositions show that the original Sylow
order equals the natural Sylow order of the model, except for the
$A_7$ case already isolated.  Thus the structural quotient of $G$ is
the indicated special or full projective group, with central kernel
$C_r\le C_d$.
For either projective case, $G C_d$ is the full preimage of that
projective group.  Since $C_d$ is central,
$(G C_d)'=G'$, so $G$ contains the perfect derived group $L$.
In the special row it is therefore $LC_r$, which exists with exactly
this kernel precisely when $\epsilon\mid r$.
In the full projective row put $k=2$ or $\mu$, respectively.
The cyclic quotient $A/L$ has order $kd/\epsilon$, and the image
of $C_d$ is its index-$k$ subgroup. The index-$t$ subgroup
has image of order $k/(k,t)$ in $C_k$ and its full inverse image
meets $C_d$ in a group of order
\[
 \epsilon\,\frac{kd/\epsilon}{\operatorname{lcm}(k,t)}
 =r(k,t).
\]
Thus it has full projective image and central intersection exactly
$C_r$ if and only if $(k,t)=1$. The subgroup is then $A[t]$.
For the Ree direct product, perfectness gives the same argument with
$L=\Ree(q)$.

Conversely, all the listed subgroups are defined by the explicit lift
groups.  Their Sylow groups act on the models with one fixed point,
freely elsewhere, and with rational quotient.  Deuring--Shafarevich
therefore gives $p$-rank zero.  The indicated projective images move
the unique Sylow fixed point, so a global fixed point is impossible.
The required order bound is then imposed on each group.
\end{proof}

\begin{remark}[Exact largeness test for the listed subgroups]
\label{rem:subgroup-bound}
For a fixed model of genus $g$, Theorem~\ref{thm:subgroups} becomes
completely numerical. In the linear rows a kernel $C_r$ qualifies
precisely when
\[
 r>\frac{48g(g-1)}{q(q^2-1)}\quad(\PSL),\qquad
 r>\frac{24g(g-1)}{q(q^2-1)}\quad(\PGL),
\]
in addition to the divisibility and existence conditions stated in the
theorem. In the unitary rows the corresponding tests are
\[
 r>\frac{24\mu g(g-1)}{q^3(q^3+1)(q^2-1)}\quad(\PSU),\qquad
 r>\frac{24g(g-1)}{q^3(q^3+1)(q^2-1)}\quad(\PGU),
\]
and in the Ree row
\[
 r>\frac{24g(g-1)}{q^3(q^3+1)(q-1)}.
\]
Since $r\mid d$, these inequalities give an immediate finite test for
all central kernels and all index-$t=d/r$ subgroups on a given model.
\end{remark}

\subsection{Completion of the classification}\label{subsec:parameter-tests}
We give the promised exact calculation for
Proposition~\ref{prop:admissible-d} and the estimates in
Remark~\ref{rem:topcovers}.

\begin{proof}[Proof of Proposition~\ref{prop:admissible-d}]
Write $\Delta_X=|\Aut(X)|-24g(g-1)$.
For $L(q,d)$, substitution of $d=(q+1)/k$ gives
\[
 \Delta_L=\frac{q^2-1}{k^2}
 \bigl((k-6)q^2+13kq-6k^2+6\bigr).
\]
This proves \eqref{eq:linear-k-test}. If $k\ge6$, the polynomial
in parentheses is increasing in $q>0$. Since $d\ge2$ implies
$q\ge2k-1$, its value is at least
$4k(k^2-2k+3)>0$. For $k=2,3,4,5$, respectively, it is
\[
 -2(2q^2-13q+9),\quad -3(q^2-13q+16),\quad
 -2(q^2-26q+45),\quad -q^2+65q-144.
\]
Positivity implies, respectively, $q\le5,11,24,62$.
Requiring $q$ to be an odd prime power, $k\mid q+1$, $d<q$
and $(q-1)(d-1)/2\ge4$ gives exactly the listed possibilities.

For the unitary family put $A=q^2-q+1$. It is odd, and it is
not divisible by $5$, since $X^2-X+1$ has no root modulo $5$.
Thus the only possible divisors $k<7$ are $1$ and $3$.
Direct substitution gives
\begin{align*}
 \Delta_U\big|_{k=1}
 &=-q^2(q^2-1)A
   (5q^4+6q^3-18q^2-13q+12),\\
 \Delta_U\big|_{k=3}
 &=-\frac{(q^2-1)A}{3}
   (q^6+2q^5-14q^4-13q^3+20q^2+16q+16).
\end{align*}
The first parenthesis is positive for $q\ge3$, since it equals
$q^2(5q^2-18)+q(6q^2-13)+12$.
The second is needed only for $q\ge5$; its first four terms are
$q^3\bigl((q-5)(q^2+7q+21)+92\bigr)$, and the remaining
terms are positive. Both values of $k$ therefore fail.
For $q\ge5$ and $k\ge7$, the genus satisfies
$g<d(q-1)(q+1)^2/2$, so
\[
 \frac{|\Aut(U(q,d))|}{24g^2}
 >\frac{kq^3}{6(q-1)(q+1)^2}>1.
\]
The last inequality follows from
$q^3-6q^2+6q+6>0$ for $q\ge5$; after $q=5+u$ this
polynomial is $u^3+9u^2+21u+11$.
For $q=3$, $d=1$ has genus $3$ and $d=7$ has $k=1$, already
excluded. This proves the unitary assertion.

For the Ree family put $z=q_0\ge3$, $q=3z^2$,
$m_-=q-3z+1$ and $n=q^3+1$. The integer $m_-$ is odd and
prime to $3$ and $5$; modulo $5$, the discriminant of
$3z^2-3z+1$ is $-3$, a nonsquare. Thus every divisor $k>1$
of $m_-$ is at least $7$.
At $k=1$,
\[
 \Delta_R=-q(q-1)n
 (5q^3+3q^2z-19q^2+6q+12)<0,
\]
since $5q^3-19q^2=q^2(5q-19)>0$ for $q\ge27$.
For $z\ge9$ and $k\ge7$, using
$g<d(q-1)(q+1)(q+3z+1)/2$ gives
\[
 \frac{|\Aut(R(q,d))|}{24g^2}
 >\frac{kq^3}{6(q^2-1)(q+3z+1)}>1.
\]
Indeed,
\[
 7q^3-6(q^2-1)(q+3z+1)
 =q^2(q-18z-6)+6(q+3z+1)>0
\]
for $q=3z^2$, $z\ge9$.
If $z=3$, then $m_-=19$ is prime and the only $k>1$ is $19$,
giving $d=1$, $g=3627$ and
\[
 |\Ree(27)|=10073444472>
 24\cdot3627\cdot3626=315636048.
\]
This completes the proof.
\end{proof}

\begin{proof}[The estimates in Remark~\ref{rem:topcovers}]
For the top unitary cover, $d=q^2-q+1$, one has
$2g=q^2(q-1)(q^2+q-1)$ and
\[
 |\Aut(X)|-3g^2=\frac{q^3(q-1)}4
 (q^6-3q^5+9q^4+11q^3-9q^2+3q+4)>0.
\]
The polynomial is positive for $q\ge3$: its first two terms are
$q^5(q-3)\ge0$, and $9q^4-9q^2>0$.
For the top Ree cover put $z=q_0\ge3$; then
$2g=q(q-1)(q^2-q-1)$ and
\begin{align*}
 |\Aut(X)|-3g^2
 &=\frac{27}{4}z^4(3z^2-1)P(z),\\
 P(z)&=81z^{10}-324z^9+351z^8-27z^6-15z^4
       -12z^3+7z^2+1.
\end{align*}
Here the first three terms of $P(z)$ equal
$27z^8\bigl(3(z-2)^2+1\bigr)\ge108z^8$, while
$27z^6+15z^4+12z^3\le54z^6$. Thus $P(z)>0$.
The genera of the top unitary and Ree covers are at least $99$
and $246051$, respectively, hence exceed $41$.
\end{proof}

\paragraph{The main theorem.}

\begin{proof}[Proof of Theorem~\ref{thm:geometric}]
Theorem~\ref{thm:main} gives the structural alternatives.
In the prime cyclic case, its Artin--Schreier family is $L(p,d)$,
while Proposition~\ref{prop:A7rigidity} identifies the exceptional
curve with $H_5=U(5,1)$.
Propositions~\ref{prop:linearcomplete}, \ref{prop:unitarymodel}, and
\ref{prop:reemodel} give all the noncyclic models and their exact
Kummer divisibilities.  Proposition~\ref{prop:fullorders} identifies
the full groups, and Theorem~\ref{thm:subgroups} gives the exact
subgroups and realization conditions. To match the range $q\ge5$
in the unitary row, note that $q=3$ allows only $d=1,7$.
The first gives genus $3$; the second gives genus $99$ and full
order $7\cdot6048=42336<24\cdot99\cdot98=232848$.
Neither belongs to the genus-$\ge4$ classification.

For fixed parameters the displayed equations determine a single
$K$-isomorphism class. In the noncyclic cases,
Lemma~\ref{lem:ASrigid} determines the first additive quotient, and
Lemma~\ref{lem:twists}, with the markings of
Lemma~\ref{lem:markings}, excludes all remaining central twists.
The prime cyclic case was identified in Section~\ref{sec:cyclic}.
The parameter $q$ is determined by the
full Sylow group, whose types are elementary abelian in the linear
case, nonabelian of exponent $p$ in the unitary case, and of exponent
$9$ in the Ree case.  Within a type the full Sylow order determines
$q$ and the genus determines $d$.  The three geometric families are
therefore disjoint.  The exceptional $A_7$ is an additional subgroup
on a curve already in the unitary family, not an additional curve.
\end{proof}

\section{Curves of genus two and three}\label{sec:smallgenera}

For $g=2,3$, the bound $24g(g-1)$ does not exceed $84(g-1)$,
so tame actions must be treated separately. The groups $A_L(q,d)$
retain their explicit meaning in \eqref{eq:lineargroup}. We verify
below the low-genus hypotheses needed when reusing the earlier
ramification and reconstruction arguments.

\begin{theorem}\label{thm:small}
Let $p>2$, $\gamma(\cX)=0$, $g\in\{2,3\}$, and let $G\le\Aut(\cX)$
fix no point and satisfy $|G|>24g(g-1)$.  The complete list is:
\begin{enumerate}[label=\textup{(\roman*)},leftmargin=*]
\item $g=2$, $p=5$, $\cX:y^2=x^5-x$;
$G\cong\operatorname{SL}_2(5)$ of order $120$, or
$G=A_L(5,2)$ of order $240$.
\item $g=3$, $p=7$, $\cX:y^2=x^7-x$;
$G\cong\PSL(2,7)$, $C_2\times\PSL(2,7)$, or $A_L(7,2)$,
of orders $168,336,672$.
\item $g=3$, $p=3$, $\cX=H_3$;
$G\cong\PSL(2,7)$ or $G=\PGU(3,3)$,
of orders $168,6048$.
\item $g=3$, $p>3$, $p\ne7$, and $p\equiv3,5,6\pmod7$;
$\cX$ is the Klein quartic
\[
 X^3Y+Y^3Z+Z^3X=0,
\]
and $G=\Aut(\cX)\cong\PSL(2,7)$, of order $168$.
\end{enumerate}
In characteristic $3$ the Klein quartic is isomorphic over $K$ to
$H_3$, so it is not a second curve in \textup{(iii)}.
\end{theorem}

\subsection{Tame lifting and rank criteria}\label{subsec:rankprelim}
Finally, the genus-three Hurwitz curve is the Klein quartic,
with group $\PSL(2,7)$ and signature $(2,3,7)$ in
characteristic prime to $168$. Its characteristic-$3$ model is
isomorphic over $K$ to $H_3$; see \cite{Elkies}.

We use tame lifting only for a Galois cover whose group has order
prime to $p$. Such a cover of smooth projective curves, together
with its distinct marked branch points and group action, lifts to
characteristic zero over a complete discrete valuation ring with
residue field $K$, allowing a finite extension of the ring. One may
choose a smooth lift of the base with disjoint sections through the
branch points; the lifted cover has the same inertia orders.
This is the prime-to-$p$ form of the tame specialization and lifting
theorem; see \cite[Expos\'e~XIII, Corollaire~2.12 and its proof]{SGA1}.
For a source of genus at least two, an isomorphism between the generic
fibers of two stable models over the same discrete valuation ring
extends uniquely to the models \cite[Tag~0E97]{Stacks}.
After a common finite extension, this also applies to an isomorphism
of geometric generic fibers. This is the lifting argument used for
the three-branch-point Klein action in Section~\ref{sec:smallgenera}.

For a smooth plane quartic $F(X,Y,Z)=0$, we use the Hasse--Witt
coefficient formula. With
$v_1=(2,1,1)$, $v_2=(1,2,1)$, $v_3=(1,1,2)$, the $(i,j)$
entry is the coefficient with exponent $pv_j-v_i$ in $F^{p-1}$,
where $F=0$ is the quartic equation; see \cite{SV} for the Cartier
operator formula. A coefficient with a negative exponent is zero.
The $p$-rank is the stable rank of the $p$-semilinear Frobenius on
$H^1(\cX,\mathcal O_{\cX})$, not in general the rank of a single
matrix. For the column-vector convention $v\mapsto \mathsf H v^{(p)}$, it is
the eventual rank of
$\mathsf H\mathsf H^{(p)}\cdots\mathsf H^{(p^{n-1})}$, where superscripts mean entrywise
powers. In particular, an invertible matrix gives $p$-rank $g$,
whereas the zero matrix gives $p$-rank zero. Only these two cases
are needed for the quartic calculation below.

\subsection{Proof of the classification}

\begin{proof}
\emph{Hyperelliptic curves.}
Every genus-$2$ curve, and any hyperelliptic genus-$3$ curve, has a
central hyperelliptic involution.  Its quotient is a line with a
branch set of size $2g+2$.  The image $\bar G$ in $\PGL_2(K)$
preserves this set, and $|G|\le2|\bar G|$.
For a $p$-element of $\bar G$, its unique fixed point on the line
must be in the branch set: otherwise its lift of order $p$ would fix
two points, contradicting zero $p$-rank.
If the reduced group is wild and fixes a point globally, that
point is a branch point by the same argument, and its unique
preimage is fixed by $G$. This is excluded by hypothesis.

The classification of finite subgroups of $\PGL_2(K)$ now applies.
A tame reduced group is cyclic, dihedral, $A_4$, $S_4$, or $A_5$.
Preserving six points bounds the full order by $48$ (the tame $A_5$
has no orbit of size at most six). For eight points the largest tame
reduced possibility is $S_4$, through its orbit of size $8$, so again
$|G|\le48$. Both contradict the respective strict thresholds.  A wild reduced group without a common fixed point is
$\PSL(2,q)$ or $\PGL(2,q)$ in its natural action, or the
cross-characteristic $A_5$ in characteristic $3$.  The latter has ten
Sylow fixed points and cannot preserve the branch set.  In the linear
cases the branch set contains $q+1$ Sylow fixed points.  For $q=3$
all other nontrivial orbits have at least six points; for $q=5$ they
have at least twenty.  Thus six branch points force $q=5$, and eight
force $q=7$.  The branch set is exactly $\PP^1(\F_q)$.
Thus the curve is $y^2=x^q-x$, with $q=5$ or $7$.
Let $Z_h=C_2$ be its hyperelliptic kernel and put
$L=A_L(q,2)'$. The reduced image of $G$ contains $\PSL(2,q)$,
so $GZ_h$ is the special or full projective preimage, whose
commutator group is $L$. Centrality of $Z_h$ gives
$(GZ_h)'=G'=L$, hence $L\le G$.
For $(q,d)=(5,2)$ the group $L$ is $\operatorname{SL}_2(5)$
and $A_L/L\cong C_2$; for $(7,2)$ it is $\PSL(2,7)$ and
$A_L/L\cong C_4$. The subgroups of these cyclic quotients give
exactly the stated lists. In the second case $L\cap Z_h=1$,
so the order-$336$ group is $C_2\times\PSL(2,7)$.

\emph{Tame nonhyperelliptic genus three.}
Here $G$ is tame and $|G|>144$.
Tame Riemann--Hurwitz leaves only signature $(2,3,7)$ and $|G|=168$:
for a rational quotient the next smallest positive normalized
signature is $(2,3,8)$, giving order $96$; four or more branch points
also give a smaller bound.  A positive-genus quotient is impossible
at this size.  The genus-three Hurwitz action is the Klein quartic
with group $\PSL(2,7)$; see \cite{Elkies}.  This tame characterization
is valid in characteristic prime to $168$: the tame cover \cite{SGA1} with its
three marked branch points lifts to characteristic zero, and the
unique genus-three Hurwitz curve specializes back uniquely by the
uniqueness of the smooth stable model.  Thus here $p\ne3,7$.
The zero-$p$-rank condition is checked below.

\emph{Wild nonhyperelliptic genus three.}
We first recover the two-orbit conclusion without using the Hurwitz
threshold.  If the quotient had positive genus, the unique wild
orbit, of length at least $s+1$, would contribute at least
$(s+1)(2s-2)>4=2g-2$, a contradiction.
With a rational quotient, three or more tame short orbits give
normalized different at least $1/2$.  Two tame orbits, one with
stabilizer at least $3$, give a value greater than $1/6$.
If both stabilizers have order $2$, then
$4=n(\mathcal D_P-1)$; hence $n=4$, $s=3$ and $\mathcal D_P=2$.
Lemma~\ref{lem:local} gives $m\le\mathcal D_P=2$, so $|G|=nsm\le24$.
A single wild orbit gives $|G|<8(g-1)^2=32$ by the one-orbit
calculation in the proof of Theorem~\ref{thm:two}, which uses
only the one-point Sylow action and positive genus.
Therefore the action has precisely one wild and one tame short orbit.

Suppose $S$ is noncyclic. Theorem~\ref{thm:GMPprelim} applies,
and Lemma~\ref{lem:quadratic-data} gives $n=s+1$ outside the
small Ree case; the direct normalizer calculation of
Section~\ref{sec:noncyclic} gives it in that case as well.
Here $s\ge p^2\ge9$. With $h=g(\cX/S)$, one-point Hurwitz and
$\mathcal D_P\ge s-1$ give $3=g\ge sh$, so $h=0$.
The cyclic two-point complement and this rationality are all that
the proofs of Lemma~\ref{lem:intersection},
Proposition~\ref{prop:master} and Theorem~\ref{thm:kernel}
use. Their semiregularity, central-kernel and genus formulas
therefore hold here without importing a genus-$\ge4$ assumption.  The linear formula cannot give
$g=3$ for $q\ge9$. The non-Singer unitary branch has
$g\ge(q-1)(q^3+1)/(2\delta)>3$, so it is excluded without the
large-genus inequality used earlier. In the Singer unitary formula the minimum is
$q(q-1)/2$, so only $q=3$, $r=t=1$ occurs; it gives $\PGU(3,3)$
and $H_3$ by the same rigidity argument. Its required inequalities
are checked directly: $t=1<3$, $tq=3<8=b$, and
$\delta t=1<2=q-1$. Thus the proof of
Propositions~\ref{prop:unitaryjumps}--\ref{prop:unitarymodel}
applies without appealing to the large-genus proof of those bounds.
For a Ree row, since $e\mid n$, the genus formula would imply
$q-1\mid n+4=q^3+5$, hence $q-1\mid6$, impossible for $q\ge27$.
The small Ree quotient is excluded by its negative descended genus
as before.

If $S$ is cyclic, the exclusion of higher cyclic groups is unchanged.
For $S=C_p$, the genus formula
$3=ph+(p-1)(j-1)/2$ gives, when $h>0$, only
$p=3,h=1,j=1$.  Here $m\mid2$, and
\[
 |G|=\frac{12me}{e-3m}.
\]
For $m=1$ this is at most $48$.  For $m=2$, the strict inequality
$|G|>144$ forces $e=7$ and $|G|=168$.
The same abelian-quotient argument as in Lemma~\ref{lem:A7} makes
$G$ perfect. Since $j=1$ and $m=2$, the tame tangent character acts
faithfully on $S=C_3$, that is, by inversion; hence $S=[S,G_P]\le G'$,
and $G/G'$ has order prime to $3$. A tame abelian quotient would then
have degree dividing both $2$ and $7$.  A nonabelian simple quotient of this perfect group has order
dividing $168$. The only possibility is $\PSL(2,7)$, already of
order $168$, so $G\cong\PSL(2,7)$.
For $h=0$, the Stichtenoth--Lehr--Matignon alternatives give either
the already listed hyperelliptic $(p,d)=(7,2)$ or the Hermitian
$(p,d)=(3,4)$.  In the latter case, under the supposition that the
Sylow group of $G$ has order $3$, the same bound
\eqref{eq:Rbound}, with $p=3$, $J=8$, $x=4$ and $m\mid8$, gives
$|G|\le12m\le96$, a contradiction.

We identify the characteristic-$3$ group of order $168$ without
assuming its curve is Hermitian.  Let $U=7{:}3\le\PSL(2,7)$ have
index $8$.  On $G/U$, wild inertia $S_3$ has orbits $2,6$, and tame
inertia $C_7$ has orbits $1,7$.
For the $U$-cover, the wild contribution is $7(2+2)=28$ and the tame
one is $3(7-1)=18$.  Hence
\[
 4=21(2g(\cX/U)-2)+46,
\]
so $\cX/U$ is rational.  Normalize the degree-$8$ intermediate map
to have zeros $1,a$ of multiplicities $6,2$ and poles $0,\infty$ of
multiplicities $1,7$.  It is
\[
 f(z)=\frac{(z-1)^6(z-a)^2}{z}.
\]
In characteristic $3$ its derivative is
$(z-1)^6(z-a)(z+a)/z^2$.
There must be no further branch point, so $a=-1$.
Its core-free Galois closure is unique.  The reduction of the Klein
quartic action \cite[pp.~79--80]{Elkies} realizes this cover, so the curve is the Klein
quartic in characteristic $3$.  Its nonsingular Frobenius-form
equation is equivalent over $K$ to the Hermitian form of degree $4$
\cite[pp.~79--80]{Elkies};
thus it is $H_3$.

\emph{The $p$-rank of the Klein quartic.}
For $p\ne2,7$, use the plane-quartic Hasse--Witt coefficient formula
with interior vectors
$v_1=(2,1,1)$, $v_2=(1,2,1)$, $v_3=(1,1,2)$.
The $(i,j)$ entry is the coefficient of the monomial with exponent
$p v_j-v_i$ in $(X^3Y+Y^3Z+Z^3X)^{p-1}$.
A coefficient is nonzero precisely when
\[
 \begin{pmatrix}3&0&1\\1&3&0\\0&1&3\end{pmatrix}^{-1}
 (pv_j-v_i)
\]
is a nonnegative integer vector.  Its entries sum to $p-1$, so any
such multinomial coefficient is nonzero modulo $p$.
Checking the six residues modulo $7$ gives, up to nonzero scalar
entries denoted by $*$,
\[
 \mathsf H_p\sim
 \begin{cases}
 \begin{pmatrix}*&0&0\\0&*&0\\0&0&*\end{pmatrix},&p\equiv1\pmod7,\\[2ex]
 \begin{pmatrix}0&*&0\\0&0&*\\*&0&0\end{pmatrix},&p\equiv2\pmod7,\\[2ex]
 \begin{pmatrix}0&0&*\\*&0&0\\0&*&0\end{pmatrix},&p\equiv4\pmod7,\\[2ex]
 0,&p\equiv3,5,6\pmod7.
 \end{cases}
\]
Thus the first three cases have $p$-rank $3$ (they are ordinary),
whereas the last three have $p$-rank zero. This proves the stated
residue condition, including the already separated characteristic $3$.

Finally, the preceding wild classification excludes additional wild
automorphisms of the Klein quartic in the remaining characteristics;
the tame bound and its order-$168$ subgroup exclude a larger tame
full group.  All groups in the list act without a common fixed point
and satisfy the numerical threshold, completing the proof.
\end{proof}

\section*{Statements and Declarations}
OpenAI's ChatGPT was used for language editing, literature searches,
assistance with drafting and revising parts of the exposition and
proofs, and checks of selected calculations and arguments. The author
takes full responsibility for the mathematical content, references,
and final formulations.

\section*{Acknowledgments}
The author is grateful to Professor Gábor Korchmáros for reading an earlier version of the manuscript and for his valuable comments and suggestions, especially concerning Lemmas 6.1 and 6.2. His suggestions led to a more detailed presentation of these proofs and helped improve the clarity of the manuscript. The author also thanks Robert Guralnick, Beth Malmskog and Rachel Pries
for their helpful correspondence concerning Theorem~3.16 of their
paper~\cite{GMP}.

\textbf{Funding.}
The author was partially supported by CNPq grant no.~302774/2025-4,
FAPESP grant no.~2024/00923-6, and FAEPEX grant no.~3485/25.

\end{document}